\documentclass[letterpaper]{article}
\ifdefined\pdfminorversion
\fi
\usepackage{arxiv}
\usepackage[utf8]{inputenc}
\usepackage[T1]{fontenc}
\usepackage[hyphens]{url}
\usepackage{graphicx}
\usepackage[authoryear,round]{natbib}
\setcitestyle{semicolon,aysep={},yysep={;}}
\usepackage{caption}
\usepackage{hyperref}
\usepackage{microtype}
\usepackage{amsmath}
\usepackage{amssymb}
\usepackage{amsfonts}
\usepackage{mathtools}
\usepackage{amsthm}
\usepackage{booktabs}
\usepackage{nicefrac}
\usepackage{multirow}
\usepackage{algorithm}
\usepackage{algpseudocode}
\usepackage{xr}
\usepackage{placeins}

\hypersetup{
  pdftitle={Sliding Methods for Holder-Smooth Convex--Concave Minimax Optimization with Bilinear Coupling},
  pdfauthor={Nhat Trung Nguyen},
  hidelinks
}

\newtheorem{theorem}{Theorem}[section]
\newtheorem{corollary}[theorem]{Corollary}
\newtheorem{lemma}[theorem]{Lemma}
\newtheorem{assumption}[theorem]{Assumption}
\theoremstyle{definition}
\newtheorem{definition}{Definition}
\newtheorem{remark}{Remark}

\DeclareMathOperator*{\argmin}{arg\,min}
\DeclareMathOperator*{\argmax}{arg\,max}

\DeclareMathOperator{\col}{col}
\DeclareMathOperator{\gap}{Gap}
\DeclareMathOperator{\diag}{diag}
\DeclareMathOperator{\range}{range}

\newcommand{\R}{\mathbb{R}}

\newcommand{\bA}{{\mathbf{A}}}
\newcommand{\bB}{{\mathbf{B}}}

\newcommand{\bC}{{\mathbf{C}}}

\newcommand{\bD}{{\mathbf{D}}}

\newcommand{\bI}{{\mathbf{I}}}

\newcommand{\bL}{{\mathbf{L}}}
\newcommand{\bM}{{\mathbf{M}}}

\newcommand{\bO}{{\mathbf{O}}}
\newcommand{\bP}{{\mathbf{P}}}

\newcommand{\bS}{{\mathbf{S}}}

\newcommand{\bU}{{\mathbf{U}}}
\newcommand{\bV}{{\mathbf{V}}}
\newcommand{\bW}{{\mathbf{W}}}

\newcommand{\bX}{{\mathbf{X}}}

\newcommand{\cA}{{\mathcal{A}}}

\newcommand{\cC}{{\mathcal{C}}}
\newcommand{\cD}{{\mathcal{D}}}

\newcommand{\cF}{{\mathcal{F}}}

\newcommand{\cK}{{\mathcal{K}}}

\newcommand{\cN}{{\mathcal{N}}}
\newcommand{\cO}{{\mathcal{O}}}

\newcommand{\cS}{{\mathcal{S}}}

\newcommand{\cU}{{\mathcal{U}}}

\newcommand{\cX}{{\mathcal{X}}}
\newcommand{\cY}{{\mathcal{Y}}}
\newcommand{\cZ}{{\mathcal{Z}}}

\newcommand{\bt}{{\mathbf{t}}}

\newcommand{\bx}{{\mathbf{x}}}

\newcommand{\angles}[1]{\left\langle#1\right\rangle}
\newcommand{\norm}[1]{\left\| #1 \right\|}
\newcommand{\normsq}[1]{\norm{#1}^2}
\newcommand{\lmax}{\lambda_{\max}}
\newcommand{\lmin}{\lambda_{\min}}
\newcommand{\lminp}{\lambda_{\min}^+}

\DeclareMathOperator{\D}{D}

\newcommand{\bbS}{{\mathbb{S}}}
\newcommand{\xout}{{x_{\mathrm{out}}}}
\newcommand{\yout}{{y_{\mathrm{out}}}}
\newcommand{\zout}{{z_{\mathrm{out}}}}
\newcommand{\xin}{{x_{\mathrm{in}}}}
\newcommand{\yin}{{y_{\mathrm{in}}}}
\newcommand{\zin}{{z_{\mathrm{in}}}}
\newcommand{\pmat}[1]{\begin{pmatrix}#1\end{pmatrix}}

\title{Sliding Methods for H\"older-Smooth Convex--Concave Minimax Optimization with Bilinear Coupling}
\author{
  Nhat Trung Nguyen \\
  MIRIAI\\
  \texttt{nguyen.n@miriai.org} \\
  \And
 Alexander Gasnikov \\
  Innopolis University, MIRIAI, IITP \\
  \texttt{gasnikov@yandex.ru} \\
}

\date{}
\renewcommand{\shorttitle}{Sliding Methods for H\"older-Smooth Minimax Optimization}

\begin{document}

\raggedbottom
\maketitle

\begin{abstract}
We study convex-concave minimax optimization problems with bilinear coupling of the form
$\min_{x\in \mathcal X}\max_{y\in \mathcal Y}
    \; f(x)+\langle y,\mathbf{B}x\rangle-g(y),$
where the functions $f$ and $g$ have H\"older continuous (sub)gradients.
This setting covers a broad range of regimes, from nonsmooth problems with
bounded subgradient variation to smooth problems with Lipschitz continuous
gradients; for a smooth component used in the coupling-induced regularizer,
its Lipschitz-gradient constant is assumed to hold in the ambient space. We
propose a sliding method that exploits the composite structure of the problem
by querying the oracles associated with $f$, $g$, and the bilinear coupling
operator at prescribed frequencies determined by their individual properties.
The method is based on a recursive sliding scheme for monotone variational
inequalities. We establish convergence guarantees under H\"older continuity
and show how the resulting complexity bounds depend explicitly on the H\"older
exponents, H\"older constants, strong convexity parameters, and spectral
properties of the coupling matrix. Our analysis covers nonstrongly convex and
partially strongly convex regimes. For stochastic function oracles, we prove
componentwise H\"older guarantees with separate variance costs: a uniform
expected-gap bound on bounded domains and, when both functions are strongly
convex on their feasible sets, a restarted Lyapunov guarantee without an
anchor noise floor.
Numerical experiments reproduce the predicted H\"older exponents and confirm
that the number of gradient evaluations required for each function separates
according to its own smoothness level rather than the worse of the two. A
tomographic benchmark shows runtime gains when gradient evaluations are more
expensive than the additional matrix-vector products.
\end{abstract}

\section{Introduction} \label{sec:intro}
Convex-concave minimax optimization plays a central role in many areas of
optimization, machine learning, game theory, and signal processing. In this
work, we focus on the important class of saddle-point problems in which the
interaction between the primal and dual variables is bilinear. Such problems
arise naturally in constrained optimization, decentralized optimization,
regularized optimal transport, inverse problems, and matrix games. We discuss several applications of this problem class in Appendix~\ref{sec:applications}. Specifically, we study minimax problems of the form
\begin{equation} \label{prob:bilinear_SPP_main}
  \min_{x\in \mathcal{X}} \max_{y \in \mathcal{Y}}
  \; F(x, y) = f(x) + \langle y, \bB x \rangle - g(y),
\end{equation}
where \(\mathcal{X} \subseteq \R^{d_x}\) and
\(\mathcal{Y} \subseteq \R^{d_y}\) are nonempty closed convex sets, and
\(\bB \in \R^{d_y \times d_x}\) is the coupling matrix.

We consider the setting where \(f \colon \mathcal{X} \to \R\) and
\(g \colon \mathcal{Y} \to \R\) are convex functions whose
(sub)gradients satisfy H\"older continuity conditions. This
assumption provides a unified framework that interpolates between nonsmooth
problems with bounded subgradient variation and smooth problems with
Lipschitz continuous gradients. The formal assumptions on \(f\), \(g\), and
the spectral structure of \(\bB\) are stated in
Section~\ref{sec:bilinear_SPP}.

In the smooth case, Problem~\eqref{prob:bilinear_SPP_main} was studied by~\citet{borodich2025linear}, who obtained separated accelerated
complexities for evaluating \(\nabla f\), evaluating \(\nabla g\), and applying
\(\bB\) and \(\bB^\top\). This paper investigates how these guarantees extend
to the H\"older-continuous setting.

\subsection{Main Contributions}
The main contributions of this paper are summarized as follows.

\paragraph{Componentwise sliding for variational inequalities.}
We develop a recursive sliding algorithm for finite-sum monotone variational
inequalities with H\"older-continuous function components. The method extends
the sliding framework of~\citet{borodich2025linear} beyond the
smooth setting by querying different components at different frequencies, which
yields componentwise oracle complexities. In the composite optimization case,
this removes the cross-component terms present in existing gradient sliding
bounds~\citep{lan2016gradient,lan2022accelerated}.

\paragraph{H\"older-smooth bilinear minimax optimization.}
We apply the recursive sliding framework to the H\"older-smooth
convex-concave minimax problem~\eqref{prob:bilinear_SPP_main}. By formulating
the problem as a structured variational inequality with components
corresponding to \(f\), \(g\), and the bilinear coupling, the method separates
the oracle calls associated with the different parts of the problem. We prove
convergence guarantees with component-wise complexity bounds for
(sub)gradient evaluations of \(f\),
(sub)gradient evaluations of \(g\), and matrix-vector products
involving \(\bB\) and \(\bB^\top\). Each count depends on the H\"older exponent
and constant of the corresponding function, on the strong convexity parameters,
and on the spectrum of \(\bB\).
We also extend the deterministic results to nonstrongly convex and partially
strongly convex settings.

\paragraph{Stochastic H\"older guarantees.}
For conditionally unbiased, bounded-variance function oracles and exact
bilinear products, we retain the componentwise H\"older complexities and add
separate variance costs. We prove a uniform expected primal--dual gap bound
on bounded domains. When \(\mu_x,\mu_y>0\), a decomposition without the
anchor regularizer yields convergence in the paper's Lyapunov metric with
variance costs \(\sigma_x^2/(\mu_x\varepsilon)\) and
\(\sigma_y^2/(\mu_y\varepsilon)\), without an anchor noise floor.
The H\"older assumptions are imposed on the feasible sets.

\paragraph{Numerical experiments.}
Synthetic tests recover the predicted componentwise exponents and separated
gradient counts. A three-regime Huber--TV tomography benchmark against four
primal--dual baselines shows runtime gains when gradient calls dominate the
additional matrix-vector cost, including a regime with a genuinely H\"older
dual gradient.

\subsection{Related Work}\label{sec:related_work}

\textbf{Convex-concave saddle-point problems.}
The study of convex-concave minimax problems has a long history in optimization. Classical first-order methods include the extragradient method of~\citet{korpelevich1976extragradient}, the mirror-prox algorithm of~\citet{nemirovski2004prox} achieving an $O(1/t)$ convergence rate for Lipschitz monotone operators, and the primal-dual algorithm of~\citet{chambolle2011first}. These general-purpose methods do not exploit any particular structure in the coupling between the primal and dual variables, and hence yield suboptimal rates for problems with bilinear coupling.

\textbf{Bilinear saddle-point problems.}
For bilinear coupling \(\langle y,\bB x\rangle\), spectral structure yields
sharper rates. Prior work gives lower bounds~\citep{zhang2022lower},
bilinear-induced linear convergence~\citep{du2019linear,ibrahim2020linear},
and optimal smooth algorithms~\citep{borodich2025linear}. The latter, closest
to this paper, uses condition numbers combining the smoothness of \(f,g\) with
the spectrum of \(\bB\). \citet{kovalev2022accelerated} give an optimal
Accelerated Primal-Dual Gradient method for the strongly convex--strongly
concave regime, but charge \(\nabla f\), \(\nabla g\), and \(\bB\) jointly
rather than separately.

Universal primal--dual methods already treat H\"older-smooth
affine-constrained optimization with bilinear Lagrangians:
\citet{yurtsever2015universal} adapt to unknown dual H\"older regularity, and
\citet{luo2024universal} give an accelerated universal method for Lipschitz and
H\"older gradients. Their guarantees count joint iterations. In this work, the
two function oracles and the coupling oracle are charged separately.

\textbf{Complexity separation.}
Per-component budgets are established in the smooth regime:
\citet{tominin2021accelerated} and \citet{lan2021mirrorprox} obtain separated
rates for strongly-convex--strongly-concave and two-component composite
problems, respectively. Closest to us, \citet{borodich2025linear} give an
\(n\)-level recursive sliding method
for composite monotone variational inequalities with separated budgets, plus
matching lower bounds. This paper adds the layer above: exact gradients are
replaced by inexact oracles, which lets each function carry its own level of
smoothness, from Lipschitz gradients to bounded subgradient variation, within a
single analysis. For composite minimization, the parameter-free universal
gradient-sliding method of~\citet{wu2026parameterfree} also separates the
(sub)gradient counts of a H\"older component and a smooth component without
knowing their constants. It does not contain a bilinear coupling oracle.
Appendix~\ref{app:separation} gives the detailed comparison.

\textbf{Variational inequality framework.}
The reduction of saddle-point problems to monotone variational inequalities (VIs) is a classical technique~\citep{facchinei2003finite}. \citet{nesterov2007dual} developed the dual extrapolation method for solving VIs, and~\citet{juditsky2016solving} studied VIs on domains given by linear minimization oracles. In this work, we formulate the bilinear saddle-point problem as a structured VI with a finite-sum decomposition into $n = 3$ components, which enables the application of multi-level sliding techniques.

\textbf{Sliding and multi-level methods.}
The gradient sliding technique, introduced by~\citet{lan2016gradient} for composite optimization, allows different components of the objective to be updated at different frequencies, thereby reducing the total number of expensive gradient evaluations. This idea was further developed by~\citet{lan2018optimal} for incremental gradient methods. Sliding methods have also been applied in distributed optimization~\citep{dvinskikh2021decentralized, kovalev2020optimal}. Closest to Section~\ref{sec:general_VI} is the Mirror-Prox Sliding method of~\citet{lan2021mirrorprox}, which treats a smooth convex function together with a monotone Lipschitz operator and charges each of the two its own number of evaluations, with a stochastic variant in which the operator is sampled; our two-component case with smooth components reproduces those rates.

\textbf{H\"older continuous gradients and inexact oracles.}
The framework of H\"older continuous gradients, parameterized by $\nu \in [0,1]$, provides a unified treatment that interpolates between non-smooth ($\nu = 0$) and smooth ($\nu = 1$) optimization. \citet{nesterov2015universal} introduced universal gradient methods that automatically adapt to the unknown smoothness level, and~\citet{nesterov2005smooth} developed smoothing techniques for non-smooth problems. For monotone VIs, the generalized Mirror-Prox method of~\citet{stonyakin2022generalized} is universal for H\"older-continuous operators. Separately, the inexact oracle model of~\citet{devolder2014first} provides a principled way to analyze first-order methods when gradient and function value computations are approximate.

\textbf{Stochastic H\"older optimization.}
\citet{rodomanov2024universal} develop universal accelerated methods for
stochastic convex optimization with the rate
\(O(HR^{1+\nu}/N^{(1+3\nu)/2}+\sigma R/\sqrt N)\).
Our stochastic extension has this dependence for each function component,
while separately charging the exact bilinear products. It uses prescribed
H\"older and variance bounds to choose the recursive budgets. The analysis
combines the H\"older model error with a martingale bound that remains valid
for the uniform primal--dual gap.

\paragraph{Paper organization.}
 Section~\ref{sec:definitions} collects definitions and notation. Section~\ref{sec:general_VI} develops a recursive sliding algorithm for monotone variational inequalities with H\"older function components and Lipschitz monotone operators. Section~\ref{sec:bilinear_SPP} specializes it to the bilinear saddle-point problem~\eqref{prob:bilinear_SPP_main} and states the deterministic and stochastic guarantees. Section~\ref{sec:experiments} presents numerical experiments. All proofs are deferred to the appendix.

\section{Definitions and Notations}\label{sec:definitions}

We denote by \(\|\cdot\|\) the standard Euclidean norm on \(\R^d\) and by \(\angles{\cdot, \cdot}\) the standard inner product. Throughout this section, \(\cU\subseteq\R^d\) denotes a nonempty closed convex set, typically a feasible set. For a positive definite matrix \(\bP \in \bbS_{++}^d\), the weighted norm is \(\|z\|_{\bP} = \angles{\bP z, z}^{1/2}\). For a symmetric matrix \(\bM\), we write \(\lmax(\bM)\), \(\lmin(\bM)\), and \(\lminp(\bM)\) for its largest eigenvalue, smallest eigenvalue, and smallest positive eigenvalue, respectively. We write \(\range(\bM)\) for the column space of \(\bM\).

\begin{definition}\label{def:subgradient}
  Let $h \colon \R^d \to \R$ be convex. A vector $\xi \in \R^d$ is called a \emph{subgradient} of $h$ at the point $u \in \R^d$ if for any $v \in \R^d$ we have
  \begin{equation} \label{eq:def_subgradient}
    h(v) \geq h(u) + \langle \xi, v - u \rangle.
  \end{equation}
  The set of all subgradients of $h$ at $u$, denoted by $\partial h(u)$, is called the \emph{subdifferential} of $h$ at $u$.

\end{definition}


\noindent\textbf{Notation and convention.}
Throughout the paper, when \(h\) is subdifferentiable at \(u\), we write
\(h'(u)\in\partial h(u)\) to denote an arbitrary subgradient of \(h\) at
\(u\). Any condition involving \(h'(u)\) must hold for every choice from
\(\partial h(u)\). When such a condition is required for every \(u\in\cU\),
we also assume that a subgradient exists at every point of \(\cU\), that is,
\[
  \partial h(u)\neq\varnothing,
  \qquad \forall u\in\cU.
\]
If \(h\) is differentiable at \(u\), then
\(\partial h(u)=\{\nabla h(u)\}\), and hence \(h'(u)=\nabla h(u)\).

\begin{definition}\label{def:strongly_convex}
  A function \(h \colon \mathcal{U} \to \R \) is called \emph{\( \mu \)-strongly convex} on $\mathcal{U} \subseteq \R^d$, for some $\mu \geq 0$, if for any $u, v \in \mathcal{U}$ and $h^\prime(u) \in \partial h(u)$ we have
  \begin{equation} \label{eq:def_strongly_convex}
    h(v) \geq h(u) + \angles{h^\prime(u), v - u} + \frac\mu2\normsq{v - u}.
  \end{equation}
  When \(\mu = 0\), the function \( h \) is said to be (non-strongly) convex.
\end{definition}

\begin{definition}\label{def:inexact_oracle}
  We say that \(h \colon \cU \to \R\) is equipped with a first-order \((\delta, L)\)-oracle with respect to the norm \(\|\cdot\|_{\bP}\) if for any \(v \in \cU\) we can compute a pair \((\tilde{h}(v), \tilde{\nabla}h(v))\) such that for all \(u \in \cU\),
  \begin{equation}
    0 \leq h(u) - \tilde{h}(v) - \langle \tilde{\nabla}h(v), u - v \rangle \leq \frac{L}{2} \|u - v\|_{\bP}^2 + \delta.
  \end{equation}
\end{definition}

\begin{definition}\label{def:Holder_condition}
  Let \(0 \leq \nu \leq 1\) and \(H > 0\). We say that a function \(h \colon \cU \to \R\) satisfies the \((\nu, H)\)-H\"older condition on \(\cU \subseteq \R^d\) with respect to \(\|\cdot\|_{\bP}\) if for any \(u, v \in \cU\) and $h^\prime(u) \in \partial h(u)$, $h^\prime(v) \in \partial h(v)$,
  \begin{equation} \label{eq:def_Holder_condition}
    \|h'(u) - h'(v)\|_{\bP^{-1}} \leq H \|u - v\|_{\bP}^{\nu}.
  \end{equation}
\end{definition}

For \(\nu > 0\) the condition implies H\"older continuity of the gradient, the limit case \(\nu = 1\) corresponds to Lipschitz-continuous gradients, while \(\nu = 0\) requires only that subgradients have bounded variation.

\begin{definition}\label{def:strongly_monotone_operator}
  An operator $G \colon \mathcal{U} \to \R^{d}$ is called \emph{$\mu$-strongly monotone} with respect to the norm $\| \cdot \|_{\bP}$ for $\mu \geq 0$ if the following inequality holds for all $u, v \in \mathcal{U}$:
  \begin{equation} \label{eq:def_strongly_monotone_operator}
    \langle G(u) - G(v), u - v \rangle \geq \mu \| u - v\|_\bP^2.
  \end{equation}
  An operator $G \colon \mathcal{U} \to \R^{d_z}$ is called \emph{ monotone} if the same inequality holds with $\mu = 0$.
\end{definition}

\begin{definition}\label{def:Lipschitz_operator}
  An operator $G \colon \mathcal{U} \to \R^{d}$ is called \emph{$M$-Lipschitz} with respect to the norm $\| \cdot \|_{\bP}$ for $M_i > 0$ if the following inequality holds for all $u, v \in \mathcal{U}$:
  \begin{equation} \label{eq:def_strongly_monotone}
    \|G(u)-G(v)\|_{\bP^{-1}}
  \leq
  M\|u-v\|_{\bP}.
  \end{equation}
\end{definition}

\section{General Variational Inequalities}\label{sec:general_VI}

In this section, we consider a general monotone variational inequality problem
of finding \(z^* \in \mathcal{Z}\) such that
\begin{equation} \label{prob:general_VI}
  p(z^*) - p(z) + \langle Q(z), z^* - z \rangle \leq 0,
  \qquad \forall z \in \mathcal{Z},
\end{equation}
where \(\mathcal{Z} \subseteq \R^{d_z}\) is nonempty, closed, and convex,
\(p \colon \mathcal{Z} \to \R\), and
\(Q \colon \mathcal{Z} \to \R^{d_z}\). We assume that both \(p\) and \(Q\)
have the finite-sum structure:
\begin{equation} \label{eq:finite_sum_structure}
  p(z) = \sum_{i=1}^n p_i(z),
  \qquad
  Q(z) = \sum_{i=1}^n Q_i(z).
\end{equation}
For each \(1\leq i\leq n\), the function
\(p_i:\mathcal Z\to\mathbb R\) is convex, and the operator
\(Q_i:\mathcal Z\to\mathbb R^{d_z}\) is continuous and monotone.

A point \(z^* \in \mathcal{Z}\) satisfying~\eqref{prob:general_VI} is called
a \emph{weak} solution of the monotone variational inequality. Since \(p\) is convex and \(Q\) is continuous and monotone, the weak and strong formulations are equivalent. In particular,
every weak solution also satisfies
\begin{equation}
  p(z^*) - p(z) + \langle Q(z^*), z^* - z \rangle \leq 0,
  \qquad \forall z \in \mathcal{Z}.
\end{equation}

\subsection{Algorithm for Inexact Oracle}\label{sec:VI_inexact}
To exploit the finite-sum structure in~\eqref{eq:finite_sum_structure}, we
allow each component to have its own operator Lipschitz constant and inexact
oracle parameters. Throughout this section, Lipschitz continuity of operators
and first-order inexact oracles are understood in the sense of
Definitions~\ref{def:Lipschitz_operator}
and~\ref{def:inexact_oracle}.

\begin{assumption}\label{ass:Monotone_and_Lipschitz_Operators_VI}
For every \(1\leq i\leq n\), the operator \(Q_i\) is monotone and
\(M_i\)-Lipschitz on \(\mathcal Z\) with respect to the norm
\(\|\cdot\|_{\bP}\).
\end{assumption}

\begin{assumption}\label{ass:inexact_oracle_functions_VI}
For every \(1\leq i\leq n\), the function \(p_i\) admits a first-order
\((\delta_i,L_i)\)-oracle on \(\mathcal Z\) with respect to the norm
\(\|\cdot\|_{\bP}\).
\end{assumption}

We first describe two special cases to illustrate the construction. When
\(n=1\), there is no inner loop, and the method reduces to
Algorithm~\ref{alg:one_level}. This algorithm recovers several classical methods. If
\(Q_1\equiv 0\), it reduces to Nesterov's accelerated gradient method
\citep{nesterov2013introductory}. If \(p_1\equiv 0\), it reduces to the extragradient method of~\citet{korpelevich1976extragradient}. When
both \(p_1\) and \(Q_1\) are present, it coincides with the Accelerated
Mirror-Prox method of~\citet{chen2017accelerated}, specialized to the
\(\|\cdot\|_{\bP}\)-norm.

For \(n=2\), the outer and inner loops handle the first and second components,
respectively, allowing query frequencies tailored to \(L_i\) and \(M_i\).
The \(n\)-level method applies this construction recursively. Appendix~
\ref{app:explicit_sliding_algorithms} gives both algorithms, and Appendix~
\ref{app:algo_implementability} discusses implementability.

The following theorem gives the main guarantee for the recursive sliding
scheme. It bounds the weak variational inequality gap at the output
\(z_{\mathrm{out}}\) uniformly over all comparison points \(z\in\mathcal Z\).
The proof is given in Appendix~\ref{app:VI_inexact_proof}. The theorem is used in
Section~\ref{sec:VI_Holder} to handle H\"older-continuous components.

\begin{theorem}\label{thm:VI_sliding_inexact}
Let Assumptions~\ref{ass:Monotone_and_Lipschitz_Operators_VI}
and~\ref{ass:inexact_oracle_functions_VI} hold, where
\(M_i,L_i\geq0\) and \(M_i+L_i>0\) for all \(i\). Let
\(T_1,\dots,T_n\) be positive integers. Suppose that
Algorithm~\ref{alg:sliding_recursive} uses the sequence
\begin{equation}\label{eq:alpha_sequence}
  \alpha_0=1,
  \qquad
  \alpha_{t+1}
  =
  \frac{2}{1+\sqrt{1+4/\alpha_t^2}},
  \qquad t\geq0,
\end{equation}
and the parameters
\begin{equation}\label{eq:eta_inexact}
  \eta_{\bt}^{k}
  =
  L_k\prod_{\ell=1}^k \alpha_{t_\ell}
  +
  M_k\prod_{\ell=1}^k
  \frac{\alpha_{t_\ell}}{\alpha_{T_\ell-1}},
  \qquad
  \bt=(t_1,\dots,t_k),
\end{equation}
for every \(1\leq k\leq n\), where
\(0\leq t_\ell\leq T_\ell-1\) for \(\ell=1,\dots,k\). Then, for every
\(z\in\mathcal Z\), the output \(z_{\mathrm{out}}\) of
Algorithm~\ref{alg:sliding_recursive} satisfies
\begin{equation}\label{eq:gap_bound_inexact}
  p(z_{\mathrm{out}})-p(z)
  +\langle Q(z),z_{\mathrm{out}}-z\rangle
  \leq
  \sum_{i=1}^n
  \Bigg[
    \frac{2^{2i-1}L_i}{\prod_{j=1}^i T_j^2}
    \|z_{\mathrm{in}}-z\|_{\bP}^2
    +
    \frac{2^{i-1}M_i}{\prod_{j=1}^i T_j}
    \|z_{\mathrm{in}}-z\|_{\bP}^2
    +\delta_i\prod_{j=1}^i T_j
  \Bigg].
\end{equation}
\end{theorem}

\subsection{The H\"older Case}\label{sec:VI_Holder}
We now specialize the inexact-oracle result to the case where the functions
\(p_i\) satisfy H\"older-type smoothness conditions. This allows us to use
exact values and exact (sub)gradients while choosing the oracle accuracy
parameters implicitly through the H\"older constants.
\begin{assumption}\label{ass:Holder_continuous_VI}
For each \(1\leq i\leq n\), the function \(p_i\) is convex and satisfies the
\((\nu_i,H_i)\)-H\"older condition on $\mathcal{Z}$ with respect to \(\|\cdot\|_{\bP}\).
\end{assumption}

For each \(1\leq i\leq n\) and \(\delta>0\), define
\begin{equation}\label{eq:Holder_L_of_delta}
  L_i(\delta)
  =
  \left[
    \frac{1-\nu_i}{2(1+\nu_i)}\cdot \frac{1}{\delta}
  \right]^{\frac{1-\nu_i}{1+\nu_i}}
  H_i^{\frac{2}{1+\nu_i}}.
\end{equation}
For \(\nu_i=1\), we use the convention \(L_i(\delta)=H_i\) for all
\(\delta\geq0\). As shown in Lemma~\ref{lem:Holder_to_inexact}, taking the
exact value \(p_i\) and an exact (sub)gradient \(p_i'\) as oracle outputs
gives a first-order \((\delta,L_i(\delta))\)-oracle for \(p_i\).

The next theorem is obtained by applying
Theorem~\ref{thm:VI_sliding_inexact} with a componentwise choice of the
oracle accuracy parameters. Its proof, together with
Lemma~\ref{lem:Holder_to_inexact}, is given in
Appendix~\ref{app:proof_VI_sliding_Holder}.

\begin{theorem}\label{thm:VI_sliding_Holder}
Let Assumptions~\ref{ass:Monotone_and_Lipschitz_Operators_VI}
and~\ref{ass:Holder_continuous_VI} hold, where \(M_i,H_i\geq0\) and
\(M_i+H_i>0\). Let \(T_1,\dots,T_n\) be positive integers and define
\[
  \Omega_{z_{\mathrm{in}}}
  :=
  \sup_{z\in\mathcal Z}
  \|z_{\mathrm{in}}-z\|_{\bP}^2
  <\infty,
  \qquad
  N_i:=\prod_{j=1}^i T_j.
\]
Run Algorithm~\ref{alg:sliding_recursive} with the sequence
\(\{\alpha_t\}\) in~\eqref{eq:alpha_sequence}, using exact values and exact
(sub)gradients in the induced oracle constructions. Choose
\begin{equation}\label{eq:Holder_eta}
  \eta_{\bt}^{k}
  =
  L_k(\delta_k^\circ)
  \prod_{\ell=1}^k \alpha_{t_\ell}
  +
  M_k
  \prod_{\ell=1}^k
  \frac{\alpha_{t_\ell}}{\alpha_{T_\ell-1}},
  \qquad
  \bt=(t_1,\dots,t_k).
\end{equation}
where, for \(0\leq\nu_i<1\),
\begin{equation}\label{eq:Holder_delta_star}
  \delta_i^\circ
  =
  \left[
    \frac{(1-\nu_i)\,2^{2i-1}K_i\,\Omega_{z_{\mathrm{in}}}}
    {(1+\nu_i)\,N_i^3}
  \right]^{(1+\nu_i)/2},
  \qquad
  K_i
  =
  \left[
    \frac{1-\nu_i}{2(1+\nu_i)}
  \right]^{\frac{1-\nu_i}{1+\nu_i}}
  H_i^{\frac{2}{1+\nu_i}}.
\end{equation}
For \(\nu_i=1\), set \(K_i=H_i\), \(\delta_i^\circ=0\), and
\(L_i(\delta_i^\circ)=H_i\). Then
\begin{equation}\label{eq:Holder_gap_bound}
  \sup_{z\in\mathcal Z}
  \left\{
    p(z_{\mathrm{out}})-p(z)
    +
    \langle Q(z),z_{\mathrm{out}}-z\rangle
  \right\}
  \leq
  \sum_{i=1}^n
  \Bigg[
    \frac{
      2^{i(1+\nu_i)}H_i
      \Omega_{z_{\mathrm{in}}}^{(1+\nu_i)/2}
    }{
      (1+\nu_i)
      N_i^{(1+3\nu_i)/2}
    }
    +
    \frac{2^{i-1}M_i
      \Omega_{z_{\mathrm{in}}}}{N_i}
  \Bigg].
\end{equation}
\end{theorem}
The following corollary states the resulting componentwise evaluation complexity. The proof is given in
Appendix~\ref{app:proof_VI_Holder_complexity}. A discussion on comparing with existing results is given in Appendix~\ref{app:discussion_sliding}

\begin{corollary}\label{cor:VI_sliding_Holder_complexity}
Let the assumptions of Theorem~\ref{thm:VI_sliding_Holder} hold, and set
\(\Omega:=\Omega_{z_{\mathrm{in}}}\). For any \(\varepsilon>0\) and
\(i=1,\dots,n\), define
\[
  R_i:=\max\!\left\{
  \left(\frac{H_i\Omega^{(1+\nu_i)/2}}{\varepsilon}\right)^{
  \frac{2}{1+3\nu_i}},\frac{M_i\Omega}{\varepsilon},1\right\}.
\]
Relabel the component pairs \((p_i,Q_i)\) so that
\(
  R_1\leq R_2\leq\cdots\leq R_n
\).
Then there exists a constant \(C_n>0\), depending only on \(n\), such that
performing no more than \(C_nR_i\) evaluations of the subgradient \(p_i'\)
and the operator \(Q_i\), for each \(1\leq i\leq n\), is sufficient to ensure
\[
  \sup_{z\in\cZ}
  \left\{
    p(z_{\mathrm{out}})-p(z)
    +
    \langle Q(z),z_{\mathrm{out}}-z\rangle
  \right\}
  \leq
  n\varepsilon .
\]
\end{corollary}

\section{Bilinear Saddle-Point Problem}\label{sec:bilinear_SPP}

We impose the following assumptions on
Problem~\eqref{prob:bilinear_SPP_main}.

\begin{assumption}\label{ass:properties_of_f}
  The function \(f \colon \R^{d_x} \to \R\) is convex and finite. Its
  restriction to \(\cX\) is \(\mu_x\)-strongly convex with \(\mu_x \geq 0\)
  and satisfies the \((\nu_x, H_x)\)-H\"older condition with respect to the
  Euclidean norm, with \(0 \leq \nu_x \leq 1\) and \(H_x > 0\).
\end{assumption}

\begin{assumption}\label{ass:properties_of_g}
  The function \(g \colon \R^{d_y} \to \R\) is convex and finite. Its
  restriction to \(\cY\) is \(\mu_y\)-strongly convex with \(\mu_y \geq 0\)
  and satisfies the \((\nu_y, H_y)\)-H\"older condition with respect to the
  Euclidean norm, with \(0 \leq \nu_y \leq 1\) and \(H_y > 0\).
\end{assumption}

\begin{assumption}\label{ass:properties_of_B}
  There exist constants \(L_{xy}>0\) and \(\mu_{xy},\mu_{yx}\geq0\), with
  \(L_{xy}>\max\{\mu_{xy},\mu_{yx}\}\). Let
  \[
    \begin{aligned}
      \mathsf R_x &:\
      \partial f(x)\subseteq\range(\bB^\top)
      \quad\forall x\in\cX,\\
      \mathsf R_y &:\
      \partial g(y)\subseteq\range(\bB)
      \quad\forall y\in\cY.
    \end{aligned}
  \]
  Then
  \[
    \begin{aligned}
      \mu_{xy}^2
      &\leq
      \begin{cases}
        \lminp(\bB^\top\bB),
        & \text{if } \mathsf R_x \text{ holds},\\
        \lmin(\bB^\top\bB),
        & \text{otherwise},
      \end{cases}
      \\[0.5ex]
      \mu_{yx}^2
      &\leq
      \begin{cases}
        \lminp(\bB\bB^\top),
        & \text{if } \mathsf R_y \text{ holds},\\
        \lmin(\bB\bB^\top),
        & \text{otherwise},
      \end{cases}
      \\[0.5ex]
      L_{xy}^2
      &\geq
      \lambda_{\max}(\bB^\top\bB)
      =
      \lambda_{\max}(\bB\bB^\top).
    \end{aligned}
  \]
  Additionally, if \(\mu_{xy}>0\) and \(\mu_{yx}>0\), then
  \(\mu_{xy}=\mu_{yx}\). Finally, if \(\mathsf R_x\) is used, we require
  \(\cX+\ker\bB=\cX\), and if \(\mathsf R_y\) is used,
  \(\cY+\ker\bB^\top=\cY\).
\end{assumption}

The last requirement is vacuous when the kernel is trivial, e.g.\ for
\(\cX=\R^{d_x}\); it keeps the iterates in the subspace on which \(\lminp\)
lower-bounds \(\bB\) (Appendix~\ref{app:subspace}).

\paragraph{Ambient-space functions and subgradients.}
In Assumptions~\ref{ass:properties_of_f}--\ref{ass:properties_of_B},
\(\partial f(x)\) and \(\partial g(y)\) denote subdifferentials in the ambient
spaces \(\R^{d_x}\) and \(\R^{d_y}\), respectively. If they were instead
defined relative to \(\cX\) and \(\cY\), their values at boundary points would
include contributions from the normal cones \(N_{\cX}(x)\) and
\(N_{\cY}(y)\). The range
conditions \(\mathsf R_x\) and \(\mathsf R_y\) would then unintentionally
restrict these normal cones. The strong-convexity and H\"older conditions are
imposed only on \(\cX\) and \(\cY\). The feasible sets enter the algorithm
through projections and the optimality conditions~\eqref{eq:opt_cond} through
their normal cones.

\paragraph{Ambient smoothness for active anchor terms.}
Whenever a result below uses the coupling-induced regularizer \(p_3\), we make
the following additional assumption for each of its active anchor terms. If
\(\nu_x=1\), then \(f\) is differentiable on \(\R^{d_x}\) and
\(\nabla f\) is \(H_x\)-Lipschitz on \(\R^{d_x}\); if \(\nu_y=1\), then
\(g\) is differentiable on \(\R^{d_y}\) and \(\nabla g\) is
\(H_y\)-Lipschitz on \(\R^{d_y}\). Thus the same constants used in the
decomposition are valid in the ambient spaces.

For input \(z_{\mathrm{in}}=(x_{\mathrm{in}},y_{\mathrm{in}})\), choose
oracle tolerances \(\delta_f,\delta_g\), with \(\delta_f>0\) if
\(\nu_x<1\) and \(\delta_g>0\) if \(\nu_y<1\). We define
\begin{align}
  L_x
  &:=
  \begin{cases}
    \left[
      \dfrac{1-\nu_x}{2(1+\nu_x)\delta_f}
    \right]^{\frac{1-\nu_x}{1+\nu_x}}
    H_x^{\frac{2}{1+\nu_x}},
    & 0\leq\nu_x<1,\\[1.2ex]
    H_x,
    & \nu_x=1,
  \end{cases}
  \\
  L_y
  &:=
  \begin{cases}
    \left[
      \dfrac{1-\nu_y}{2(1+\nu_y)\delta_g}
    \right]^{\frac{1-\nu_y}{1+\nu_y}}
    H_y^{\frac{2}{1+\nu_y}},
    & 0\leq\nu_y<1,\\[1.2ex]
    H_y,
    & \nu_y=1.
  \end{cases}
\end{align}

When \(\nu_x=1\), the ambient-smoothness convention makes \(H_x\) an
ambient-space Lipschitz constant of \(\nabla f\) whenever its anchor term is
active, and we write \(L_x := H_x\). Similarly, when \(\nu_y=1\), we write
\(L_y := H_y\).

For convex \(h\) with an ambient \(L_h\)-Lipschitz gradient, the standard
self-bounding inequality gives, for all \(u,v\),
\begin{equation}\label{eq:ambient_self_bounding}
  \|\nabla h(u)-\nabla h(v)\|^2
  \leq
  2L_h\bigl(
    h(u)-h(v)
    -\langle\nabla h(v),u-v\rangle
  \bigr).
\end{equation}
Smoothness restricted to an arbitrary closed convex set is insufficient for
\eqref{eq:ambient_self_bounding}.

We define
\begin{equation}\label{eq:SPP_beta_xy_holder}
  \beta_x:=\frac{\mathbf{1}_{\{\nu_y=1\}}}{4L_y},
  \qquad
  \beta_y:=\frac{\mathbf{1}_{\{\nu_x=1\}}}{4L_x}.
\end{equation}

We use the following effective strong convexity parameters:
\begin{equation}\label{eq:effective_strong_convexity_holder}
  \delta_x := \mu_x + 4\beta_x\mu_{xy}^2,
  \qquad
  \delta_y := \mu_y + 4\beta_y\mu_{yx}^2 .
\end{equation}
Assuming that \(\min\{\delta_x,\delta_y\}>0\), we define
\begin{align}
  \kappa_{x} &= \frac{L_{x}}{\delta_{x}}, \qquad
  \kappa_{y} = \frac{L_{y}}{\delta_{y}}, \qquad
  \kappa_{xy} = \frac{L_{xy}^2}{\delta_{x} \delta_{y}}, \label{eq:kappa_s}
\end{align}
and
\begin{align}
  \tilde{H}_{x} &= H_x \delta_{x}^{-(1+\nu_x)/2}, \qquad
  \tilde{H}_{y} = H_y \delta_{y}^{-(1+\nu_y)/2}. \label{eq:H_s}
\end{align}

In the smooth case, \citet{borodich2025linear} obtain separated complexities
\(\widetilde{\mathcal O}(\sqrt{\kappa_x})\),
\(\widetilde{\mathcal O}(\sqrt{\kappa_y})\), and
\(\widetilde{\mathcal O}(\sqrt{\kappa_{xy}})\) for \(\nabla f\),
\(\nabla g\), and products with \(\bB,\bB^\top\), respectively, where
\(\widetilde{\mathcal O}\) suppresses logarithms in \(1/\epsilon\) and the
problem parameters. We extend these guarantees to H\"older-continuous
\(f\) and \(g\).

We set \(n=3\) and
\begin{align}
  p_1(x, y) &= f(x), \quad p_2(x, y) = g(y), \nonumber\\
  p_3(x,y)&=\tfrac{\beta_x}{2}\|\bB x-g'(\yin)\|^2
  +\tfrac{\beta_y}{2}\|\bB^\top y+f'(\xin)\|^2.
  \label{eq:SPP_decomposition}
\end{align}

and the operators
\begin{equation}\label{eq:SPP_operator}
  \begin{aligned}
  Q_1(x,y)&=0,\qquad Q_2(x,y)=0,\\
  Q_3(x,y)&=
    \pmat{\bO_{d_x}&\bB^\top\\-\bB&\bO_{d_y}}
    \pmat{x\\y}.
  \end{aligned}
\end{equation}

We define the norm matrix \(\bP \in \bbS_{++}^{d_z}\) as
\begin{equation}
  \bP = \diag (\delta_{x} \bI_{d_x}, \delta_{y} \bI_{d_y}).
\end{equation}

Algorithm~\ref{alg:three_level_bilinear_spp} in
Appendix~\ref{app:explicit_sliding_algorithms} specializes
Algorithm~\ref{alg:sliding_recursive} to Problem~\eqref{prob:bilinear_SPP_main}.
Since level \(i\) is evaluated \(\prod_{j\leq i}T_j\) times and the additive
decomposition is permutation invariant, we place the cheapest component
outermost.

Let \(\cS\neq\varnothing\) be the solution set of
Problem~\eqref{prob:bilinear_SPP_main}. By
Lemma~\ref{lem:solution_set_structure} we may fix \(z^*=(x^*,y^*)\in\cS\) whose
\(x\)-block satisfies \(\xin-x^*\in\range(\bB^\top)\) whenever \(\mathsf R_x\)
is invoked, and whose \(y\)-block satisfies \(\yin-y^*\in\range(\bB)\) whenever
\(\mathsf R_y\) is invoked; on a block whose range condition is not invoked no
restriction is placed on \(z^*\), and
Assumption~\ref{ass:properties_of_B} supplies the unrestricted bound
\(\lmin\) there. By first-order optimality
there are \(f'(x^*)\in\partial f(x^*)\), \(g'(y^*)\in\partial g(y^*)\) and
normal vectors \(\xi\in N_{\cX}(x^*)\), \(\zeta\in N_{\cY}(y^*)\) with
\begin{equation}\label{eq:opt_cond}
  f'(x^*)+\bB^\top y^*+\xi=0,
  \qquad
  g'(y^*)-\bB x^*+\zeta=0;
\end{equation}
we fix such a choice. Define the Bregman divergences
\begin{equation}\label{eq:def_D_f_g}
  \begin{aligned}
  \D_f(x,x^*)
  &=
  f(x)-f(x^*)
  -\langle f'(x^*),x-x^*\rangle,\\
  \D_g(y,y^*)
  &=
  g(y)-g(y^*)
  -\langle g'(y^*),y-y^*\rangle.
  \end{aligned}
\end{equation}
We denote by \(N_f\), \(N_g\), and \(N_B\) the numbers of evaluations of
\(f'\), of \(g'\), and of matrix-vector products with \(\bB\) or \(\bB^\top\).
The next theorem gives a one-step contraction for a single call of
Algorithm~\ref{alg:sliding_recursive} initialized at \(z_{\mathrm{in}}\); the
proof is in Appendix~\ref{app:proof_bilinear_SPP_Holder}.

\begin{theorem}\label{thm:bilinear_SPP_Holder}
  Let Assumptions~\ref{ass:properties_of_f}--\ref{ass:properties_of_B}
  hold, together with the ambient-smoothness condition above for every active
  smooth anchor term. Thus~\eqref{eq:ambient_self_bounding} is available
  whenever it is used below. Assume \(\min\{\delta_x,\delta_y\}>0\) and
  $
    \beta_x\delta_y \leq \frac14,
    \beta_y\delta_x \leq \frac14
  $. Assume further that the normal vectors in the optimality
  conditions~\eqref{eq:opt_cond} satisfy \(\zeta=0\) if \(\beta_x>0\) and
  \(\xi=0\) if \(\beta_y>0\). This condition is automatic at interior saddle
  points and vacuous when \(\nu_x,\nu_y<1\)
  (Remark~\ref{rem:exact_optimality}). Let \(\Omega_{\mathrm{in}}\) be any
  upper bound satisfying
  $
    \Psi(\zin)\leq \Omega_{\mathrm{in}}$,  where
  \begin{equation}\label{eq:Psi_s}
    \Psi(z)
    :=
    \|z-z^*\|_{\bP}^2
    +
    12\D_f(x,x^*)
    +
    12\D_g(y,y^*).
  \end{equation}
  Associate with \(p_1,p_2,p_3\) the constants and componentwise budgets
  \begin{equation}\label{eq:holder_spp_oracle_params}
    \begin{aligned}
    &L_1=\kappa_x,\quad L_2=\kappa_y,\quad L_3=\kappa_{xy},
    \qquad M_1=M_2=0,\quad M_3=\sqrt{\kappa_{xy}},\\
    &R_1=\max\{\widehat R_1,1\},\quad
    R_2=\max\{\widehat R_2,1\},\quad
    R_3=\max\{\sqrt{\kappa_{xy}},1\}.
    \end{aligned}
  \end{equation}
  where
  \(
    \widehat R_1
    :=
    \tilde H_x^{\frac{2}{1+3\nu_x}}
    \Omega_{\mathrm{in}}^{\frac{\nu_x-1}{1+3\nu_x}}
  \)
  and
  \(
    \widehat R_2
    :=
    \tilde H_y^{\frac{2}{1+3\nu_y}}
    \Omega_{\mathrm{in}}^{\frac{\nu_y-1}{1+3\nu_y}}
  \).
  Choose \(\delta_f,\delta_g\) as in the proof and apply
  Algorithm~\ref{alg:sliding_recursive} to
  \eqref{eq:SPP_decomposition}--\eqref{eq:SPP_operator} with the components
  sorted along the levels, i.e.\ with \(\big(p_{\pi(i)},Q_{\pi(i)}\big)\) at
  level \(i\) for a permutation \(\pi\) of \(\{1,2,3\}\) satisfying
  \(R_{\pi(1)}\leq R_{\pi(2)}\leq R_{\pi(3)}\).
  Then the output \(z_{\mathrm{out}}\) satisfies
\(
  \Psi(\zout)
  \leq
  \frac34\Omega_{\mathrm{in}}
\),
with
\(
  N_f
  =
  \mathcal O\left(R_1\right)
\),
\(
  N_g
  =
  \mathcal O\left(R_2\right)
\),
and
\(
  N_B
  =
  \mathcal O\left(R_3\right)
\).

\end{theorem}

We apply a standard restarting scheme. This yields the following
global complexity bound. The proof is given in
Appendix~\ref{app:proof_bilinear_SPP_restart}.

\begin{corollary}\label{cor:bilinear_SPP_restart}
Let the assumptions of Theorem~\ref{thm:bilinear_SPP_Holder} hold. Given
\(z^0\in\cZ\) and \(\varepsilon\in(0,\Psi(z^0)]\), run the restarted
Sliding Algorithm with
\(
  \Omega_s:=\left(3/4\right)^s\Psi(z^0)
\)
for
\(
  S:=
  \left\lceil
    \frac{\log(\Psi(z^0)/\varepsilon)}{\log(4/3)}
  \right\rceil
\)
restarts. Then the output \(z^S\) satisfies \(\Psi(z^S)\leq\varepsilon\),
and the total evaluation complexity is given in
Table~\ref{tab:bilinear_SPP_restart_complexity}.
\end{corollary}

\begin{table}[H]
  \centering
  \footnotesize
  \setlength{\tabcolsep}{3pt}
  \renewcommand{\arraystretch}{1.35}
  \begin{tabular}{@{}ll@{}}
    \toprule
    Oracle & Complexity \\
    \midrule
    \(f'\) evaluations
    &
    \(\displaystyle
      \begin{cases}
        \mathcal O\!\left(
          \tilde H_x^{\frac{2}{1+3\nu_x}}
          \varepsilon^{\frac{\nu_x-1}{1+3\nu_x}}
        \right), & 0\leq\nu_x<1,\\
        \widetilde{\mathcal O}\!\left(\sqrt{\kappa_x}\right),
          & \nu_x=1
      \end{cases}
    \)
    \\[0.8ex]
    \(g'\) evaluations
    &
    \(\displaystyle
      \begin{cases}
        \mathcal O\!\left(
          \tilde H_y^{\frac{2}{1+3\nu_y}}
          \varepsilon^{\frac{\nu_y-1}{1+3\nu_y}}
        \right), & 0\leq\nu_y<1,\\
        \widetilde{\mathcal O}\!\left(\sqrt{\kappa_y}\right),
          & \nu_y=1
      \end{cases}
    \)
    \\[0.8ex]
    \(\bB,\bB^\top\) products
    &
    \(\displaystyle
      \widetilde{\mathcal O}\!\left(\sqrt{\kappa_{xy}}\right)
    \)
    \\
    \bottomrule
  \end{tabular}
  \caption{Evaluation complexity for the bilinear saddle-point
  problem~\eqref{prob:bilinear_SPP_main}. The notation
  \(\widetilde{\mathcal O}\) hides logarithmic factors in
  \(\Psi_0/\varepsilon\).}
  \label{tab:bilinear_SPP_restart_complexity}
\end{table}

Table~\ref{tab:bilinear_SPP_restart_complexity} interpolates componentwise
between H\"older and smooth strongly convex rates. For \(0\leq\nu_x<1\),
\(N_f=\mathcal O\bigl(\tilde H_x^{\frac{2}{1+3\nu_x}}
\varepsilon^{-\frac{1-\nu_x}{1+3\nu_x}}\bigr)\), with the analogous bound
for \(N_g\). At \(\nu_x=1\) or \(\nu_y=1\), the corresponding rate becomes
\(\widetilde{\mathcal O}(\sqrt{\kappa_x})\) or
\(\widetilde{\mathcal O}(\sqrt{\kappa_y})\), with logarithmic accuracy
dependence. Coupling requires
\(\widetilde{\mathcal O}(\sqrt{\kappa_{xy}})\) products independently of
\(\nu_x,\nu_y\). Thus \((\nu_x,\nu_y)=(1,1)\) recovers the smooth bilinear
complexity.

\paragraph{The case \(\delta_x=\delta_y=0\)}
\label{sec:bilinear_SPP_degenerate}

We finally discuss the case \(\delta_x=\delta_y=0\). Assuming that \(\cZ\) is bounded, we apply
Theorem~\ref{thm:VI_sliding_Holder} directly with respect to the Euclidean
product norm. Since no strong convexity is available, we set
$ p_3\equiv 0$.
Consequently, this result requires only the H\"older conditions on
\(\cX,\cY\), not the ambient-smoothness condition for active anchor terms.

Thus only the bilinear coupling remains in \(Q_3\), which is
\(L_{xy}\)-Lipschitz. Hence, with
\(\Omega:=\sup_{z\in\cZ}\|z^0-z\|^2\),
Corollary~\ref{cor:VI_sliding_Holder_complexity} gives
\(N_f=\mathcal O\bigl((H_x\Omega^{(1+\nu_x)/2}/\varepsilon)^{
2/(1+3\nu_x)}\bigr)\). Analogously,
\(N_g=\mathcal O\bigl((H_y\Omega^{(1+\nu_y)/2}/\varepsilon)^{
2/(1+3\nu_y)}\bigr)\). The coupling complexity is
\(N_B=\mathcal O(L_{xy}\Omega/\varepsilon)\).
These bounds describe the genuinely non-strongly-convex--non-strongly-concave
regime. The \(H_x\)- and \(H_y\)-terms match the known lower-bound scaling for
first-order convex optimization with H\"older-continuous gradients
\citep{nemirovskij1983problem,nesterov2013introductory,nesterov2015universal},
interpolating from \(\mathcal O(\varepsilon^{-2})\) at \(\nu=0\) to the
accelerated smooth rate \(\mathcal O(\varepsilon^{-1/2})\) at \(\nu=1\).
The bilinear term \(\mathcal O(L_{xy}\Omega/\varepsilon)\) matches the
optimal Mirror-Prox rate for monotone Lipschitz variational inequalities and
is consistent with lower bounds for bilinear saddle-point problems
\citep{nemirovski2004prox,zhang2022lower}.

\paragraph{Mixed cases.}
\label{sec:bilinear_SPP_mixed}

We next consider the regimes \(\delta_x>0,\delta_y=0\) and
\(\delta_x=0,\delta_y>0\), where only one effective
strong-convexity parameter is positive. In contrast to the fully degenerate
case \(\delta_x=\delta_y=0\), we do not discard the coupling-induced
regularization \(p_3\) in~\eqref{eq:SPP_decomposition}. Instead, we add
regularization only in the block whose effective curvature is zero and rebuild
\(p_3\) using the regularized component's H\"older/smoothness constant and
anchor whenever that part of \(p_3\) is active. The corrected parameters and
compatibility conditions are given in Appendix~\ref{app:mixed_cases}. After
simplification, the rates are summarized in
Table~\ref{tab:bilinear_SPP_mixed_complexity}.

\begin{table}[H]
  \centering
  \footnotesize
  \setlength{\tabcolsep}{3pt}
  \renewcommand{\arraystretch}{1.25}
  \begin{tabular}{@{}ll@{}}
    \toprule
    Oracle & Complexity \\
    \midrule
    \(f'\) evaluations
    &
    \(\displaystyle
      \begin{cases}
        \mathcal O\!\left(
          \tilde H_x^{\frac{2}{1+3\nu_x}}
          \varepsilon^{\frac{\nu_x-1}{1+3\nu_x}}
        \right), & 0\leq\nu_x<1,\\
        \widetilde{\mathcal O}\!\left(\sqrt{\kappa_x}\right),
          & \nu_x=1
      \end{cases}
    \)
    \\[0.8ex]
    \(g'\) evaluations
    &
    \(\displaystyle
      \begin{cases}
        \mathcal O\!\left(
          \left(
            \frac{H_y\Omega_y^{(1+\nu_y)/2}}{\varepsilon}
          \right)^{\frac{2}{1+3\nu_y}}
        \right), & 0\leq\nu_y<1,\\
        \widetilde{\mathcal O}\!\left(
          \sqrt{\frac{H_y\Omega_y}{\varepsilon}}
        \right), & \nu_y=1
      \end{cases}
    \)
    \\[0.8ex]
    \(\bB,\bB^\top\) products
    &
    \(\displaystyle
      \widetilde{\mathcal O}\!\left(
        L_{xy}\sqrt{\frac{\Omega_y}{\delta_x\varepsilon}}
      \right)
    \)
    \\
    \bottomrule
  \end{tabular}
  \caption{Mixed-regime complexity for \(\delta_x>0,\delta_y=0\).}
  \label{tab:bilinear_SPP_mixed_complexity}
\end{table}

\subsection{Stochastic oracle setting}
\label{sec:bilinear_SPP_stochastic_main}

We allow \(0\leq\nu_x,\nu_y\leq1\) and access \(f\) and \(g\) through
conditionally unbiased stochastic (sub)gradient oracles with Euclidean
conditional variances bounded by \(\sigma_x^2\) and \(\sigma_y^2\).
The bilinear products are evaluated exactly. At \(\nu_q=0\), the oracle mean
is a measurable selected subgradient; the precise filtration and moment
conditions are stated in Assumption~\ref{ass:stochastic_oracle_VI}.
The parameter choices use valid H\"older and variance bounds.

The stochastic recursive estimate combines the deterministic H\"older model
error with the martingale error. Optimizing their parameters separately gives
the componentwise terms
\[
  \frac{H_iR^{1+\nu_i}}{N_i^{(1+3\nu_i)/2}}
  +\frac{M_iR^2}{N_i}
  +\frac{\sigma_iR}{\sqrt{N_i}},
\]
up to constants depending only on the number of levels. Here \(R^2\)
bounds the expected initial squared distance to a fixed comparator; for
the uniform gap we use \(R^2=\Omega\), defined below. The uniform version
is proved before taking expectation of the supremum in
Theorem~\ref{thm:VI_sliding_stochastic_uniform}; a fixed-comparator estimate
alone does not provide a uniform gap guarantee.

Table~\ref{tab:bilinear_SPP_stochastic_complexity} summarizes two consequences.
For the bounded convex--concave regime, let
\[
  \Omega:=\sup_{z\in\cZ}\|z^0-z\|^2<\infty,
  \qquad
  \cC_q(\varepsilon)
  :=\left(\frac{H_q\Omega^{(1+\nu_q)/2}}{\varepsilon}\right)^{\!2/(1+3\nu_q)},
  \quad q\in\{x,y\}.
\]
For the strongly convex--strongly concave regime, assume
\(\mu_x,\mu_y>0\), set \(\bP=\diag(\mu_x I,\mu_y I)\), and let
\(\Gamma_0>0\) be a known upper bound on \(\Psi(z^0)\), where \(\Psi\) is defined by
\eqref{eq:Psi_s}. For \(0<\varepsilon\leq\Gamma_0\), write
\[
  \ell_\varepsilon:=1+\log\frac{\Gamma_0}{\varepsilon},\qquad
  \theta_q:=\frac{1-\nu_q}{1+3\nu_q},\qquad
  \cK_q:=H_q^{2/(1+3\nu_q)}\mu_q^{-(1+\nu_q)/(1+3\nu_q)},
\]
and define the deterministic restart budget
\[
  \cD_q(\varepsilon):=
  \begin{cases}
    \displaystyle
    \ell_\varepsilon+
    \frac{\cK_q\varepsilon^{-\theta_q}}{1-2^{-\theta_q}},
      &0\leq\nu_q<1,\\[1ex]
    \displaystyle
    \left(1+\sqrt{H_q/\mu_q}\right)\ell_\varepsilon,
      &\nu_q=1.
  \end{cases}
\]
The finite geometric sum in
Theorem~\ref{thm:bilinear_SPP_stochastic_Holder_strongly_convex} gives a
sharper bound near \(\nu_q=1\). For each fixed \(\nu_q<1\), the factor
\((1-2^{-\theta_q})^{-1}\) is a constant and the H\"older accuracy exponent
is \((1-\nu_q)/(1+3\nu_q)\).

\begin{table}[H]
  \centering
  \small
  \setlength{\tabcolsep}{4pt}
  \renewcommand{\arraystretch}{1.35}
  \begin{tabular}{@{}lcc@{}}
    \toprule
    Oracle & Bounded convex--concave & \(\mu_x,\mu_y>0\) \\
    \midrule
    \(f'\) samples
    & \(\displaystyle
       \cO\!\left(1+\cC_x(\varepsilon)
                     +\frac{\sigma_x^2\Omega}{\varepsilon^2}\right)\)
    & \(\displaystyle
       \cO\!\left(\cD_x(\varepsilon)
                     +\frac{\sigma_x^2}{\mu_x\varepsilon}\right)\) \\
    \(g'\) samples
    & \(\displaystyle
       \cO\!\left(1+\cC_y(\varepsilon)
                     +\frac{\sigma_y^2\Omega}{\varepsilon^2}\right)\)
    & \(\displaystyle
       \cO\!\left(\cD_y(\varepsilon)
                     +\frac{\sigma_y^2}{\mu_y\varepsilon}\right)\) \\
    \(\bB,\bB^\top\) products
    & \(\displaystyle
       \cO\!\left(1+\frac{L_{xy}\Omega}{\varepsilon}\right)\)
    & \(\displaystyle
       \cO\!\left(\left(1+
          \frac{L_{xy}}{\sqrt{\mu_x\mu_y}}\right)\ell_\varepsilon\right)\) \\
    \bottomrule
  \end{tabular}
  \caption{Stochastic H\"older oracle complexities for
  \(0\leq\nu_x,\nu_y\leq1\). The bounded-domain column guarantees
  \(\mathbb E[\gap(z^{\rm out})]\leq\varepsilon\); the strongly convex
  column guarantees \(\mathbb E[\Psi(z^S)]\leq\varepsilon\).
  The latter uses intrinsic curvature \(\mu_x,\mu_y>0\) and no anchor
  regularizer. The quantities \(\cC_q,\cD_q\) are defined above.}
  \label{tab:bilinear_SPP_stochastic_complexity}
\end{table}

For bounded \(\cZ\), we set \(p_3\equiv0\), use the Euclidean product
norm, and obtain the uniform expected primal--dual gap in
Theorem~\ref{thm:bilinear_SPP_stochastic_Holder_bounded}. This result includes
the degenerate regime \(\delta_x=\delta_y=0\). Under intrinsic strong
convexity, we again set \(p_3\equiv0\), now in the weighted norm above, and
restart Algorithm~\ref{alg:sliding_recursive}. The fixed saddle comparator
controls \(\Psi\), giving
\begin{equation}\label{eq:stoch_strong_accuracy_main}
  \mathbb E[\Psi(z^S)]\leq\varepsilon.
\end{equation}
No anchor samples are used, so this guarantee has no anchor noise floor.
The variance terms are free of restart logarithms. At every call or restart,
the recursive levels are ordered by their complete target budgets, including
the variance contributions; this preserves the separate coupling count.

The strongly convex theorem imposes H\"older regularity on the feasible sets,
as in Assumptions~\ref{ass:properties_of_f}--\ref{ass:properties_of_g}.
Global strong convexity and global H\"older regularity with exponent below
one cannot hold simultaneously on an unbounded domain. The stochastic
strongly convex column uses \(\mu_x,\mu_y>0\); it does not assert the same
result for curvature supplied only by the coupling. At
\(\nu_x=\nu_y=1\), the bounds recover the separated smooth stochastic
rates, and at \(\sigma_x=\sigma_y=0\) they recover the corresponding
deterministic H\"older rates.

\section{Experiments}\label{sec:experiments}
We report exact gaps and oracle counts on synthetic and tomography instances.
Appendix~\ref{app:experiments} gives the setup and tuning.

\paragraph{Synthetic results.}
Across the tested \(\nu_x\), the fitted \(N_f\) exponent is within \(4\%\) of
\(2/(1+3\nu_x)\), with \(R^2\geq0.9999\) and per-seed spread below \(0.04\).
All \(90\) runs satisfy~\eqref{eq:Holder_gap_bound}, attaining
\(8\%\)--\(27\%\) of the bound (Table~\ref{tab:e1_slopes}).
On \((\nu_x,\nu_y)=(0.75,0.25)\), recursion reaches gap below \(0.1\) with
\(48\) \(f'\)-calls, versus \(707\) for the lockstep ablation and \(2{,}477\)
for Universal Mirror-Prox. Relative to the two-level ablation, it keeps
\(48\) \(f'\)-calls and trades \(1{,}584\to1{,}920\) \(g'\)-calls for
\(6{,}336\to3{,}840\) \(\bB,\bB^\top\)-products
(Figure~\ref{fig:e2_separation} in Appendix~\ref{app:experiments}).
The appendix also covers other regularities. On a smooth instance, sliding uses
\(4.7\)--\(8.5\times\) fewer gradients but \(18\)--\(40\times\) more products,
with a gradient-to-product cost crossover of \(22\)--\(44\)
(Appendix~\ref{app:exp_cost}).

\paragraph{Tomography.}
At relative gap \(5\cdot10^{-4}\), we compare against Mirror-Prox, Accelerated
Mirror-Prox, and Mirror-Prox Sliding (MPS) on five unseen instances. Standard
Huber--TV is the negative control. Figure~\ref{fig:e5_convergence_main} shows
that MPS reaches the target first with diagonal \(g'\), whereas recursion has
the lowest weighted cost with nonlocal and nonlocal H\"older \(g'\). Across the
five instances, MPS wins \(2.43\) vs.\ \(3.18\) s in the diagonal regime. A
nonlocal graph makes \(g'\) \(48\times\) costlier.
Recursion uses \(480\) rather than \(3840\) \(g'\)-calls and takes \(3.96\)
rather than \(8.39\) s with the same \(160\) data gradients, reaching
\(30.81\) dB at \(128^2\). Adding a \(\nu_y=\tfrac12\) power term makes \(g'\)
genuinely H\"older. The times are \(3.98\) s for recursion, \(8.56\) s for MPS,
and \(10.8\) s for Generalized Universal Mirror-Prox at its target-matched
\(576\)-iteration budget. Appendix~\ref{app:exp_ct} gives diagnostics and
curves. Coupling calls exceed \(g'\)-calls by three orders of magnitude, so
tomography does not reveal the \(\nu_y\)-dependent exponent.

\begin{figure}[H]
  \centering
  \includegraphics[width=\textwidth]{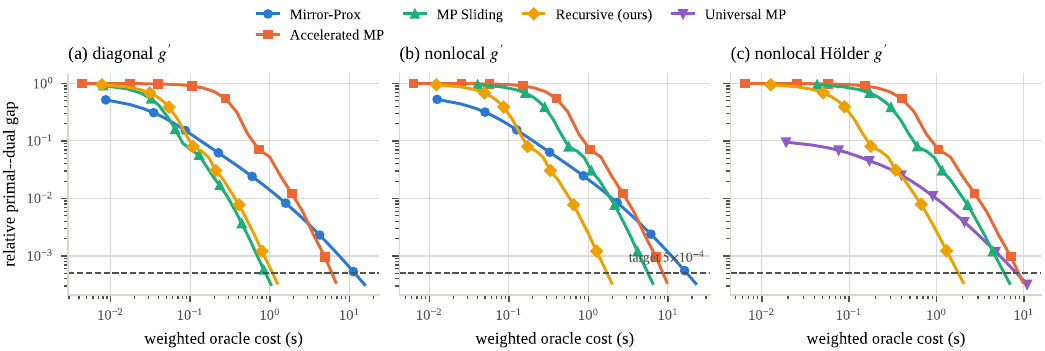}
  \caption{Tomographic convergence on the first unseen instance versus
  calibrated cost \(c_fN_f+c_gN_g+c_BN_B\), where \(N_f,N_g,N_B\) count
  \(f'\), \(g'\), and \(B/B^\top\) calls and \(c_f,c_g,c_B\) are their
  measured times per call. Panels (a)--(c) use diagonal, nonlocal, and
  nonlocal H\"older \(g'\), with
  \((\nu_x,\nu_y)=(1,1),(1,1),(1,\tfrac12)\), respectively. The dashed line
  marks the relative-gap target \(5\cdot10^{-4}\).}
  \label{fig:e5_convergence_main}
\end{figure}

\section{Conclusions}\label{sec:conclusions}

We extended recursive componentwise sliding to componentwise inexact
first-order oracles and applied it to H\"older-smooth bilinear saddle-point
problems. The resulting bounds charge \(f'\), \(g'\), and
\(\bB,\bB^\top\) separately, with the two gradient counts governed by their
respective \((H_x,\nu_x)\) and \((H_y,\nu_y)\). At
\(\nu_x=\nu_y=1\), they reduce to the known optimal smooth rates. The analysis
also covers nonstrongly convex, mixed, and strongly curved deterministic
regimes under the stated conditions. For stochastic function oracles, the
H\"older model error and the variance contribution separate in the recursive
bound. This gives a uniform expected-gap guarantee on bounded domains for
all \(\nu_x,\nu_y\in[0,1]\). With intrinsic strong convexity
\(\mu_x,\mu_y>0\), restarts of the decomposition with \(p_3=0\) give an
expected Lyapunov guarantee and separate \(\sigma_q^2/(\mu_q\varepsilon)\)
variance costs without an anchor noise floor. H\"older regularity is required
on the feasible sets, and the coupling products remain exact.
Synthetic experiments recover the predicted componentwise exponents. The
tomography benchmark favors recursive sliding for expensive nonlocal \(g'\),
including the H\"older case, and MPS for cheap diagonal \(g'\).
Future work includes parameter-free recursive schedules, stochastic coupling
oracles, and heterogeneous-exponent lower bounds.

\section*{Acknowledgements}
The research is supported by the Ministry of Science and Higher
Education of the Russian Federation, project No.~FSMG-2024-0011.

\begingroup
\small
\setlength{\bibhang}{1.5em}
\setlength{\bibsep}{3pt plus 1pt minus 1pt}
\bibliographystyle{abbrvnat}
\bibliography{references}
\endgroup
\clearpage
\appendix

\section{The Sliding Algorithm and Explicit Special Cases}
\label{app:explicit_sliding_algorithms}

The general recursive scheme referred to throughout the paper is
Algorithm~\ref{alg:sliding_recursive}. Its one- and two-level cases and the
three-level bilinear specialization follow it.

\begin{algorithm}[H]
\caption{Sliding Algorithm}\label{alg:sliding_recursive}
\begin{algorithmic}[1]
\State \textbf{Input:} $z_{\mathrm{in}} \in \cZ$. \label{line:sliding_input}
\State \textbf{Parameters:} $\{\alpha_t\}$, $\{T_i\}_{i=1}^n$, $\{\eta^k_{\bt}\}$ \label{line:sliding_parameters}

\For{$k=1,\dots,n$} \label{line:init_z_loop}
  \State \(z^k_{\mathbf{0}_k}=z_{\mathrm{in}}\), where
  \(\mathbf{0}_k=(0,\dots,0)\in\mathbb{N}^k\). \label{line:init_z}
\EndFor

\For{$i=1,\dots,n$} \label{line:init_p_loop}
  \State \(p_i^{0,\varnothing}(z)\equiv p_i(z)\). \label{line:init_p}
\EndFor

\State $z_{\mathrm{out}} = \textsc{RecursiveProcedure}(0,\varnothing)$ \label{line:compute_z_out}
\State \Return $z_{\mathrm{out}}$ \label{line:return_z_out}
\State

\Function{RecursiveProcedure}{$k, \bt$} \Comment{\(z^\ell_{\bt,0,\dots,0}\) is initialized for all \(\ell>k\)} \label{line:recursive_procedure}
  \If{$k = n$} \label{line:base_case_if}
    \State \Return $\argmin_{z\in\cZ} \sum_{i=1}^n p_i^{n,\bt}(z)$ \label{line:base_case_return}
  \EndIf

  \State Set $\bar{z}_{\bt, 0}^{k+1} = z^{k+1}_{\bt, 0}$. \label{line:init_bar_z}

  \For{$j = 0, \dots, T_{k+1} - 1$} \label{line:level_loop}
    \State Set $\bt' = (\bt, j)$, $\bt'' = (\bt, j+1)$ \label{line:define_bt_prime}

    \For{$i = 1, \dots, n$} \label{line:hat_p_loop}
    \State $\hat{p}_i^{k+1,\bt'}(z) =
      \begin{cases}
        \alpha_j^{-1} \, p_i^{k,\bt}(\alpha_j z + (1-\alpha_j)\bar{z}_{\bt'}^{k+1}), & i > k, \\
        p_i^{k,\bt}(z), & i \leq k,
      \end{cases}$ \label{line:define_hat_p}
    \EndFor

    \For{$i = 1, \dots, n$} \label{line:p_loop}
    \State $p_i^{k+1,\bt'}(z) =
      \begin{cases}
        \frac{\eta_{\bt^\prime}^{k+1}}{2}\|z - z^{k+1}_{\bt'}\|_{\bP}^2
        + \langle \tilde{\nabla}\hat{p}_{k+1}^{k+1,\bt'}(z^{k+1}_{\bt'})
        + Q_{k+1}(z^{k+1}_{\bt'}), z \rangle, & i = k+1, \\
        \hat{p}_i^{k+1,\bt'}(z), & i \neq k+1,
      \end{cases}$ \label{line:define_p_next}
    \EndFor
    \State $\tilde{z}_{\bt'}^{k+1}
    = \textsc{RecursiveProcedure}(k+1, \bt')$ \label{line:recursive_call}

    \State $\bar{z}_{\bt''}^{k+1}
    = \alpha_j \tilde{z}_{\bt'}^{k+1}
    + (1 - \alpha_j) \bar{z}_{\bt'}^{k+1}$ \label{line:update_bar_z}

    \State $z_{\bt''}^{k+1}
    =
    \argmin_{z\in\cZ}
    \left\{
    \left\langle
    Q_{k+1}(\tilde z_{\bt'}^{k+1})
    -
    Q_{k+1}(z_{\bt'}^{k+1}),
    z
    \right\rangle
    +
    \frac{\eta_{\bt'}^{k+1}}{2}
    \|z-\tilde z_{\bt'}^{k+1}\|_{\bP}^2
    \right\}$ \label{line:update_z}

    \If{\(k+2\leq n\)} \label{line:warm_start_if}
      \For{\(\ell=k+2,\dots,n\)} \label{line:warm_start_loop}
        \State Set $z^\ell_{\bt'',0,\dots,0}
        = z^\ell_{\bt',T_{k+2},0,\dots,0}$. \label{line:warm_start}
      \EndFor
    \EndIf
  \EndFor

  \State \Return $\bar{z}_{\bt, T_{k+1}}^{k+1}$ \label{line:recursive_return}
\EndFunction
\end{algorithmic}
\end{algorithm}

\begin{algorithm}[H]
\caption{One-Level Sliding (\(n = 1\))}\label{alg:one_level}
\begin{algorithmic}[1]
\State Choose $z_0 \in \cZ$ and set $\bar{z}^0 = z_0$.
\For{$t = 0, \dots, T - 1$}
  \State $z^t = \alpha_t z_t + (1 - \alpha_t) \bar{z}^t$
  \State $\tilde{z}_t = \argmin_{z\in\cZ} \left\{ \langle \tilde{\nabla}p(z^t) + Q(z_t), z \rangle + \frac{\beta_t}{2} \|z - z_t\|_{\bP}^2 \right\}$
  \State $z_{t+1} = \argmin_{z\in\cZ} \left\{ \langle Q(\tilde{z}_t) - Q(z_t), z \rangle + \frac{\beta_t}{2} \|z - \tilde{z}_t\|_{\bP}^2 \right\}$
  \State $\bar{z}^{t+1} = \alpha_t \tilde{z}_t + (1 - \alpha_t) \bar{z}^t$
\EndFor
\State \Return $\bar{z}^{T}$
\end{algorithmic}
\end{algorithm}

\begin{algorithm}[H]
\caption{Two-Level Sliding (special case \(n = 2\))}\label{alg:two_level}
\begin{algorithmic}[1]
\State Choose $z_0 \in \cZ$ and set $\bar{z}_0 = z_0^0 = z_0$.
\For{$k = 0, \dots, T_1 - 1$}
  \State $\underline{z}_k = \alpha_k z_k + (1 - \alpha_k) \bar{z}_k$
  \State Set $\bar{z}_k^0 = z_k^0$
  \For{$t = 0, \dots, T_2 - 1$}
    \State $\hat{z}^t_k = \alpha_t z_k^t + (1-\alpha_t)\bar{z}^t_k$
    \State $\underline{z}^t_k = \alpha_k \hat{z}^t_k + (1-\alpha_k)\bar{z}_k$
    \State $\tilde{z}^t_k = \argmin_{z\in\cZ} \big\{ \langle \tilde{\nabla}p_1(\underline{z}_k) + \tilde{\nabla}p_2(\underline{z}^t_k) + Q_1(z_k) + Q_2(z_k^t), z \rangle + \frac{\beta_k}{2}\|z - z_k\|_{\bP}^2 + \frac{\eta_k^t}{2}\|z - z_k^t\|_{\bP}^2 \big\}$
    \State $\bar{z}^{t+1}_k = \alpha_t \tilde{z}^t_k + (1 - \alpha_t) \bar{z}^t_k$
    \State $z_k^{t+1}
    =
    \argmin_{z\in\cZ}
    \left\{
    \left\langle
    Q_2(\tilde z_k^t)-Q_2(z_k^t),z
    \right\rangle
    +
    \frac{\eta_k^t}{2}
    \|z-\tilde z_k^t\|_{\bP}^2
    \right\}$
  \EndFor
  \State Set $\tilde{z}_k = \bar{z}^{T_2}_k$
  \State $\bar{z}_{k+1} = \alpha_k \tilde{z}_k + (1 - \alpha_k) \bar{z}_k$
  \State $z_{k+1}
    =
    \argmin_{z\in\cZ}
    \left\{
    \left\langle
    Q_1(\tilde z_k)-Q_1(z_k),z
    \right\rangle
    +
    \frac{\beta_k}{2}
    \|z-\tilde z_k\|_{\bP}^2
    \right\}$
  \State Set $z_{k+1}^0 = z_k^{T_2}$
\EndFor
\State \Return $\bar{z}_{T_1}$
\end{algorithmic}
\end{algorithm}

\begin{algorithm}[H]
\caption{Three-Level Sliding for the Bilinear Saddle-Point Problem}
\label{alg:three_level_bilinear_spp}
\begin{algorithmic}[1]
\State Choose \(z_0=(x_0,y_0)\in\cZ\) and set
\[
  \bar z_0=z_0^0=z_0^{0,0}=z_0 .
\]

\For{\(k=0,\dots,T_1-1\)}
  \State \(\underline z_k=\alpha_k z_k+(1-\alpha_k)\bar z_k\).
  \State Set \(\bar z_k^0=z_k^0\).

  \For{\(t=0,\dots,T_2-1\)}
    \State \(\hat z_k^t=\alpha_t z_k^t+(1-\alpha_t)\bar z_k^t\).
    \State \(\underline z_k^t=\alpha_k\hat z_k^t+(1-\alpha_k)\bar z_k\).
    \State Set \(\bar z_k^{t,0}=z_k^{t,0}\).

    \For{\(r=0,\dots,T_3-1\)}
      \State \(\hat z_k^{t,r}
      =
      \alpha_r z_k^{t,r}
      +
      (1-\alpha_r)\bar z_k^{t,r}\).

      \State \(\underline z_k^{t,r}
      =
      \alpha_k
      \left(
        \alpha_t\hat z_k^{t,r}
        +
        (1-\alpha_t)\bar z_k^t
      \right)
      +
      (1-\alpha_k)\bar z_k\).

      \State Write
      \[
        \underline z_k=(\underline x_k,\underline y_k),
        \qquad
        \underline z_k^t=(\underline x_k^t,\underline y_k^t).
      \]

      \State Define

    \[
  G_k^{t,r}
  :=
  \begin{pmatrix}
    f'(\underline x_k)
    +
    \beta_x \bB^\top
    \bigl(
      \bB \underline x_k^{t,r}
      -
      g'(y_{\mathrm{in}})
    \bigr)
    +
    \bB^\top y_k^{t,r}
    \\[1.2ex]
    g'(\underline y_k^t)
    +
    \beta_y \bB
    \bigl(
      \bB^\top \underline y_k^{t,r}
      +
      f'(x_{\mathrm{in}})
    \bigr)
    -
    \bB x_k^{t,r}
  \end{pmatrix}.
\]
      \State
      \[
      \tilde z_k^{t,r}
      =
      \argmin_{z\in\cZ}
      \left\{
        \langle G_k^{t,r},z\rangle
        +
        \frac{\lambda_k}{2}\|z-z_k\|_{\bP}^2
        +
        \frac{\eta_k^t}{2}\|z-z_k^t\|_{\bP}^2
        +
        \frac{\gamma_k^{t,r}}{2}\|z-z_k^{t,r}\|_{\bP}^2
      \right\}.
      \]

      \State
      \[
        \bar z_k^{t,r+1}
        =
        \alpha_r\tilde z_k^{t,r}
        +
        (1-\alpha_r)\bar z_k^{t,r}.
      \]

      \State
      \[
  z_k^{t,r+1}
  =
  \argmin_{z=(x,y)\in\cZ}
  \left\{
    \left\langle
      \begin{pmatrix}
        \bB^\top(\tilde y_k^{t,r}-y_k^{t,r})\\[0.8ex]
        -\bB(\tilde x_k^{t,r}-x_k^{t,r})
      \end{pmatrix},
      \begin{pmatrix}
        x\\y
      \end{pmatrix}
    \right\rangle
    +
    \frac{\gamma_k^{t,r}}{2}
    \|z-\tilde z_k^{t,r}\|_{\bP}^2
  \right\}.
\]
    \EndFor
    \State Set \(\tilde z_k^t=\bar z_k^{t,T_3}\).
    \State
    \[
      \bar z_k^{t+1}
      =
      \alpha_t\tilde z_k^t
      +
      (1-\alpha_t)\bar z_k^t .
    \]
    \State Set \(z_k^{t+1}=\tilde z_k^t\).
    \State Initialize the next inner loop by setting
    \[
      z_k^{t+1,0}=z_k^{t,T_3}.
    \]
  \EndFor

  \State Set \(\tilde z_k=\bar z_k^{T_2}\).
  \State
  \[
    \bar z_{k+1}
    =
    \alpha_k\tilde z_k
    +
    (1-\alpha_k)\bar z_k .
  \]
  \State Set \(z_{k+1}=\tilde z_k\).
  \State Initialize the next middle and inner loops by setting
  \[
    z_{k+1}^{0}=z_k^{T_2},
    \qquad
    z_{k+1}^{0,0}=z_k^{T_2,0}.
  \]
\EndFor

\State \Return \(\bar z_{T_1}\).
\end{algorithmic}
\end{algorithm}

\section{Applications of the Bilinear Saddle-Point Problem}\label{sec:applications}

The bilinear saddle-point formulation~\eqref{prob:bilinear_SPP_main} arises naturally in a variety of settings. We describe several important examples below.

\textbf{Linearly constrained optimization.}
Consider the problem of minimizing a convex function subject to linear equality constraints:
\begin{equation}\label{eq:app_constrained}
  \min_{x \in \mathcal{X}} \; f(x) \quad \text{subject to} \quad \bA x = b,
\end{equation}
where $\bA \in \R^{m \times d_x}$ and $b \in \R^m$.  Introducing
Lagrange multipliers $y \in \R^m$ yields the bilinear saddle-point problem $\min_{x \in \mathcal{X}} \max_{y \in \R^m} \;
f(x) + \langle y, \bA x \rangle - \langle y, b \rangle$, an instance of~\eqref{prob:bilinear_SPP_main} with $\bB = \bA$,
$g(y) = \langle y, b \rangle$, and $\mathcal{Y} = \R^m$.
Adding a strongly convex regularizer $r(y)$ to the dual
variable yields the augmented Lagrangian form, which also fits the framework.

\textbf{Decentralized optimization over networks.}
In decentralized optimization~\citep{lan2020communication, scaman2018optimal}, $m$ agents connected by a communication network collaboratively solve
\begin{equation}\label{eq:app_decentralized}
  \min_{x \in \mathcal{X}} \; \frac{1}{m}\sum_{i=1}^m f_i(x),
\end{equation}
where each agent $i$ has access only to its local objective $f_i$. Introducing local copies $\bx = \col(x_1, \dots, x_m)$ and a symmetric positive semidefinite \emph{gossip matrix} $\bW \in \R^{m \times m}$ with $\ker(\bW) = \operatorname{span}\{\mathbf{1}_m\}$, the consensus constraint $x_1 = \cdots = x_m$ is equivalent to $(\bW \otimes \bI_{d_x})\bx = 0$.  Its Lagrangian dual is a bilinear saddle-point problem of the form~\eqref{prob:bilinear_SPP_main} with $f(\bx) = \frac{1}{m}\sum_{i=1}^m f_i(x_i)$, $\bB = \bW \otimes \bI_{d_x}$, and $g \equiv 0$.

\textbf{Regularized optimal transport.}
Given discrete probability vectors $a \in \R^{n_1}$ and $b \in \R^{n_2}$,
the regularized optimal transport problem~\citep{dvurechensky2018computational} is
\begin{equation}\label{eq:app_OT}
  \begin{aligned}
  \min_{\bX \in \R^{n_1 \times n_2}}\quad&
  \langle \bC, \bX \rangle + r(\bX)\\
  \text{subject to}\quad&
  \bX \mathbf{1}_{n_2} = a,\\
  &\bX^\top \mathbf{1}_{n_1} = b.
  \end{aligned}
\end{equation}
where $\bC$ is the cost matrix and $r(\cdot)$ is a strongly convex
regularizer (e.g., squared Frobenius norm).  Introducing Lagrange
multipliers $u \in \R^{n_1}$ and $v \in \R^{n_2}$ for the marginal
constraints yields a bilinear saddle-point problem of the
form~\eqref{prob:bilinear_SPP_main} with $f(\bX) = \langle \bC,
\bX \rangle + r(\bX)$, $g(u, v) = \langle u, a \rangle + \langle v,
b \rangle$, and the coupling operator
$\bB \colon \bX \mapsto \col(\bX \mathbf{1}_{n_2},\,
\bX^\top \mathbf{1}_{n_1})$.

\textbf{Inverse problems.}
Recovering a signal $x \in \R^d$ from noisy measurements
$b = \bA x + \xi$ with an $\ell_p$-penalty on a linear transform
$\bD$ takes the form
\begin{equation}\label{eq:app_ellp}
  \min_{x \in \R^d} \; \frac{1}{2}\|\bA x - b\|^2
  + \frac{\lambda}{p}\|\bD x\|_p^p,
\end{equation}
where $\bD \in \R^{m \times d}$, $p \ge 2$, and $\lambda > 0$.
Using the Fenchel conjugate of $\frac{1}{p}\|\cdot\|_p^p$, this
admits a bilinear saddle-point reformulation of the
form~\eqref{prob:bilinear_SPP_main} with
$f(x) = \frac{1}{2}\|\bA x - b\|^2$,
$\bB = \bD$, and a dual regularizer $g$ whose gradient is
$(q-1)$-H\"older continuous with $1/p + 1/q = 1$.

\textbf{Matrix games.}
A two-player zero-sum game with strongly convex strategy regularizers
$f$ and $g$ takes the form
\begin{equation}\label{eq:app_games}
  \min_{x \in \mathcal{X}} \max_{y \in \mathcal{Y}} \;
  f(x) + \langle y, \bA x \rangle - g(y),
\end{equation}
where $\bA \in \R^{d_y \times d_x}$ is the payoff matrix and
$\mathcal{X}, \mathcal{Y}$ are convex strategy sets.  This is an
instance of~\eqref{prob:bilinear_SPP_main} with $\bB = \bA$.

\section{Discussion of Algorithm~\ref{alg:sliding_recursive}}

\subsection{Algorithmic Implementability}\label{app:algo_implementability}
\begin{assumption}\label{ass:prox_setup}
The projection operator
\[
  \Pi_{\cZ}^{\bP}(v)
  :=
  \argmin_{z\in\cZ}
  \frac12\|z-v\|_{\bP}^2
\]
can be computed exactly and at negligible cost compared with evaluations of
the component oracles \(p_i'\) and \(Q_i\). Equivalently, all proximal
subproblems generated by Algorithm~\ref{alg:sliding_recursive} are assumed
to be solved exactly. If this assumption is not satisfied, then the cost of
these proximal subproblems must be added separately to the oracle complexity.
\end{assumption}

Under Assumption~\ref{ass:prox_setup}, the oracle complexities reported below
count only evaluations of the component first-order oracles \(p_i'\) and
\(Q_i\). Each minimization step in Algorithm~\ref{alg:sliding_recursive} is a
proximal step with a linear objective and a quadratic regularizer. In
particular, subproblems of the form
\[
  \begin{aligned}
  &\argmin_{z\in\cZ}
  \left\{
    \langle c,z\rangle
    +
    \sum_{m=1}^q
    \frac{\theta_m}{2}\|z-a_m\|_{\bP}^2
  \right\},\\
  &\qquad
  \theta_m\geq0,\qquad
  \sum_{m=1}^q\theta_m>0.
  \end{aligned}
\]
can be written as
\[
  \Pi_{\cZ}^{\bP}
  \left(
    \frac{\sum_{m=1}^q\theta_m a_m}{\sum_{m=1}^q\theta_m}
    -
    \frac{1}{\sum_{m=1}^q\theta_m}\bP^{-1}c
  \right).
\]
Thus the algorithm is directly implementable whenever projection onto
\(\cZ\) in the \(\bP\)-norm is cheap, for example for common product sets
such as boxes, Euclidean balls, or unconstrained domains. For the bilinear specialization, the proximal subproblems reduce to
projections onto \(\cX\) and \(\cY\).

\subsection{Special Cases and Mirror-Prox Sliding}\label{app:discussion_sliding}

The general complexity bound recovers familiar guarantees in its one- and
two-component special cases. When \(n=1\) and \(Q_1\equiv0\),
Corollary~\ref{cor:VI_sliding_Holder_complexity} gives
\[
  O\left(
    \left(
      \frac{H_1\Omega_{z_{\mathrm{in}}}^{(1+\nu_1)/2}}{\varepsilon}
    \right)^{\frac{2}{1+3\nu_1}}
  \right),
\]
which is the usual universal gradient rate, interpolating between nonsmooth
subgradient and accelerated smooth complexity~\citep{nesterov2015universal}.

When \(\nu_i=1\), the function part of the complexity becomes $O\left(\sqrt{\frac{H_i\Omega}{\varepsilon}}\right)$,
which recovers the accelerated smooth rate. When \(\nu_i=0\), the function
part becomes
$O\left(\frac{H_i^2\Omega}{\varepsilon^2}\right)$,
matching the standard nonsmooth subgradient complexity up to constants. The
operator part contributes the Mirror-Prox-type term
\(O(M_i\Omega/\varepsilon)\).

\paragraph{Relation to Mirror-Prox Sliding.}
The closest two-component comparison is Mirror-Prox Sliding
of~\citet{lan2021mirrorprox}, for VIs with operator \(\nabla G+H\), where
\(G\) is convex with \(L\)-Lipschitz gradient and \(H\) is monotone and
\(M\)-Lipschitz. In our notation, the two component pairs are
\((G,0)\) and \((0,H)\). Mirror-Prox Sliding requires
\(\mathcal O(\sqrt{L\Omega/\varepsilon})\) evaluations of \(\nabla G\) and
\(\mathcal O(\sqrt{L\Omega/\varepsilon}+M\Omega/\varepsilon)\) evaluations of
\(H\). The ascending relabeling of~\citet{borodich2025linear}, used in
Corollary~\ref{cor:VI_sliding_Holder_complexity}, instead gives
\(\mathcal O(\max\{\sqrt{L\Omega/\varepsilon},1\})\) and
\(\mathcal O(\max\{M\Omega/\varepsilon,1\})\) evaluations, respectively,
removing the cross-component square-root term from the \(H\)-oracle count.

Stochastic Mirror-Prox Sliding keeps \(\nabla G\) exact and samples \(H\),
adding \(\mathcal O(\sigma_H^2\Omega/\varepsilon^2)\) to its \(H\)-oracle
complexity. Appendix~\ref{app:stoch_cases} considers a different model:
\(\bB,\bB^\top\) are exact while the function (sub)gradients are stochastic.
For \(q\in\{x,y\}\), its variance costs are charged separately to
\(N_f,N_g\), scaling as \(\sigma_q^2\Omega/\varepsilon^2\) for the
bounded-domain uniform expected gap and as
\(\sigma_q^2/(\mu_q\varepsilon)\) for the expected Lyapunov guarantee when
\(\mu_x,\mu_y>0\). These costs add to the respective deterministic H\"older
budgets. The strongly convex result uses \(p_3=0\), so no anchor residual
appears; the H\"older conditions hold on the feasible sets.

\section{Proof of Theorem~\ref{thm:VI_sliding_inexact}}\label{app:VI_inexact_proof}

\begin{lemma}
\label{lem:alpha_properties}
Let \(\{\alpha_t\}_{t\geq0}\) be the sequence defined in
\eqref{eq:alpha_sequence}. Then \(0<\alpha_t\leq1\) for all \(t\geq0\),
the sequence \(\{\alpha_t\}\) is nonincreasing, and the following properties
hold:
\begin{align}
  \frac{1-\alpha_t}{\alpha_t^2}
  &=
  \frac{1}{\alpha_{t-1}^2},
  \qquad t\geq1,
  \label{eq:alpha_identity} \\
  \frac{1}{t+1}
  \leq
  \alpha_t
  &\leq
  \frac{2}{t+2},
  \qquad t\geq0.
  \label{eq:alpha_t_bounds}
\end{align}
\end{lemma}

\begin{proof}
Set \(a_t:=1/\alpha_t\). By the recursion in
\eqref{eq:alpha_sequence},
\[
  a_{t+1}
  =
  \frac{1+\sqrt{1+4a_t^2}}{2},
\]
and hence
\[
  a_{t+1}^2-a_{t+1}=a_t^2.
\]
Replacing \(t+1\) by \(t\) gives, for \(t\geq1\),
\[
  \frac{1-\alpha_t}{\alpha_t^2}
  =
  \frac{1}{\alpha_{t-1}^2}.
\]

Moreover, since \(a_0=1\) and
\[
  a_{t+1}
  =
  \frac{1+\sqrt{1+4a_t^2}}{2}
  \geq a_t,
\]
the sequence \(\{a_t\}\) is nondecreasing. Hence \(\{\alpha_t\}\) is
nonincreasing, and therefore \(0<\alpha_t\leq\alpha_0=1\).

It remains to prove~\eqref{eq:alpha_t_bounds}. Since
\[
  \sqrt{1+4a_t^2}\geq 2a_t,
\]
we have
\[
  a_{t+1}
  =
  \frac{1+\sqrt{1+4a_t^2}}{2}
  \geq
  a_t+\frac12.
\]
By induction and \(a_0=1\), this gives
\[
  a_t\geq 1+\frac{t}{2}=\frac{t+2}{2}.
\]
Equivalently,
\[
  \alpha_t=\frac1{a_t}\leq \frac{2}{t+2}.
\]

On the other hand, since
\[
  \sqrt{1+4a_t^2}\leq 1+2a_t,
\]
we have
\[
  a_{t+1}
  =
  \frac{1+\sqrt{1+4a_t^2}}{2}
  \leq
  a_t+1.
\]
Again by induction and \(a_0=1\), this gives
\[
  a_t\leq t+1.
\]
Equivalently,
\[
  \alpha_t=\frac1{a_t}\geq \frac{1}{t+1}.
\]
Combining the two inequalities proves~\eqref{eq:alpha_t_bounds}.
\end{proof}

\begin{lemma}
\label{lem:affine_rescaled_oracle}
Suppose that \(0<\alpha_t\leq 1\) for all \(t\), and that all averaging
points generated by Algorithm~\ref{alg:sliding_recursive} belong to
\(\mathcal Z\). For \(k\geq 1\), let
\[
  \bt=(t_1,\dots,t_k),
  \qquad
  0\leq t_\ell\leq T_\ell-1,\quad \ell=1,\dots,k,
\]
and define
\[
  \Gamma_{\bt}:=\prod_{\ell=1}^k \alpha_{t_\ell}.
\]
Then, for every \(1\leq k\leq i\leq n\), the function
\(\hat p_i^{k,\bt}\) is equipped with a first-order
\((\delta_i^{k,\bt},L_i^{k,\bt})\)-oracle on \(\mathcal Z\), with respect to
\(\|\cdot\|_{\bP}\), where
\[
  \delta_i^{k,\bt}
  =
  \frac{\delta_i}{\Gamma_{\bt}},
  \qquad
  L_i^{k,\bt}
  =
  L_i\Gamma_{\bt}.
\]
\end{lemma}

\begin{proof}
We prove the claim by induction on \(k\). Throughout the proof, the oracle
for each transformed function is the oracle induced by the affine
transformation in line~\ref{line:define_hat_p} of
Algorithm~\ref{alg:sliding_recursive}.

\paragraph{Base case \(k=1\).}
The lemma's index condition \(i\geq k\) becomes \(i\geq 1\). By
lines~\ref{line:init_p} and~\ref{line:define_hat_p} of
Algorithm~\ref{alg:sliding_recursive}, with \(k=0\), \(\bt=\varnothing\),
and \(\bt'=(j)\), for all \(1\leq i\leq n\) and
\(0\leq j\leq T_1-1\),
\[
  \begin{aligned}
  \hat p_i^{1,j}(z)
  &=
  \alpha_j^{-1}
  p_i^{0,\varnothing}\bigl(
    \alpha_j z+(1-\alpha_j)\bar z_j^1
  \bigr)\\
  &=
  \alpha_j^{-1}
  p_i\bigl(
    \alpha_j z+(1-\alpha_j)\bar z_j^1
  \bigr).
  \end{aligned}
\]
The induced oracle is therefore
\[
  \begin{aligned}
  \tilde{\hat p}_i^{1,j}(z)
  &=
  \alpha_j^{-1}
  \tilde p_i\bigl(
    \alpha_j z+(1-\alpha_j)\bar z_j^1
  \bigr),\\
  \tilde\nabla \hat p_i^{1,j}(z)
  &=
  \tilde\nabla p_i\bigl(
    \alpha_j z+(1-\alpha_j)\bar z_j^1
  \bigr).
  \end{aligned}
\]
Fix \(u,v\in\mathcal Z\) and define
\[
  u'=\alpha_j u+(1-\alpha_j)\bar z_j^1,
  \qquad
  v'=\alpha_j v+(1-\alpha_j)\bar z_j^1.
\]
By lines~\ref{line:init_bar_z} and~\ref{line:update_bar_z} of
Algorithm~\ref{alg:sliding_recursive}, the averaging point
\(\bar z_j^1\) belongs to \(\mathcal Z\). Since \(\mathcal Z\) is convex,
we have \(u',v'\in\mathcal Z\). Hence,
\begin{align*}
&\hat p_i^{1,j}(u)
-
\left(
  \tilde{\hat p}_i^{1,j}(v)
  +
  \left\langle
    \tilde\nabla\hat p_i^{1,j}(v), u-v
  \right\rangle
\right) \\
&\quad =
\alpha_j^{-1}
\left[
  p_i(u')
  -
  \left(
    \tilde p_i(v')
    +
    \alpha_j
    \left\langle
      \tilde\nabla p_i(v'), u-v
    \right\rangle
  \right)
\right] \\
&\quad =
\alpha_j^{-1}
\left[
  p_i(u')
  -
  \left(
    \tilde p_i(v')
    +
    \left\langle
      \tilde\nabla p_i(v'), u'-v'
    \right\rangle
  \right)
\right].
\end{align*}
Using the \((\delta_i,L_i)\)-oracle bound for \(p_i\), we obtain
\begin{align*}
0
&\leq
\hat p_i^{1,j}(u)
-
\left(
  \tilde{\hat p}_i^{1,j}(v)
  +
  \left\langle
    \tilde\nabla\hat p_i^{1,j}(v), u-v
  \right\rangle
\right) \\
&\leq
\alpha_j^{-1}
\left(
  \frac{L_i}{2}\|u'-v'\|_{\bP}^2+\delta_i
\right) \\
&=
\frac{L_i\alpha_j}{2}\|u-v\|_{\bP}^2
+
\frac{\delta_i}{\alpha_j} \\
&=
\frac{L_i^{1,j}}{2}\|u-v\|_{\bP}^2
+
\delta_i^{1,j}.
\end{align*}
Thus \(\hat p_i^{1,j}\) is equipped with a
\((\delta_i^{1,j},L_i^{1,j})\)-oracle.

\paragraph{Inductive step.}
Assume the claim holds at level \(k\), namely that
\(\hat p_i^{k,\bt}\) is equipped with a
\((\delta_i^{k,\bt},L_i^{k,\bt})\)-oracle for all \(i\geq k\). We prove the
claim at level \(k+1\). Let \(0\leq j\leq T_{k+1}-1\) and set
\[
  \bt'=(\bt,j).
\]
Fix \(u,v\in\mathcal Z\) and define
\[
  u'=\alpha_j u+(1-\alpha_j)\bar z_{\bt'}^{k+1},
  \qquad
  v'=\alpha_j v+(1-\alpha_j)\bar z_{\bt'}^{k+1}.
\]
By lines~\ref{line:init_bar_z} and~\ref{line:update_bar_z} of
Algorithm~\ref{alg:sliding_recursive}, the averaging point
\(\bar z_{\bt'}^{k+1}\) belongs to \(\mathcal Z\). Since
\(\mathcal Z\) is convex, we have \(u',v'\in\mathcal Z\).

Now fix \(i\geq k+1\). By line~\ref{line:define_hat_p} of
Algorithm~\ref{alg:sliding_recursive},
\[
  \hat p_i^{k+1,\bt'}(z)
  =
  \alpha_j^{-1}
  p_i^{k,\bt}\bigl(\alpha_j z+(1-\alpha_j)\bar z_{\bt'}^{k+1}\bigr).
\]
Since \(i\geq k+1\), we have \(i\neq k\). Hence, by
line~\ref{line:define_p_next} of Algorithm~\ref{alg:sliding_recursive}
applied at level \(k\),
\[
  p_i^{k,\bt}(z)=\hat p_i^{k,\bt}(z).
\]
Therefore,
\[
  \hat p_i^{k+1,\bt'}(z)
  =
  \alpha_j^{-1}
  \hat p_i^{k,\bt}\bigl(\alpha_j z+(1-\alpha_j)\bar z_{\bt'}^{k+1}\bigr).
\]
The induced oracle is
\[
  \tilde{\hat p}_i^{k+1,\bt'}(z)
  =
  \alpha_j^{-1}
  \tilde{\hat p}_i^{k,\bt}
  \bigl(\alpha_j z+(1-\alpha_j)\bar z_{\bt'}^{k+1}\bigr),
\]
and
\[
  \tilde\nabla \hat p_i^{k+1,\bt'}(z)
  =
  \tilde\nabla \hat p_i^{k,\bt}
  \bigl(\alpha_j z+(1-\alpha_j)\bar z_{\bt'}^{k+1}\bigr).
\]
Consequently,
\begin{align*}
&\hat p_i^{k+1,\bt'}(u)
-
\tilde{\hat p}_i^{k+1,\bt'}(v)\\
&\quad-
\left\langle
  \tilde\nabla\hat p_i^{k+1,\bt'}(v), u-v
\right\rangle \\
&\quad =
\alpha_j^{-1}
\left[
  \hat p_i^{k,\bt}(u')
  -
  \tilde{\hat p}_i^{k,\bt}(v')
  \right.\\
&\hspace{7em}\left.
  -
  \alpha_j
  \left\langle
    \tilde\nabla\hat p_i^{k,\bt}(v'), u-v
  \right\rangle
\right] \\
&\quad =
\alpha_j^{-1}
\left[
  \hat p_i^{k,\bt}(u')
  -
  \tilde{\hat p}_i^{k,\bt}(v')
  \right.\\
&\hspace{7em}\left.
  -
  \left\langle
    \tilde\nabla\hat p_i^{k,\bt}(v'), u'-v'
  \right\rangle
\right].
\end{align*}
Using the induction hypothesis for \(\hat p_i^{k,\bt}\), we obtain
\begin{align*}
0
&\leq
\hat p_i^{k+1,\bt'}(u)
-
\left(
  \tilde{\hat p}_i^{k+1,\bt'}(v)
  +
  \left\langle
    \tilde\nabla\hat p_i^{k+1,\bt'}(v), u-v
  \right\rangle
\right) \\
&\leq
\alpha_j^{-1}
\left(
  \frac{L_i^{k,\bt}}{2}\|u'-v'\|_{\bP}^2
  +
  \delta_i^{k,\bt}
\right) \\
&=
\frac{L_i^{k,\bt}\alpha_j}{2}\|u-v\|_{\bP}^2
+
\frac{\delta_i^{k,\bt}}{\alpha_j} \\
&=
\frac{L_i^{k+1,\bt'}}{2}\|u-v\|_{\bP}^2
+
\delta_i^{k+1,\bt'}.
\end{align*}
Thus \(\hat p_i^{k+1,\bt'}\) is equipped with a
\((\delta_i^{k+1,\bt'},L_i^{k+1,\bt'})\)-oracle. The induction is complete.
\end{proof}

For fixed $k$, $\mathbf{t} = (t_1, \dots t_k)$ and $0\leq t\leq T_k -1 $, we
define
\begin{align}
  y^{k, \mathbf{t}}_t      & = z_{t_1,\dots,t_{k-1}, t}^k
\end{align}

Fix \(\hat z\in\mathcal Z\). For \(0\leq k\leq n\) and
\(\bt=(t_1,\dots,t_k)\), define
\[
  f_{\bt}^k(z)
  :=
  \sum_{i=1}^n p_i^{k,\bt}(z),
  \qquad
  \hat f_{\bt}^k(z)
  :=
  \sum_{i=1}^n \hat p_i^{k,\bt}(z),
\]
and
\[
  g_{\hat z}^k(z)
  :=
  \sum_{i=k+1}^n \langle Q_i(\hat z), z\rangle,
  \qquad
  h_{\bt,\hat z}^k(z)
  :=
  f_{\bt}^k(z)+g_{\hat z}^k(z).
\]
For \(u,v\in\mathcal Z\), write
\[
  \begin{aligned}
  \Delta f_{\bt}^k(u,v)
  &:=
  f_{\bt}^k(u)-f_{\bt}^k(v),\\
  \Delta \hat f_{\bt}^k(u,v)
  &:=
  \hat f_{\bt}^k(u)-\hat f_{\bt}^k(v),
  \end{aligned}
\]
and
\[
  \begin{aligned}
  \Delta g_{\hat z}^k(u,v)
  &:=
  g_{\hat z}^k(u)-g_{\hat z}^k(v),\\
  \Delta h_{\bt,\hat z}^k(u,v)
  &:=
  h_{\bt,\hat z}^k(u)-h_{\bt,\hat z}^k(v).
  \end{aligned}
\]
Finally, define
\[
  \Phi_{\bt,\hat z}^k(u,v)
  :=
  \Delta h_{\bt,\hat z}^k(u,v)
  +
  \frac{H_\Sigma^k}{2}\|u-v\|_{\bP}^2,
\]
where
\[
  H_\Sigma^k
  :=
  \sum_{\ell=1}^k \eta_{t_1,\dots,t_\ell}^{\ell}.
\]
We use the convention that empty sums are equal to zero.

The constants \(L_i^{k,\bt}\) and \(\delta_i^{k,\bt}\) are those from
Lemma~\ref{lem:affine_rescaled_oracle}. We define the transformed Lipschitz constants for the
operators by
\[
  M_i^{k,\bt}
  :=
  M_i
  \prod_{\ell=1}^k
  \frac{\alpha_{t_\ell}}{\alpha_{T_\ell-1}}.
\]

\begin{lemma}\label{lem:recursive_gap_bound}
Fix \(\hat z\in\mathcal Z\). For every \(0\leq k\leq n\) and
\(\bt=(t_1,\dots,t_k)\), where for \(k\geq 1\)
\[
  0\leq t_\ell\leq T_\ell-1,
  \qquad \ell=1,\dots,k,
\]
and where \(\bt=\varnothing\) when \(k=0\), we have
\[
  \Phi_{\bt,\hat z}^k(\tilde z_{\bt}^k,\hat z)
  \leq
  S_{\bt}^k+R_{\bt}^k.
\]
Here
\[
  A_i^{k,\bt}
  :=
  L_i^{k,\bt}
  \prod_{j=k+1}^i\alpha_{T_j-1}^2
  +
  M_i^{k,\bt}
  \prod_{j=k+1}^i\alpha_{T_j-1},
\]
and
\[
  S_{\bt}^k
  :=
  \sum_{i=k+1}^n
  \frac{A_i^{k,\bt}}{2}
  \Bigl(
    \|z_{\bt^{i,0}}^i-\hat z\|_{\bP}^2
    -
    \|z_{\bt_+^{i,0}}^i-\hat z\|_{\bP}^2
  \Bigr),
\]
and
\[
  R_{\bt}^k
  :=
  \sum_{i=k+1}^n
  \frac{
    \delta_i^{k,\bt}
  }{
    \prod_{j=k+1}^i \alpha_{T_j-1}
  }.
\]
We use the conventions
\[
  L_i^{0,\varnothing}=L_i,
  \qquad
  M_i^{0,\varnothing}=M_i,
  \qquad
  \delta_i^{0,\varnothing}=\delta_i,
\]
and empty products are equal to \(1\).

For \(i\geq k+1\), the multi-indices in \(S_{\bt}^k\) are defined by
\[
  \mathbf 0_r:=(0,\dots,0)\in\mathbb N^r,
\]
\[
  \bt^{i,0}
  :=
  (\bt,\mathbf 0_{i-k})
  \in\mathbb N^i.
\]
For \(k\geq1\), set
\[
  \bt_+:=(t_1,\dots,t_{k-1},t_k+1).
\]
Then
\[
  \bt_+^{i,0}\in\mathbb N^i,
  \qquad
  \bt_+^{i,0}
  :=
  \begin{cases}
    (T_1,\mathbf 0_{i-1}),
    & k=0, \\[2mm]
    (\bt_+,\mathbf 0_{i-k}),
    & 1\leq k\leq n-1.
  \end{cases}
\]
When \(k=n\), the sums defining \(S_{\bt}^n\) and \(R_{\bt}^n\) are empty,
so \(S_{\bt}^n=R_{\bt}^n=0\).
\end{lemma}


\begin{proof}
We prove the claim by backward induction on \(k\).

\paragraph{Base case \(k=n\).}
Let \(\bt=(t_1,\dots,t_n)\), where
\(0\leq t_\ell\leq T_\ell-1\) for \(\ell=1,\dots,n\).
Since the sum defining \(g_{\hat z}^n\) is empty, we have
\[
  g_{\hat z}^n(z)\equiv 0,
  \qquad
  h_{\bt,\hat z}^n(z)=f_{\bt}^n(z).
\]
Moreover, by line~\ref{line:base_case_return} of
Algorithm~\ref{alg:sliding_recursive},
\[
  \tilde z_{\bt}^n
  =
  \argmin_{z\in\mathcal Z} f_{\bt}^n(z).
\]
Since \(f_{\bt}^n\) is \(H_\Sigma^n\)-strongly convex with respect to
\(\|\cdot\|_{\bP}\), it follows that, for every \(\hat z\in\mathcal Z\),
\[
  f_{\bt}^n(\hat z)
  \geq
  f_{\bt}^n(\tilde z_{\bt}^n)
  +
  \frac{H_\Sigma^n}{2}
  \|\tilde z_{\bt}^n-\hat z\|_{\bP}^2.
\]
Therefore,
\[
  \Phi_{\bt,\hat z}^n(\tilde z_{\bt}^n,\hat z)
  =
  f_{\bt}^n(\tilde z_{\bt}^n)
  -
  f_{\bt}^n(\hat z)
  +
  \frac{H_\Sigma^n}{2}
  \|\tilde z_{\bt}^n-\hat z\|_{\bP}^2
  \leq 0.
\]
On the other hand, when \(k=n\), both sums defining \(S_{\bt}^n\) and
\(R_{\bt}^n\) are empty, so
\[
  S_{\bt}^n=0,
  \qquad
  R_{\bt}^n=0.
\]
Hence
\[
  \Phi_{\bt,\hat z}^n(\tilde z_{\bt}^n,\hat z)
  \leq
  S_{\bt}^n+R_{\bt}^n.
\]



\paragraph{Inductive step.}
Fix \(1\leq k\leq n\). Assume that the claim holds at level \(k\), i.e.,
for every \(\bt=(t_1,\dots,t_k)\), where
\(0\leq t_\ell\leq T_\ell-1\) for \(\ell=1,\dots,k\),
\[
  \Phi_{\bt,\hat z}^k(\tilde z_{\bt}^k,\hat z)
  \leq
  S_{\bt}^k+R_{\bt}^k.
\]
Let
\[
  \bt'=(t_1,\dots,t_{k-1}),
\]
with the convention that \(\bt'=\varnothing\) when \(k=1\). We prove
\[
  \Phi_{\bt',\hat z}^{k-1}
  (\tilde z_{\bt'}^{k-1},\hat z)
  \leq
  S_{\bt'}^{k-1}+R_{\bt'}^{k-1}.
\]

For any \(\bt=(\bt',t_k)\), \(0\leq t_{k}\leq T_k-1\), we have
\begin{align*}
  \Phi_{\bt,\hat z}^k(\tilde z_{\bt}^k,\hat z)
  &=
  \Delta h_{\bt,\hat z}^k(\tilde z_{\bt}^k,\hat z)
  +
  \frac{H_\Sigma^k}{2}
  \|\tilde z_{\bt}^k-\hat z\|_{\bP}^2.
\end{align*}

For any \(z\in\mathcal Z\), by line~\ref{line:define_p_next} of
Algorithm~\ref{alg:sliding_recursive}, the only component in which
\(f_{\bt}^k\) and \(\hat f_{\bt}^k\) differ is the \(k\)-th component.
Hence
\begin{align*}
  f_{\bt}^k(z)-\hat f_{\bt}^k(z)
  &=
  p_k^{k,\bt}(z)-\hat p_k^{k,\bt}(z),
\end{align*}
and consequently
\begin{align*}
  \Delta f_{\bt}^k(\tilde z_{\bt}^k,\hat z)
  &=
  \Delta \hat f_{\bt}^k(\tilde z_{\bt}^k,\hat z)
  +
  \Delta p_k^{k,\bt}(\tilde z_{\bt}^k,\hat z)
  -
  \Delta \hat p_k^{k,\bt}(\tilde z_{\bt}^k,\hat z)
  \\
  &=
  \Delta \hat f_{\bt}^k(\tilde z_{\bt}^k,\hat z)
  +
  \left[
    p_k^{k,\bt}(\tilde z_{\bt}^k)
    -
    \hat p_k^{k,\bt}(\tilde z_{\bt}^k)
  \right]
  -
  \left[
    p_k^{k,\bt}(\hat z)
    -
    \hat p_k^{k,\bt}(\hat z)
  \right].
\end{align*}


  For any function \(f\) equipped with a first-order \((\delta,L)\)-oracle,
we have, for all \(x,y,z\in\mathcal Z\),
\begin{equation}
  f(z)-f(y)
  \geq
  \left\langle \tilde\nabla f(x), z-y \right\rangle
  -
  \frac{L}{2}\|y-x\|_{\bP}^2
  -
  \delta .
  \label{eq:oracle_lower_three_point}
\end{equation}

  By Lemma~\ref{lem:affine_rescaled_oracle},
\(\hat p_k^{k,\bt}\) is equipped with a first-order
\((\delta_k^{k,\bt},L_k^{k,\bt})\)-oracle. Therefore,
using line~\ref{line:define_p_next} of
Algorithm~\ref{alg:sliding_recursive}, we obtain
\begin{align*}
  \Delta f_{\bt}^k(\tilde z_{\bt}^k,\hat z)
  &=
  \Delta \hat f_{\bt}^k(\tilde z_{\bt}^k,\hat z)
  +
  \left[
    p_k^{k,\bt}(\tilde z_{\bt}^k)-p_k^{k,\bt}(\hat z)
  \right]
  -
  \Delta \hat p_k^{k,\bt}(\tilde z_{\bt}^k,\hat z)
  \\
  &=
  \Delta \hat f_{\bt}^k(\tilde z_{\bt}^k,\hat z)
  +
  \frac{\eta_{\bt}^k}{2}
  \left(
    \|\tilde z_{\bt}^k-z_{\bt}^k\|_{\bP}^2
    -
    \|\hat z-z_{\bt}^k\|_{\bP}^2
  \right)
  \\
  &\quad
  +
  \left\langle
    \tilde\nabla \hat p_k^{k,\bt}(z_{\bt}^k)
    +
    Q_k(z_{\bt}^k),
    \tilde z_{\bt}^k-\hat z
  \right\rangle
  -
  \Delta \hat p_k^{k,\bt}(\tilde z_{\bt}^k,\hat z)
  \\
  &\geq
  \Delta \hat f_{\bt}^k(\tilde z_{\bt}^k,\hat z)
  +
  \frac{\eta_{\bt}^k-L_k^{k,\bt}}{2}
  \|\tilde z_{\bt}^k-z_{\bt}^k\|_{\bP}^2
  \\
  &\quad
  -
  \frac{\eta_{\bt}^k}{2}
  \|\hat z-z_{\bt}^k\|_{\bP}^2
  +
  \left\langle
    Q_k(z_{\bt}^k),
    \tilde z_{\bt}^k-\hat z
  \right\rangle
  -
  \delta_k^{k,\bt}.
\end{align*}
Define
\[
  \Delta_{Q_k}^{k,\bt}
  :=
  Q_k(z_{\bt}^k)-Q_k(\tilde z_{\bt}^k).
\]
By monotonicity of \(Q_k\),
\begin{align*}
  \left\langle
    Q_k(z_{\bt}^k),
    \tilde z_{\bt}^k-\hat z
  \right\rangle
  &=
  \left\langle
    \Delta_{Q_k}^{k,\bt}
    +
    Q_k(\tilde z_{\bt}^k),
    \tilde z_{\bt}^k-\hat z
  \right\rangle \\
  &\geq
  \left\langle
    \Delta_{Q_k}^{k,\bt}
    +
    Q_k(\hat z),
    \tilde z_{\bt}^k-\hat z
  \right\rangle .
\end{align*}
Moreover, since
\[
  g_{\hat z}^{k-1}(z)
  =
  g_{\hat z}^{k}(z)
  +
  \langle Q_k(\hat z), z\rangle,
\]
we have
\[
  \Delta g_{\hat z}^{k-1}(\tilde z_{\bt}^k,\hat z)
  =
  \Delta g_{\hat z}^{k}(\tilde z_{\bt}^k,\hat z)
  +
  \left\langle
    Q_k(\hat z),
    \tilde z_{\bt}^k-\hat z
  \right\rangle .
\]

Therefore, using \(h_{\bt,\hat z}^k=f_{\bt}^k+g_{\hat z}^k\) and the previous
lower bound on \(\Delta f_{\bt}^k(\tilde z_{\bt}^k,\hat z)\), we obtain
\begin{align*}
  \Delta h_{\bt,\hat z}^k(\tilde z_{\bt}^k,\hat z)
  &\geq
  \Delta\hat f_{\bt}^k(\tilde z_{\bt}^k,\hat z)
  +
  \Delta g_{\hat z}^{k-1}(\tilde z_{\bt}^k,\hat z) \\
  &\quad
  +
  \frac{\eta_{\bt}^k-L_k^{k,\bt}}{2}
  \|\tilde z_{\bt}^k-z_{\bt}^k\|_{\bP}^2
  -
  \frac{\eta_{\bt}^k}{2}
  \|\hat z-z_{\bt}^k\|_{\bP}^2 \\
  &\quad
  +
  \left\langle
    \Delta_{Q_k}^{k,\bt},
    \tilde z_{\bt}^k-\hat z
  \right\rangle
  -
  \delta_k^{k,\bt}.
\end{align*}
Let
\[
  \hat\bt=(t_1,\dots,t_{k-1},t_k+1).
\]
From line~\ref{line:update_z} of Algorithm~\ref{alg:sliding_recursive},
we have
\[
  z_{\hat\bt}^k
  =
  \argmin_{z\in\mathcal Z}
  \Bigl\{
    \left\langle
      Q_k(\tilde z_\bt^k)-Q_k(z_\bt^k),z
    \right\rangle
    +
    \frac{\eta_\bt^k}{2}
    \|z-\tilde z_\bt^k\|_{\bP}^2
  \Bigr\}.
\]
Since
\[
  \Delta_{Q_k}^{k,\bt}
  :=
  Q_k(z_\bt^k)-Q_k(\tilde z_\bt^k),
\]
the optimality of \(z_{\hat\bt}^k\) gives, for every \(y\in\mathcal Z\),
\begin{align*}
  \left\langle
    -\Delta_{Q_k}^{k,\bt},
    z_{\hat\bt}^k-y
  \right\rangle
  &\leq
  \frac{\eta_\bt^k}{2}
  \Bigl(
    \|\tilde z_\bt^k-y\|_{\bP}^2
    -
    \|z_{\hat\bt}^k-y\|_{\bP}^2
    -
    \|z_{\hat\bt}^k-\tilde z_\bt^k\|_{\bP}^2
  \Bigr).
\end{align*}
Equivalently,
\begin{align*}
  \left\langle
    \Delta_{Q_k}^{k,\bt},
    z_{\hat\bt}^k-y
  \right\rangle
  &\geq
  \frac{\eta_\bt^k}{2}
  \Bigl(
    \|z_{\hat\bt}^k-y\|_{\bP}^2
    +
    \|z_{\hat\bt}^k-\tilde z_\bt^k\|_{\bP}^2
    -
    \|\tilde z_\bt^k-y\|_{\bP}^2
  \Bigr).
\end{align*}
Taking \(y=\hat z\), we have
\begin{align*}
  \left\langle
    \Delta_{Q_k}^{k,\bt},
    \tilde z_\bt^k-\hat z
  \right\rangle
  &=
  \left\langle
    \Delta_{Q_k}^{k,\bt},
    z_{\hat\bt}^k-\hat z
  \right\rangle
  +
  \left\langle
    \Delta_{Q_k}^{k,\bt},
    \tilde z_\bt^k-z_{\hat\bt}^k
  \right\rangle                                    \\
  &\geq
  \frac{\eta_\bt^k}{2}
  \left(
    \|z_{\hat\bt}^k-\hat z\|_{\bP}^2
    -
    \|\tilde z_\bt^k-\hat z\|_{\bP}^2
  \right)\\
  &\quad+
  \frac{\eta_\bt^k}{2}
  \|z_{\hat\bt}^k-\tilde z_\bt^k\|_{\bP}^2
  +
  \left\langle
    \Delta_{Q_k}^{k,\bt},
    \tilde z_\bt^k-z_{\hat\bt}^k
  \right\rangle .
\end{align*}
By Young's inequality,
\begin{align*}
  \left\langle
    \Delta_{Q_k}^{k,\bt},
    \tilde z_\bt^k-z_{\hat\bt}^k
  \right\rangle
  &\geq
  -
  \frac{1}{2\eta_\bt^k}
  \|\Delta_{Q_k}^{k,\bt}\|_{\bP^{-1}}^2
  -
  \frac{\eta_\bt^k}{2}
  \|\tilde z_\bt^k-z_{\hat\bt}^k\|_{\bP}^2.
\end{align*}
Therefore,
\begin{align*}
  \left\langle
    \Delta_{Q_k}^{k,\bt},
    \tilde z_\bt^k-\hat z
  \right\rangle
  &\geq
  \frac{\eta_\bt^k}{2}
  \left(
    \|z_{\hat\bt}^k-\hat z\|_{\bP}^2
    -
    \|\tilde z_\bt^k-\hat z\|_{\bP}^2
  \right)
  -
  \frac{1}{2\eta_\bt^k}
  \|\Delta_{Q_k}^{k,\bt}\|_{\bP^{-1}}^2                         \\
  &\geq
  \frac{\eta_\bt^k}{2}
  \left(
    \|z_{\hat\bt}^k-\hat z\|_{\bP}^2
    -
    \|\tilde z_\bt^k-\hat z\|_{\bP}^2
  \right)
  -
  \frac{(M_k^{k,\bt})^2}{2\eta_\bt^k}
  \|\tilde z_\bt^k-z_\bt^k\|_{\bP}^2 .
\end{align*}
where the last inequality follows from Assumption~\ref{ass:Monotone_and_Lipschitz_Operators_VI} and the definition
\[
  M_k^{k,\bt}
  =
  M_k\prod_{\ell=1}^k
  \frac{\alpha_{t_\ell}}{\alpha_{T_\ell-1}}.
\]
Indeed, since \(0\leq t_\ell\leq T_\ell-1\) and \(\{\alpha_t\}\) is
positive and nonincreasing, \(M_k^{k,\bt}\geq M_k\), and hence
\[
  \|\Delta_{Q_k}^{k,\bt}\|_{\bP^{-1}}
  \leq
  M_k\|\tilde z_\bt^k-z_\bt^k\|_{\bP}
  \leq
  M_k^{k,\bt}\|\tilde z_\bt^k-z_\bt^k\|_{\bP}.
\]
  Therefore,
\begin{align}
  \Delta h_{\bt,\hat z}^k(\tilde z_\bt^k,\hat z)
  &\geq
  \Delta \hat f_\bt^k(\tilde z_\bt^k,\hat z)
  +
  \Delta g_{\hat z}^{k-1}(\tilde z_\bt^k,\hat z) \notag\\
  &\quad
  +
  \frac{1}{2}
  \left(
    \eta_\bt^k
    -
    L_k^{k,\bt}
    -
    \frac{(M_k^{k,\bt})^2}{\eta_\bt^k}
  \right)
  \|\tilde z_\bt^k-z_\bt^k\|_{\bP}^2 \notag\\
  &\quad
  +
  \frac{\eta_\bt^k}{2}
  \left(
    \|z_{\hat\bt}^k-\hat z\|_{\bP}^2
    -
    \|\tilde z_\bt^k-\hat z\|_{\bP}^2
    -
    \|z_\bt^k-\hat z\|_{\bP}^2
  \right)
  -
  \delta_k^{k,\bt},
  \label{eq:delta_h_lower_true}
\end{align}
Since \(\eta_\bt^k=L_k^{k,\bt}+M_k^{k,\bt}\), we have
\[
  \eta_\bt^k
  -
  L_k^{k,\bt}
  -
  \frac{(M_k^{k,\bt})^2}{\eta_\bt^k}
  =
  M_k^{k,\bt}
  -
  \frac{(M_k^{k,\bt})^2}{L_k^{k,\bt}+M_k^{k,\bt}}
  \geq 0.
\]
Dropping the corresponding nonnegative term in
\eqref{eq:delta_h_lower_true}, we obtain
\begin{align}
  \Delta h_{\bt,\hat z}^k(\tilde z_\bt^k,\hat z)
  &\geq
  \Delta \hat f_\bt^k(\tilde z_\bt^k,\hat z)
  +
  \Delta g_{\hat z}^{k-1}(\tilde z_\bt^k,\hat z) \notag\\
  &\quad
  +
  \frac{\eta_\bt^k}{2}
  \left(
    \|z_{\hat\bt}^k-\hat z\|_{\bP}^2
    -
    \|\tilde z_\bt^k-\hat z\|_{\bP}^2
    -
    \|z_\bt^k-\hat z\|_{\bP}^2
  \right)
  -
  \delta_k^{k,\bt}.
  \label{eq:delta_h_lower_simplified}
\end{align}
  Using~\eqref{eq:delta_h_lower_simplified} and
\(H_\Sigma^k=H_\Sigma^{k-1}+\eta_\bt^k\), we obtain
\begin{align}
  \Phi_{\bt,\hat z}^k(\tilde z_\bt^k,\hat z)
  &=
  \Delta h_{\bt,\hat z}^k(\tilde z_\bt^k,\hat z)
  +
  \frac{H_\Sigma^k}{2}
  \|\tilde z_\bt^k-\hat z\|_{\bP}^2
  \notag\\
  &\geq
  \Delta \hat f_\bt^k(\tilde z_\bt^k,\hat z)
  +
  \Delta g_{\hat z}^{k-1}(\tilde z_\bt^k,\hat z)
  +
  \frac{H_\Sigma^{k-1}}{2}
  \|\tilde z_\bt^k-\hat z\|_{\bP}^2
  \notag\\
  &\quad
  +
  \frac{\eta_\bt^k}{2}
  \left(
    \|z_{\hat\bt}^k-\hat z\|_{\bP}^2
    -
    \|z_\bt^k-\hat z\|_{\bP}^2
  \right)
  -
  \delta_k^{k,\bt}.
  \label{eq:phi_lower_level_k}
\end{align}
Since
\[
  \Phi_{\bt,\hat z}^k(\tilde z_\bt^k,\hat z)
  \leq
  S_\bt^k+R_\bt^k,
\]
it follows from~\eqref{eq:phi_lower_level_k} that
\begin{align}
  S_\bt^k+R_\bt^k
  &\geq
  \Delta \hat f_\bt^k(\tilde z_\bt^k,\hat z)
  +
  \Delta g_{\hat z}^{k-1}(\tilde z_\bt^k,\hat z)
  +
  \frac{H_\Sigma^{k-1}}{2}
  \|\tilde z_\bt^k-\hat z\|_{\bP}^2
  \notag\\
  &\quad
  +
  \frac{\eta_\bt^k}{2}
  \left(
    \|z_{\hat\bt}^k-\hat z\|_{\bP}^2
    -
    \|z_\bt^k-\hat z\|_{\bP}^2
  \right)
  -
  \delta_k^{k,\bt}.
  \label{eq:induction_consequence_level_k}
\end{align}

Let $\alpha=\alpha_{t_k}$. By line~\ref{line:define_hat_p} of
Algorithm~\ref{alg:sliding_recursive}, we have
\begin{align*}
  \Delta \hat f_\bt^k(\tilde z_\bt^k,\hat z)
  &=
  \sum_{i=1}^n
  \left(
    \hat p_i^{k,\bt}(\tilde z_\bt^k)
    -
    \hat p_i^{k,\bt}(\hat z)
  \right) \\
  &=
  \sum_{i=1}^{k-1}
  \left(
    p_i^{k-1,\bt'}(\tilde z_\bt^k)
    -
    p_i^{k-1,\bt'}(\hat z)
  \right) \\
  &\quad
  +
  \frac{1}{\alpha}
  \sum_{i=k}^{n}
    p_i^{k-1,\bt'}
    \bigl(\alpha\tilde z_\bt^k+(1-\alpha)\bar z_\bt^k\bigr)
  \\
  &\quad
  -
  \frac{1}{\alpha}
  \sum_{i=k}^{n}
    p_i^{k-1,\bt'}
    \bigl(\alpha\hat z+(1-\alpha)\bar z_\bt^k\bigr)
  .
\end{align*}
By line~\ref{line:update_bar_z} of
Algorithm~\ref{alg:sliding_recursive},
\[
  \bar z_{\hat\bt}^k
  =
  \alpha\tilde z_\bt^k+(1-\alpha)\bar z_\bt^k.
\]
Therefore,
\begin{align*}
  \Delta \hat f_\bt^k(\tilde z_\bt^k,\hat z)
  &=
  \sum_{i=1}^{k-1}
  \left(
    p_i^{k-1,\bt'}(\tilde z_\bt^k)
    -
    p_i^{k-1,\bt'}(\hat z)
  \right) \\
  &\quad
  +
  \frac{1}{\alpha}
  \sum_{i=k}^{n}
    p_i^{k-1,\bt'}(\bar z_{\hat\bt}^k)
  \\
  &\quad
  -
  \frac{1}{\alpha}
  \sum_{i=k}^{n}
    p_i^{k-1,\bt'}
    \bigl(\alpha\hat z+(1-\alpha)\bar z_\bt^k\bigr)
  .
\end{align*}
By convexity of \(p_i^{k-1,\bt'}\), for \(i\geq k\),
\[
  p_i^{k-1,\bt'}
  \bigl(\alpha\hat z+(1-\alpha)\bar z_\bt^k\bigr)
  \leq
  \alpha p_i^{k-1,\bt'}(\hat z)
  +(1-\alpha)p_i^{k-1,\bt'}(\bar z_\bt^k),
\]
and, for \(i\leq k-1\),
\[
  p_i^{k-1,\bt'}(\tilde z_\bt^k)
  \geq
  \frac{1}{\alpha}p_i^{k-1,\bt'}(\bar z_{\hat\bt}^k)
  -
  \frac{1-\alpha}{\alpha}p_i^{k-1,\bt'}(\bar z_\bt^k).
\]
Consequently,
\begin{align}
  \Delta \hat f_\bt^k(\tilde z_\bt^k,\hat z)
  &\geq
  \frac{1}{\alpha}
  \sum_{i=1}^{n}
  p_i^{k-1,\bt'}(\bar z_{\hat\bt}^k)
  -
  \frac{1-\alpha}{\alpha}
  \sum_{i=1}^{n}
  p_i^{k-1,\bt'}(\bar z_\bt^k)
  \notag\\
  &\quad-
  \sum_{i=1}^{n}
  p_i^{k-1,\bt'}(\hat z) \notag\\
  &=
  \frac{1}{\alpha}
  \Delta f_{\bt'}^{k-1}(\bar z_{\hat\bt}^k,\hat z)
  -
  \frac{1-\alpha}{\alpha}
  \Delta f_{\bt'}^{k-1}(\bar z_\bt^k,\hat z).
  \label{eq:delta_hat_f_lower}
\end{align}


  Moreover, by line~\ref{line:update_bar_z} of
Algorithm~\ref{alg:sliding_recursive}, with
\(\alpha=\alpha_{t_k}\), we have
\[
  \bar z_{\hat\bt}^k
  =
  \alpha \tilde z_\bt^k
  +
  (1-\alpha)\bar z_\bt^k.
\]
Equivalently,
\[
  \tilde z_\bt^k-\hat z
  =
  \frac{1}{\alpha}
  \bigl(\bar z_{\hat\bt}^k-\hat z\bigr)
  -
  \frac{1-\alpha}{\alpha}
  \bigl(\bar z_\bt^k-\hat z\bigr).
\]
Therefore,
\begin{align*}
  \Delta g_{\hat z}^{k-1}(\tilde z_\bt^k,\hat z)
  &=
  \sum_{i=k}^n
  \left\langle
    Q_i(\hat z),\tilde z_\bt^k-\hat z
  \right\rangle \\
  &=
  \frac{1}{\alpha}
  \sum_{i=k}^n
  \left\langle
    Q_i(\hat z),\bar z_{\hat\bt}^k-\hat z
  \right\rangle
  -
  \frac{1-\alpha}{\alpha}
  \sum_{i=k}^n
  \left\langle
    Q_i(\hat z),\bar z_\bt^k-\hat z
  \right\rangle \\
  &=
  \frac{1}{\alpha}
  \Delta g_{\hat z}^{k-1}(\bar z_{\hat\bt}^k,\hat z)
  -
  \frac{1-\alpha}{\alpha}
  \Delta g_{\hat z}^{k-1}(\bar z_\bt^k,\hat z).
\end{align*}
By convexity of the squared \(\bP\)-norm,
\[
  \|\bar z_{\hat\bt}^k-\hat z\|_{\bP}^2
  \leq
  \alpha\|\tilde z_\bt^k-\hat z\|_{\bP}^2
  +
  (1-\alpha)\|\bar z_\bt^k-\hat z\|_{\bP}^2.
\]
Rearranging gives
\[
  \frac{H_\Sigma^{k-1}}{2}
  \|\tilde z_\bt^k-\hat z\|_{\bP}^2
  \geq
  \frac{H_\Sigma^{k-1}}{2\alpha}
  \|\bar z_{\hat\bt}^k-\hat z\|_{\bP}^2
  -
  \frac{(1-\alpha)H_\Sigma^{k-1}}{2\alpha}
  \|\bar z_\bt^k-\hat z\|_{\bP}^2.
\]
Consequently,
\begin{align*}
  &\Delta g_{\hat z}^{k-1}(\tilde z_\bt^k,\hat z)
  +
  \frac{H_\Sigma^{k-1}}{2}
  \|\tilde z_\bt^k-\hat z\|_{\bP}^2 \\
  &\quad\geq
  \frac{1}{\alpha}
  \left[
    \Delta g_{\hat z}^{k-1}(\bar z_{\hat\bt}^k,\hat z)
    +
    \frac{H_\Sigma^{k-1}}{2}
    \|\bar z_{\hat\bt}^k-\hat z\|_{\bP}^2
  \right] \\
  &\qquad
  -
  \frac{1-\alpha}{\alpha}
  \left[
    \Delta g_{\hat z}^{k-1}(\bar z_\bt^k,\hat z)
    +
    \frac{H_\Sigma^{k-1}}{2}
    \|\bar z_\bt^k-\hat z\|_{\bP}^2
  \right].
\end{align*}
It follows from~\eqref{eq:induction_consequence_level_k},
\eqref{eq:delta_hat_f_lower}, and the preceding bound on
\(\Delta g_{\hat z}^{k-1}\) that
\begin{align}
  S_\bt^k+R_\bt^k
  &\geq
  \frac{1}{\alpha_{t_k}}
  \Phi_{\bt',\hat z}^{k-1}(\bar z_{\hat\bt}^k,\hat z)
  -
  \frac{1-\alpha_{t_k}}{\alpha_{t_k}}
  \Phi_{\bt',\hat z}^{k-1}(\bar z_{\bt}^k,\hat z) \notag\\
  &\quad
  +
  \frac{\eta_\bt^k}{2}
  \left(
    \|z_{\hat\bt}^k-\hat z\|_{\bP}^2
    -
    \|z_{\bt}^k-\hat z\|_{\bP}^2
  \right)
  -
  \delta_k^{k,\bt}.
  \label{eq:pre_telescoping_level_k}
\end{align}
Dividing both sides by \(\alpha_{t_k}\), we obtain
\begin{align}
  \frac{S_\bt^k+R_\bt^k}{\alpha_{t_k}}
  &\geq
  \frac{1}{\alpha_{t_k}^2}
  \Phi_{\bt',\hat z}^{k-1}(\bar z_{\hat\bt}^k,\hat z)
  -
  \frac{1-\alpha_{t_k}}{\alpha_{t_k}^2}
  \Phi_{\bt',\hat z}^{k-1}(\bar z_{\bt}^k,\hat z) \notag\\
  &\quad
  +
  \frac{\eta_\bt^k}{2\alpha_{t_k}}
  \left(
    \|z_{\hat\bt}^k-\hat z\|_{\bP}^2
    -
    \|z_{\bt}^k-\hat z\|_{\bP}^2
  \right)
  -
  \frac{\delta_k^{k,\bt}}{\alpha_{t_k}}.
  \label{eq:pre_telescoping_level_k_divided}
\end{align}
Using
\[
  \frac{1-\alpha_{t_k}}{\alpha_{t_k}^2}
  =
  \frac{1}{\alpha_{t_k-1}^2},
  \qquad t_k\geq 1,
\]
with the convention \(1/\alpha_{-1}^2:=0\), and using
\(\delta_k^{k,\bt}=\delta_k^{k-1,\bt'}/\alpha_{t_k}\), we get
\begin{align}
  \frac{S_\bt^k+R_\bt^k}{\alpha_{t_k}}
  &\geq
  \frac{1}{\alpha_{t_k}^2}
  \Phi_{\bt',\hat z}^{k-1}(\bar z_{\hat\bt}^k,\hat z)
  -
  \frac{1}{\alpha_{t_k-1}^2}
  \Phi_{\bt',\hat z}^{k-1}(\bar z_{\bt}^k,\hat z) \notag\\
  &\quad
  +
  \frac{\eta_\bt^k}{2\alpha_{t_k}}
  \left(
    \|z_{\hat\bt}^k-\hat z\|_{\bP}^2
    -
    \|z_{\bt}^k-\hat z\|_{\bP}^2
  \right)
  -
  \frac{\delta_k^{k-1,\bt'}}{\alpha_{t_k}^2}.
  \label{eq:telescoping_step_level_k}
\end{align}
  Summing~\eqref{eq:telescoping_step_level_k} over
\(t_k=0,\dots,T_k-1\), we obtain
\begin{align*}
  \sum_{t_k=0}^{T_k-1}
  \frac{S_\bt^k+R_\bt^k}{\alpha_{t_k}}
  &\geq
  \frac{1}{\alpha_{T_k-1}^2}
  \Phi_{\bt',\hat z}^{k-1}
  \bigl(\bar z_{\bt',T_k}^k,\hat z\bigr) \\
  &\quad
  +
  \sum_{t_k=0}^{T_k-1}
  \frac{\eta_\bt^k}{2\alpha_{t_k}}
  \left(
    \|z_{\hat\bt}^k-\hat z\|_{\bP}^2
    -
    \|z_\bt^k-\hat z\|_{\bP}^2
  \right) \\
  &\quad
  -
  \delta_k^{k-1,\bt'}
  \sum_{t_k=0}^{T_k-1}
  \frac{1}{\alpha_{t_k}^2}.
\end{align*}
Moreover, by the definitions of \(L_k^{k,\bt}\), \(M_k^{k,\bt}\), and
\(\eta_\bt^k\),
\[
  \frac{\eta_\bt^k}{\alpha_{t_k}}
  =
  L_k^{k-1,\bt'}
  +
  \frac{M_k^{k-1,\bt'}}{\alpha_{T_k-1}},
\]
which is independent of \(t_k\). For brevity, set
\[
  \begin{aligned}
  D_k^{\bt'}
  &:=
  L_k^{k-1,\bt'}
  +\frac{M_k^{k-1,\bt'}}{\alpha_{T_k-1}},\\
  C_k^{\bt'}
  &:=
  \alpha_{T_k-1}^2D_k^{\bt'}\\
  &=
  L_k^{k-1,\bt'}\alpha_{T_k-1}^2
  +M_k^{k-1,\bt'}\alpha_{T_k-1}.
  \end{aligned}
\]
Hence the second sum telescopes:
\begin{align*}
  \sum_{t_k=0}^{T_k-1}
  \frac{S_\bt^k+R_\bt^k}{\alpha_{t_k}}
  &\geq
  \frac{1}{\alpha_{T_k-1}^2}
  \Phi_{\bt',\hat z}^{k-1}
  \bigl(\bar z_{\bt',T_k}^k,\hat z\bigr) \\
  &\quad
  +
  \frac{D_k^{\bt'}}{2}
  \left(
    \|z_{\bt',T_k}^k-\hat z\|_{\bP}^2
    -
    \|z_{\bt',0}^k-\hat z\|_{\bP}^2
  \right) \\
  &\quad
  -
  \delta_k^{k-1,\bt'}
  \sum_{t_k=0}^{T_k-1}
  \frac{1}{\alpha_{t_k}^2}.
\end{align*}
After multiplying both sides by \(\alpha_{T_k-1}^2\), we obtain
\begin{align}
  \sum_{t_k=0}^{T_k-1}
  \frac{\alpha_{T_k-1}^2}{\alpha_{t_k}}
  \left(S_\bt^k+R_\bt^k\right)
  &\geq
  \Phi_{\bt',\hat z}^{k-1}
  \bigl(\bar z_{\bt',T_k}^k,\hat z\bigr) \notag\\
  &\qquad
  +
  \frac{C_k^{\bt'}}{2}
  \left(
    \|z_{\bt',T_k}^k-\hat z\|_{\bP}^2
    -
    \|z_{\bt',0}^k-\hat z\|_{\bP}^2
  \right) \notag\\
  &\qquad
  -
  \delta_k^{k-1,\bt'}
  \sum_{t_k=0}^{T_k-1}
  \frac{\alpha_{T_k-1}^2}{\alpha_{t_k}^2}.
  \label{eq:after_multiply_alpha_T}
\end{align}
By lines~\ref{line:recursive_call} and~\ref{line:recursive_return} of
Algorithm~\ref{alg:sliding_recursive}, we have
\[
  \tilde z_{\bt'}^{k-1}
  =
  \bar z_{\bt',T_k}^k.
\]
Therefore, \eqref{eq:after_multiply_alpha_T} becomes
\begin{align}
  \sum_{t_k=0}^{T_k-1}
  \frac{\alpha_{T_k-1}^2}{\alpha_{t_k}}
  \left(S_\bt^k+R_\bt^k\right)
  &\geq
  \Phi_{\bt',\hat z}^{k-1}
  \bigl(\tilde z_{\bt'}^{k-1},\hat z\bigr) \notag\\
  &\qquad
  +
  \frac{C_k^{\bt'}}{2}
  \left(
    \|z_{\bt',T_k}^k-\hat z\|_{\bP}^2
    -
    \|z_{\bt',0}^k-\hat z\|_{\bP}^2
  \right) \notag\\
  &\qquad
  -
  \delta_k^{k-1,\bt'}
  \sum_{t_k=0}^{T_k-1}
  \frac{\alpha_{T_k-1}^2}{\alpha_{t_k}^2}.
  \label{eq:after_multiply_alpha_T_tilde}
\end{align}
Rearranging~\eqref{eq:after_multiply_alpha_T_tilde}, we get
\begin{align}
  \Phi_{\bt',\hat z}^{k-1}
  \bigl(\tilde z_{\bt'}^{k-1},\hat z\bigr)
  &\leq
  \sum_{t_k=0}^{T_k-1}
  \frac{\alpha_{T_k-1}^2}{\alpha_{t_k}}
  \left(S_\bt^k+R_\bt^k\right) \notag\\
  &\quad
  +
  \frac{C_k^{\bt'}}{2}
  \left(
    \|z_{\bt',0}^k-\hat z\|_{\bP}^2
    -
    \|z_{\bt',T_k}^k-\hat z\|_{\bP}^2
  \right) \notag\\
  &\quad
  +
  \delta_k^{k-1,\bt'}
  \sum_{t_k=0}^{T_k-1}
  \frac{\alpha_{T_k-1}^2}{\alpha_{t_k}^2}.
  \label{eq:rearranged_induction_step}
\end{align}

By the definitions of \(L_i^{k,\bt}\), \(M_i^{k,\bt}\), and \(S_\bt^k\),
and by the warm-start relation in line~\ref{line:warm_start} of
Algorithm~\ref{alg:sliding_recursive}, the terms involving squared
distances telescope. Hence
\begin{align}
  \sum_{t_k=0}^{T_k-1}
  \frac{\alpha_{T_k-1}^2}{\alpha_{t_k}}S_\bt^k
  +
  \frac{C_k^{\bt'}}{2}
  \left(
    \|z_{\bt',0}^k-\hat z\|_{\bP}^2
    -
    \|z_{\bt',T_k}^k-\hat z\|_{\bP}^2
  \right)
  &=
  S_{\bt'}^{k-1}.
  \label{eq:S_telescoping_level_k}
\end{align}

We have
\[
  \frac{1}{\alpha_t^2}
  -
  \frac{1}{\alpha_{t-1}^2}
  =
  \frac{1}{\alpha_t},
  \qquad t\geq 1.
\]
Therefore,
\[
  \sum_{t=0}^{T-1}\frac{1}{\alpha_t}
  =
  \frac{1}{\alpha_{T-1}^2}.
\]
Since \(\{\alpha_t\}\) is nonincreasing, for \(0\leq t\leq T-1\),
\[
  \frac{\alpha_{T-1}^2}{\alpha_t^2}
  \leq
  \frac{\alpha_{T-1}}{\alpha_t}.
\]
Hence
\[
  \sum_{t=0}^{T-1}
  \frac{\alpha_{T-1}^2}{\alpha_t^2}
  \leq
  \alpha_{T-1}
  \sum_{t=0}^{T-1}\frac{1}{\alpha_t}
  =
  \frac{1}{\alpha_{T-1}}.
\]
Applying this with \(T=T_k\), and using
\[
  \delta_i^{k,\bt}
  =
  \frac{\delta_i^{k-1,\bt'}}{\alpha_{t_k}},
  \qquad
  \bt=(\bt',t_k),
\]
we obtain
\begin{align}
  &\sum_{t_k=0}^{T_k-1}
  \frac{\alpha_{T_k-1}^2}{\alpha_{t_k}}R_\bt^k
  +
  \delta_k^{k-1,\bt'}
  \sum_{t_k=0}^{T_k-1}
  \frac{\alpha_{T_k-1}^2}{\alpha_{t_k}^2}
  \notag\\
  &\quad =
  \left(
    \sum_{t_k=0}^{T_k-1}
    \frac{\alpha_{T_k-1}^2}{\alpha_{t_k}^2}
  \right)
  \left[
    \delta_k^{k-1,\bt'}
    +
    \sum_{i=k+1}^n
    \frac{
      \delta_i^{k-1,\bt'}
    }{
      \prod_{j=k+1}^i \alpha_{T_j-1}
    }
  \right]
  \notag\\
  &\quad \leq
  \frac{\delta_k^{k-1,\bt'}}{\alpha_{T_k-1}}
  +
  \sum_{i=k+1}^n
  \frac{
    \delta_i^{k-1,\bt'}
  }{
    \prod_{j=k}^i \alpha_{T_j-1}
  }
  =
  R_{\bt'}^{k-1}.
  \label{eq:R_bound_level_k}
\end{align}
Combining~\eqref{eq:rearranged_induction_step},
\eqref{eq:S_telescoping_level_k}, and~\eqref{eq:R_bound_level_k}, we get
\[
  \Phi_{\bt',\hat z}^{k-1}
  \bigl(\tilde z_{\bt'}^{k-1},\hat z\bigr)
  \leq
  S_{\bt'}^{k-1}+R_{\bt'}^{k-1}.
\]
Since the induction step is valid for every \(1\leq k\leq n\), the claim
holds for all \(0\leq k\leq n\).

\end{proof}

\begin{proof}[Proof of Theorem~\ref{thm:VI_sliding_inexact}]
Fix an arbitrary \(z\in\mathcal Z\) and set \(\hat z=z\). We use the
empty multi-index convention and write
\[
  \begin{aligned}
  f_\varnothing^0(u)&=p(u),\\
  g_{\hat z}^0(u)&=\langle Q(\hat z),u\rangle,\\
  h_{\varnothing,\hat z}^0(u)
  &=p(u)+\langle Q(\hat z),u\rangle .
  \end{aligned}
\]
Since \(H_\Sigma^0=0\), we have
\[
  \Phi_{\varnothing,\hat z}^0(u,\hat z)
  =
  \Delta h_{\varnothing,\hat z}^0(u,\hat z)
  =
  p(u)-p(\hat z)+\langle Q(\hat z),u-\hat z\rangle .
\]
Let \(z_{\mathrm{out}}=\tilde z_\varnothing^0\) denote the output of
Algorithm~\ref{alg:sliding_recursive}. Applying
Lemma~\ref{lem:recursive_gap_bound} with \(k=0\) and \(\hat z=z\), we get
\[
  \mathcal G(z)
  :=
  p(z_{\mathrm{out}})-p(z)
  +
  \langle Q(z),z_{\mathrm{out}}-z\rangle .
\]
Thus
\[
  \mathcal G(z)
  =
  \Phi_{\varnothing,z}^0(z_{\mathrm{out}},z)
  \leq
  S_\varnothing^0+R_\varnothing^0 .
\]

We first bound \(R_\varnothing^0\). By the definition of \(R_\bt^k\),
\[
  R_\varnothing^0
  =
  \sum_{i=1}^n
  \frac{\delta_i}{\prod_{j=1}^i\alpha_{T_j-1}} .
\]
By Lemma~\ref{lem:alpha_properties},
\[
  \alpha_{T_j-1}\geq \frac{1}{T_j},
  \qquad j=1,\dots,n.
\]
Hence
\[
  R_\varnothing^0
  \leq
  \sum_{i=1}^n
  \delta_i
  \prod_{j=1}^i T_j .
\]

Next, we bound \(S_\varnothing^0\). By the definition of \(S_\bt^k\), the
initialization \(z^i_{\mathbf 0_i}=z_{\mathrm{in}}\), and the nonpositivity
of the terminal squared-distance terms, we have
\[
  S_\varnothing^0
  \leq
  \sum_{i=1}^n
  \frac{
    L_i\prod_{j=1}^i\alpha_{T_j-1}^2
    +
    M_i\prod_{j=1}^i\alpha_{T_j-1}
  }{2}
  \|z_{\mathrm{in}}-z\|_{\bP}^2 .
\]
Using Lemma~\ref{lem:alpha_properties} again,
\[
  \alpha_{T_j-1}\leq \frac{2}{T_j},
  \qquad j=1,\dots,n,
\]
and therefore
\[
  S_\varnothing^0
  \leq
  \sum_{i=1}^n
  \left(
    \frac{4^iL_i}{\prod_{j=1}^i T_j^2}
    +
    \frac{2^iM_i}{\prod_{j=1}^i T_j}
  \right)
  \frac{1}{2}
  \|z_{\mathrm{in}}-z\|_{\bP}^2 .
\]

Combining the bounds on \(S_\varnothing^0\) and \(R_\varnothing^0\), we
obtain
\[
  \mathcal G(z)
  \leq
  \sum_{i=1}^n
  \left[
    \left(
      \frac{2^{2i-1}L_i}{\prod_{j=1}^iT_j^2}
      +
      \frac{2^{i-1}M_i}{\prod_{j=1}^iT_j}
    \right)
    \|z_{\mathrm{in}}-z\|_{\bP}^2
    +
    \delta_i\prod_{j=1}^iT_j
  \right].
\]
Since \(z\in\mathcal Z\) was arbitrary, the proof is complete.
\end{proof}
\section{Proof of Theorem~\ref{thm:VI_sliding_Holder}}\label{app:proof_VI_sliding_Holder}

\begin{proof}

\noindent\textbf{Step 1: From the H\"older condition to an inexact oracle.}

\begin{lemma}[H\"older-to-inexact oracle]\label{lem:Holder_to_inexact}
  Let \(\cU \subseteq \R^{d_z}\) be a convex set and let \(h \colon \cU \to \R\) be a convex function which satisfies the \((\nu, H)\)-H\"older condition on \(\cU\) with respect to \(\|\cdot\|_{\bP}\) (Assumption~\ref{ass:Holder_continuous_VI}), with \(0 \leq \nu \leq 1\) and \(H \geq 0\). Then for every \(\delta > 0\), the pair
  \[
    \tilde{h}(v) = h(v), \qquad \tilde{\nabla}h(v) = h'(v) \in \partial h(v)
  \]
  defines a first-order \((\delta, L(\delta))\)-oracle for \(h\) on \(\cU\) with respect to \(\|\cdot\|_{\bP}\), where
  \begin{equation}\label{eq:DGN_L_of_delta}
    L(\delta) = \Bigl[\frac{1-\nu}{2(1+\nu)} \cdot \frac{1}{\delta}\Bigr]^{\frac{1-\nu}{1+\nu}} H^{\frac{2}{1+\nu}}.
  \end{equation}
  In the limit \(\nu = 1\), the formula reduces to \(L(\delta) = H\) for every \(\delta \geq 0\) (a \((0, H)\)-oracle, i.e., exact gradient with \(H\)-Lipschitz smoothness in \(\|\cdot\|_{\bP}\)).
\end{lemma}

\begin{proof}
Fix \(u,v\in\cU\), \(h'(u)\in\partial h(u)\), and
\(h'(v)\in\partial h(v)\). The lower bound of
Definition~\ref{def:inexact_oracle} is the standard convexity inequality
\begin{equation}\label{eq:holder_lower_bound}
  h(u)-h(v)-\langle h'(v),u-v\rangle
  \geq0.
\end{equation}
We turn to the upper bound. Define \(\varphi\colon[0,1]\to\R\) by
\[
  \varphi(s)=h\bigl(v+s(u-v)\bigr).
\]
Since \(\cU\) is convex and \(h\) is convex, \(\varphi\) is finite, convex,
and absolutely continuous on \([0,1]\), with (sub)derivative
\begin{equation}\label{eq:holder_varphi}
  \varphi'(s)
  =
  \left\langle h'\bigl(v+s(u-v)\bigr),u-v\right\rangle
\end{equation}
for almost every \(s\in(0,1)\). Hence
\begin{align}
  &h(u)-h(v)-\langle h'(v),u-v\rangle\notag\\
  &\quad=
  \int_0^1
  \left\langle
    h'\bigl(v+s(u-v)\bigr)-h'(v),u-v
  \right\rangle ds .
  \label{eq:holder_integral}
\end{align}
By Cauchy--Schwarz with respect to the dual pairing
\(\|\cdot\|_{\bP},\|\cdot\|_{\bP^{-1}}\) and the H\"older hypothesis applied
to \((v+s(u-v),v)\in\cU\times\cU\),
\begin{align}
  &\left\langle
    h'\bigl(v+s(u-v)\bigr)-h'(v),u-v
  \right\rangle\notag\\
  &\quad\leq
  \left\|h'\bigl(v+s(u-v)\bigr)-h'(v)\right\|_{\bP^{-1}}
  \|u-v\|_{\bP}\notag\\
  &\quad\leq
  H\|s(u-v)\|_{\bP}^{\nu}\|u-v\|_{\bP}\notag\\
  &\quad=
  Hs^\nu\|u-v\|_{\bP}^{1+\nu}.
  \label{eq:holder_cs_holder}
\end{align}
Integrating \(s\in[0,1]\) gives
\begin{align}
  &h(u)-h(v)-\langle h'(v),u-v\rangle\notag\\
  &\quad\leq
  \frac{H}{1+\nu}\|u-v\|_{\bP}^{1+\nu}.
  \label{eq:holder_integrated}
\end{align}
It remains to show that
\begin{equation}\label{eq:holder_young_target}
  \frac{H}{1+\nu}r^{1+\nu}
  \leq
  \frac{L(\delta)}{2}r^2+\delta
\end{equation}
for all \(r=\|u-v\|_{\bP}\geq0\).
If \(\nu=1\), then
\[
  \frac{H}{1+\nu}r^{1+\nu}
  =
  \frac{H}{2}r^2,
\]
so the desired bound holds with \(L=H\) and \(\delta=0\).

Now let \(0\le\nu<1\). If \(H=0\), the claim is immediate with
\(L(\delta)=0\). Assume \(H>0\). For \(L>0\), consider
\[
  F(r)
  =
  \frac{H}{1+\nu}r^{1+\nu}
  -
  \frac{L}{2}r^2,
  \qquad r\ge0.
\]
The maximum of \(F\) is attained at
\[
  r_\star=\left(\frac{H}{L}\right)^{1/(1-\nu)},
\]
and direct computation gives
\[
  \sup_{r\ge0}F(r)
  =
  \frac{1-\nu}{2(1+\nu)}
  H^{2/(1-\nu)}
  L^{-(1+\nu)/(1-\nu)}.
\]
Therefore
\[
  \frac{H}{1+\nu}r^{1+\nu}
  \le
  \frac{L}{2}r^2+\delta
  \quad\forall r\ge0
\]
whenever
\[
  \frac{1-\nu}{2(1+\nu)}
  H^{2/(1-\nu)}
  L^{-(1+\nu)/(1-\nu)}
  \le
  \delta.
\]
Solving for \(L\) gives
\[
  L
  \ge
  \left[
    \frac{1-\nu}{2(1+\nu)}\cdot\frac{1}{\delta}
  \right]^{\frac{1-\nu}{1+\nu}}
  H^{\frac{2}{1+\nu}}.
\]
Choosing equality gives \(L=L(\delta)\). Combining with
\eqref{eq:holder_integrated},
\begin{align}
  0
  &\leq
  h(u)-h(v)-\langle h'(v),u-v\rangle\notag\\
  &\leq
  \frac{L(\delta)}{2}\|u-v\|_{\bP}^2+\delta .
  \label{eq:holder_oracle_bound}
\end{align}
This is the \((\delta,L(\delta))\)-oracle bound of
Definition~\ref{def:inexact_oracle}.
\end{proof}

\medskip
\noindent\textbf{Step 2: Reduction to Theorem~\ref{thm:VI_sliding_inexact}.}
Temporarily fix \(\delta_1,\dots,\delta_n>0\) and run the Sliding
Algorithm with \(L_k(\delta_k)\) in place of \(L_k(\delta_k^\circ)\) in
the step sizes~\eqref{eq:Holder_eta}. Lemma~\ref{lem:Holder_to_inexact}
shows that the exact value and (sub)gradient of each \(p_i\) form a
\((\delta_i,L_i(\delta_i))\)-oracle. Hence
Theorem~\ref{thm:VI_sliding_inexact} applies. In Step~4 we set
\(\delta_i=\delta_i^\circ\), at which point these step sizes coincide
with~\eqref{eq:Holder_eta}. For brevity, set
\[
  \mathcal G(z)
  :=
  p(z_{\mathrm{out}})-p(z)
  +\langle Q(z),z_{\mathrm{out}}-z\rangle .
\]
Also write
\[
  D_z:=\|z_{\mathrm{in}}-z\|_{\bP},
  \qquad
  N_i:=\prod_{j=1}^iT_j.
\]
Therefore, for every \(z\in\cZ\),
\begin{align}
  \mathcal G(z)
  &\leq
  \sum_{i=1}^n
  \left[
    \left(
      \frac{2^{2i-1}L_i(\delta_i)}{N_i^2}
      +
      \frac{2^{i-1}M_i}{N_i}
    \right)D_z^2
    +
    \delta_iN_i
  \right].
  \label{eq:Holder_pre_optimize}
\end{align}

\medskip
\noindent\textbf{Step 3: Choice of \(\delta_i\).}
Let
\[
  a_i:=\frac{1-\nu_i}{1+\nu_i}.
\]
For \(0\leq\nu_i<1\), write
\[
  L_i(\delta)=K_i\delta^{-a_i},
  \qquad
  K_i
  =
  \left[
    \frac{1-\nu_i}{2(1+\nu_i)}
  \right]^{\frac{1-\nu_i}{1+\nu_i}}
  H_i^{\frac{2}{1+\nu_i}}.
\]
The \(\delta_i\)-dependent contribution in
\eqref{eq:Holder_pre_optimize}, with
\(D=\|z_{\mathrm{in}}-z\|_{\bP}\), is
\begin{equation}\label{eq:Holder_gi}
  g_i(\delta;D)
  =
  \frac{4^iL_i(\delta)}{2}\frac{D^2}{N_i^2}
  +
  \delta N_i
  =
  \frac{2^{2i-1}K_iD^2}{N_i^2}\delta^{-a_i}
  +
  N_i\delta .
\end{equation}
If \(H_iD>0\), then \(g_i(\cdot;D)\) is strictly convex on
\((0,\infty)\), and its first-order condition is
\begin{equation}\label{eq:Holder_foc}
  -a_i
  \frac{2^{2i-1}K_iD^2}{N_i^2}
  \delta^{-a_i-1}
  +
  N_i
  =
  0.
\end{equation}
Equivalently,
\[
  -\frac{1-\nu_i}{1+\nu_i}
  \frac{2^{2i-1}K_iD^2}{N_i^2}
  \delta^{-2/(1+\nu_i)}
  +
  N_i
  =
  0.
\]
Thus the unique minimizer is
\begin{equation}\label{eq:Holder_delta_optimal}
  \delta_i^\star(D)
  =
  \left[
    \frac{(1-\nu_i)\,2^{2i-1}K_iD^2}
    {(1+\nu_i)N_i^3}
  \right]^{(1+\nu_i)/2}.
\end{equation}
If \(H_i=0\) or \(D=0\), the corresponding function contribution is trivial
and the same expression may be interpreted by continuity.
At the minimum, the two summands satisfy
\begin{equation}\label{eq:Holder_ratio}
  \frac{4^iL_i(\delta_i^\star)}{N_i^2}\frac{D^2}{2}
  =
  \frac{1+\nu_i}{1-\nu_i}
  \delta_i^\star N_i,
\end{equation}
and therefore
\begin{equation}\label{eq:Holder_optimal_value}
  g_i(\delta_i^\star(D);D)
  =
  \frac{2}{1-\nu_i}N_i\delta_i^\star(D).
\end{equation}
Using
\begin{equation}\label{eq:Holder_Ki_identity}
  K_i^{(1+\nu_i)/2}
  =
  \left[
    \frac{1-\nu_i}{2(1+\nu_i)}
  \right]^{(1-\nu_i)/2}
  H_i,
\end{equation}
we obtain
\begin{equation}\label{eq:Holder_per_component}
  g_i(\delta_i^\star(D);D)
  =
  \frac{
    2^{i(1+\nu_i)}H_iD^{1+\nu_i}
  }{
    1+\nu_i
  }
  \frac{1}{N_i^{(1+3\nu_i)/2}}.
\end{equation}

\medskip
\medskip
\noindent\textbf{Step 4: Choosing \(\delta_i\) from
\(\Omega_{z_{\mathrm{in}}}\) and taking the supremum over \(z\).}

The minimizer \(\delta_i^\star(D)\) in~\eqref{eq:Holder_delta_optimal}
depends on
\[
  D=\|z_{\mathrm{in}}-z\|_{\bP},
\]
and hence generally depends on the comparator \(z\). Since the algorithm
must use a single value of \(\delta_i\) throughout the run, we choose it
using the worst-case bound
\[
  D^2
  \leq
  \Omega_{z_{\mathrm{in}}}
  :=
  \sup_{z\in\mathcal Z}
  \|z_{\mathrm{in}}-z\|_{\bP}^2.
\]
Thus, for \(0\leq\nu_i<1\), we set
\[
  \delta_i^\circ
  =
  \delta_i^\star\bigl(\sqrt{\Omega_{z_{\mathrm{in}}}}\bigr)
  =
  \left[
    \frac{(1-\nu_i)\,2^{2i-1}K_i\,\Omega_{z_{\mathrm{in}}}}
    {(1+\nu_i)N_i^3}
  \right]^{(1+\nu_i)/2}.
\]
If \(H_i=0\), then \(K_i=0\), and we interpret this choice as
\(\delta_i^\circ=0\) and \(L_i(\delta_i^\circ)=0\). For \(\nu_i=1\), we set
\[
  \delta_i^\circ=0,
  \qquad
  L_i(\delta_i^\circ)=H_i.
\]

Substituting \(\delta_i=\delta_i^\circ\) into
\eqref{eq:Holder_pre_optimize}, we obtain, for every \(z\in\mathcal Z\),
\[
  \mathcal G(z)
  \leq
  \sum_{i=1}^n
  \left[
    \left(
      \frac{2^{2i-1}L_i(\delta_i^\circ)}{N_i^2}
      +
      \frac{2^{i-1}M_i}{N_i}
    \right)D_z^2
    +
    \delta_i^\circ N_i
  \right].
\]
Since
\[
  \|z_{\mathrm{in}}-z\|_{\bP}^2
  \leq
  \Omega_{z_{\mathrm{in}}},
  \qquad \forall z\in\mathcal Z,
\]
we get
\[
  \sup_{z\in\mathcal Z}\mathcal G(z)
  \leq
  \sum_{i=1}^n
  \left[
    \left(
      \frac{2^{2i-1}L_i(\delta_i^\circ)}{N_i^2}
      +
      \frac{2^{i-1}M_i}{N_i}
    \right)
    \Omega_{z_{\mathrm{in}}}
    +
    \delta_i^\circ N_i
  \right].
\]

For \(0\leq\nu_i<1\), the choice of \(\delta_i^\circ\) gives
\[
  \frac{4^iL_i(\delta_i^\circ)}{N_i^2}
  \frac{\Omega_{z_{\mathrm{in}}}}{2}
  +
  \delta_i^\circ N_i
  =
  \frac{
    2^{i(1+\nu_i)}H_i
    \Omega_{z_{\mathrm{in}}}^{(1+\nu_i)/2}
  }{
    (1+\nu_i)
    N_i^{(1+3\nu_i)/2}
  }.
\]
For \(\nu_i=1\), the same expression is obtained by taking
\(\delta_i^\circ=0\) and \(L_i(\delta_i^\circ)=H_i\), since
\[
  \frac{4^iH_i}{N_i^2}
  \frac{\Omega_{z_{\mathrm{in}}}}{2}
  =
  \frac{2^{2i-1}H_i\Omega_{z_{\mathrm{in}}}}{N_i^2}
  =
  \frac{
    2^{i(1+\nu_i)}H_i
    \Omega_{z_{\mathrm{in}}}^{(1+\nu_i)/2}
  }{
    (1+\nu_i)
    N_i^{(1+3\nu_i)/2}
  }
  \bigg|_{\nu_i=1}.
\]
Therefore,
\[
  \sup_{z\in\mathcal Z}\mathcal G(z)
  \leq
  \sum_{i=1}^n
  \left[
    \frac{
      2^{i(1+\nu_i)}H_i
      \Omega_{z_{\mathrm{in}}}^{(1+\nu_i)/2}
    }{
      (1+\nu_i)N_i^{(1+3\nu_i)/2}
    }
    +
    \frac{2^{i-1}M_i\Omega_{z_{\mathrm{in}}}}{N_i}
  \right].
\]
This proves~\eqref{eq:Holder_gap_bound}.
\end{proof}

\section{Proof of Corollary~\ref{cor:VI_sliding_Holder_complexity}} \label{app:proof_VI_Holder_complexity}
\begin{proof}

Set \(\Omega:=\Omega_{z_{\mathrm{in}}}\). Since the finite-sum
representation
\[
  p=\sum_{i=1}^n p_i,
  \qquad
  Q=\sum_{i=1}^n Q_i
\]
is unchanged by a simultaneous relabeling of the pairs \((p_i,Q_i)\), we
order the components so that
\[
  R_1\leq R_2\leq\cdots\leq R_n .
\]

We first compare \(R_i\) with the explicit lower bounds needed in
Theorem~\ref{thm:VI_sliding_Holder}. Define
\[
  \widetilde R_{i,H}
  :=
  \left(
    \frac{
      2^{i(1+\nu_i)+1}H_i\Omega^{(1+\nu_i)/2}
    }{
      (1+\nu_i)\varepsilon
    }
  \right)^{\frac{2}{1+3\nu_i}},
\]
\[
  \widetilde R_{i,M}
  :=
  \frac{2^iM_i\Omega}{\varepsilon}.
\]
Then set
\[
  \widetilde R_i
  :=
  \max\left\{
    \widetilde R_{i,H},
    \widetilde R_{i,M},
    1
  \right\}.
\]
We claim that there exists a constant \(\kappa_n>0\), depending only on
\(n\), such that
\[
  \widetilde R_i\leq \kappa_n R_i,
  \qquad i=1,\dots,n.
\]
Indeed, since \(i\leq n\), \(0\leq\nu_i\leq1\), and \(1+\nu_i\geq1\),
\[
  \left(
    \frac{2^{i(1+\nu_i)+1}}{1+\nu_i}
  \right)^{\frac{2}{1+3\nu_i}}
  \leq
  2^{2n+2}.
\]
Also \(2^i\leq 2^n\leq 2^{2n+2}\). Hence we may take
\[
  \kappa_n:=2^{2n+2},
\]
and obtain \(\widetilde R_i\leq\kappa_n R_i\) for all \(i\).

Next, we construct positive integers \(T_1,\dots,T_n\). Let \(N_0:=1\), and
define recursively
\[
  T_i
  :=
  \max\left\{
    1,
    \left\lceil \frac{\kappa_n R_i}{N_{i-1}}\right\rceil
  \right\}.
\]
Set
\[
  N_i:=N_{i-1}T_i,
  \qquad i=1,\dots,n.
\]
Then each \(T_i\) is a positive integer. We claim that
\[
  \kappa_n R_i\leq N_i\leq 2\kappa_n R_i,
  \qquad i=1,\dots,n.
\]
The lower bound follows directly from the definition of \(T_i\). For the
upper bound, use induction. For \(i=1\),
\[
  N_1
  =
  \left\lceil \kappa_n R_1\right\rceil
  \leq
  \kappa_n R_1+1
  \leq
  2\kappa_n R_1,
\]
because \(R_1\geq1\). Suppose \(N_{i-1}\leq2\kappa_n R_{i-1}\). If
\(\kappa_n R_i\leq N_{i-1}\), then \(T_i=1\), and therefore
\[
  N_i=N_{i-1}\leq2\kappa_n R_{i-1}\leq2\kappa_n R_i.
\]
If \(\kappa_n R_i>N_{i-1}\), then
\[
  T_i
  =
  \left\lceil \frac{\kappa_n R_i}{N_{i-1}}\right\rceil
  \leq
  \frac{\kappa_n R_i}{N_{i-1}}+1,
\]
and hence
\[
  N_i
  =
  N_{i-1}T_i
  \leq
  \kappa_n R_i+N_{i-1}
  <
  2\kappa_n R_i.
\]
Thus \(N_i\leq2\kappa_n R_i\) for all \(i\).

Since
\[
  N_i\geq \kappa_n R_i\geq \widetilde R_i,
\]
the choice of \(N_i\) satisfies the explicit lower bounds
\[
  \begin{aligned}
  N_i
  \geq
  \left(
    \frac{
      2^{i(1+\nu_i)+1}H_i\Omega^{(1+\nu_i)/2}
    }{
      (1+\nu_i)\varepsilon
    }
  \right)^{\frac{2}{1+3\nu_i}},\\
  N_i
  \geq
  \frac{2^iM_i\Omega}{\varepsilon}.
  \end{aligned}
\]
Therefore, by Theorem~\ref{thm:VI_sliding_Holder},
\[
  \sup_{z\in\cZ}
  \left\{
    p(z_{\mathrm{out}})-p(z)
    +
    \langle Q(z),z_{\mathrm{out}}-z\rangle
  \right\}
  \leq
  \sum_{i=1}^n(A_i+B_i),
\]
where
\[
  A_i
  :=
  \frac{
    2^{i(1+\nu_i)}H_i\Omega^{(1+\nu_i)/2}
  }{
    (1+\nu_i)N_i^{(1+3\nu_i)/2}
  },
  \qquad
  B_i
  :=
  \frac{2^{i-1}M_i\Omega}{N_i}.
\]
The first lower bound on \(N_i\) gives \(A_i\leq\varepsilon/2\), and the
second gives \(B_i\leq\varepsilon/2\). Hence
\[
  A_i+B_i\leq\varepsilon,
  \qquad i=1,\dots,n.
\]
Summing over \(i\) yields
\[
  \sum_{i=1}^n(A_i+B_i)\leq n\varepsilon.
\]
Thus
\[
  \sup_{z\in\cZ}
  \left\{
    p(z_{\mathrm{out}})-p(z)
    +
    \langle Q(z),z_{\mathrm{out}}-z\rangle
  \right\}
  \leq
  n\varepsilon .
\]

Finally, since \(N_i\leq 2\kappa_n R_i\), the required number of evaluations
of \(p_i'\) and \(Q_i\) is at most \(C_nR_i\), with \(C_n=2^{2n+4}\). This
proves the corollary.

\end{proof}

\section{Subspace Structure under Assumption~\ref{ass:properties_of_B}}
\label{app:subspace}

Throughout this section we write
\[
  V_x:=\range(\bB^\top)=(\ker\bB)^\perp,
  \qquad
  V_y:=\range(\bB)=(\ker\bB^\top)^\perp .
\]
For the zero matrix, Assumption~\ref{ass:properties_of_B} is understood with
the convention \(\lminp(\mathbf 0)=0\).
Under \(\mathsf R_x\), Assumption~\ref{ass:properties_of_B} uses the
Rayleigh-quotient bound \(\lminp(\bB^\top\bB)\) on \(V_x\). Otherwise it uses
the ambient-space bound \(\lmin(\bB^\top\bB)\). The lemmas below establish
the required subspace invariance. The \(y\)-block is symmetric.

\begin{lemma}[Structure of the solution set]
\label{lem:solution_set_structure}
Suppose \(\mathsf R_x\) holds and \(\cX+\ker\bB=\cX\). Then:
\begin{enumerate}
  \item \(f(x+u)=f(x)\) and \(\partial f(x+u)=\partial f(x)\) for every
    \(x\in\cX\) and \(u\in\ker\bB\).
  \item if \(\mu_x>0\), then \(\ker\bB=\{0\}\).
  \item \(F(x+u,y)=F(x,y)\) for every \((x,y)\in\cZ\) and \(u\in\ker\bB\).
    Consequently \(\cS+(\ker\bB\times\{0\})=\cS\).
\end{enumerate}
The symmetric statements hold for \(g\), \(\cY\) and \(\ker\bB^\top\) under
\(\mathsf R_y\). If \(\cS\neq\varnothing\), it contains a point \(z^*\) such
that
\[
  \xin-x^*\in V_x.
\]
If, in addition, \(\mathsf R_y\) and
\(\cY+\ker\bB^\top=\cY\) hold, \(z^*\) may be chosen so that
\[
  \zin-z^*\in V_x\times V_y.
\]
\end{lemma}

\begin{proof}
\noindent\emph{Part (1): invariance of \(f\) and \(\partial f\).}
Fix \(x\in\cX\) and \(u\in\ker\bB\). The cylinder condition gives
\(x+u\in\cX\), while \(\mathsf R_x\) gives
\[
  f'(x),f'(x+u)\in V_x=(\ker\bB)^\perp.
\]
Applying convexity in both directions,
\[
  f(x+u)
  \geq f(x)+\langle f'(x),u\rangle
  =f(x),
\]
and
\[
  f(x)
  \geq f(x+u)+\langle f'(x+u),-u\rangle
  =f(x+u).
\]
Hence \(f(x+u)=f(x)\).

Now take \(q\in\partial f(x)\). Since \(q\in V_x\), for every \(w\in\R^{d_x}\),
\[
  f(w)
  \geq
  f(x)+\langle q,w-x\rangle
  =
  f(x+u)+\langle q,w-(x+u)\rangle.
\]
Thus \(q\in\partial f(x+u)\), proving
\(\partial f(x)\subseteq\partial f(x+u)\). Replacing \(u\) by \(-u\) gives
the reverse inclusion.

\smallskip
\noindent\emph{Part (2): strong convexity eliminates kernel directions.}
Fix \(x\in\cX\) and \(u\in\ker\bB\). By part~(1) and \(\mathsf R_x\),
\(f(x+u)=f(x)\) and \(\langle f'(x),u\rangle=0\). Strong convexity therefore
gives
\[
  0
  =
  f(x+u)-f(x)
  \geq
  \frac{\mu_x}{2}\|u\|^2.
\]
If \(\mu_x>0\), then \(u=0\). Hence \(\ker\bB=\{0\}\).

\smallskip
\noindent\emph{Part (3): invariance of \(F\) and \(\cS\).}
For \(u\in\ker\bB\), part~(1) gives
\[
  F(x+u,y)
  =
  f(x)+\langle y,\bB x+\bB u\rangle-g(y)
  =
  F(x,y).
\]
The cylinder condition preserves feasibility, so translating the
\(x\)-component of any saddle point by \(u\) yields another saddle point.
This proves
\[
  \cS+(\ker\bB\times\{0\})=\cS.
\]
Interchanging \(x,f,\bB\) with \(y,g,\bB^\top\) proves the symmetric
statements.

\smallskip
\noindent\emph{Choice of representative.}
Pick \(\hat z=(\hat x,\hat y)\in\cS\) and define
\[
  x^*
  :=
  \hat x+\Pi_{\ker\bB}(\xin-\hat x),
  \qquad
  y^*:=\hat y.
\]
Part~(3) shows that \(z^*:=(x^*,y^*)\in\cS\), and
\[
  \xin-x^*
  =
  \Pi_{V_x}(\xin-\hat x)
  \in V_x.
\]
If the symmetric hypotheses also hold, replace \(y^*\) by
\[
  y^*
  :=
  \hat y+\Pi_{\ker\bB^\top}(\yin-\hat y).
\]
The symmetric version of part~(3) preserves \(z^*\in\cS\), and
\[
  \yin-y^*
  =
  \Pi_{V_y}(\yin-\hat y)
  \in V_y.
\]
\end{proof}

\paragraph{Consequences for the Lyapunov function.}
Under the \(x\)-block hypotheses, kernel translations generate the solution
orbit
\[
  \mathcal K_x(z^*)
  :=
  \{(x^*+u,y^*):u\in\ker\bB\}
  \subseteq\cS.
\]
Using the same subgradient
\(f'(x^*)\in\partial f(x^*+u)\), part~(1) of
Lemma~\ref{lem:solution_set_structure} gives
\[
  \D_f(x,x^*+u)
  =
  f(x)-f(x^*)-\langle f'(x^*),x-x^*\rangle+\langle f'(x^*),u\rangle
  =
  \D_f(x,x^*).
\]
Moreover, if \(x-x^*\in V_x\), orthogonality gives
\[
  \delta_x\|x-x^*\|^2
  =
  \inf_{u\in\ker\bB}\delta_x\|x-(x^*+u)\|^2.
\]
The symmetric statements hold for the \(y\)-block. In particular, when both
sets of hypotheses hold and
\[
  \mathcal K(z^*)
  :=
  z^*+(\ker\bB\times\ker\bB^\top)\subseteq\cS,
\]
then, for \(z-z^*\in V_x\times V_y\),
\[
  \|z-z^*\|_{\bP}^2
  =
  \operatorname{dist}_{\bP}(z,\mathcal K(z^*))^2.
\]
Since \(\mathcal K(z^*)\subseteq\cS\), this quantity upper-bounds
\(\operatorname{dist}_{\bP}(z,\cS)^2\). The restart analysis uses the fixed
aligned representative \(z^*\).

\begin{lemma}[Subspace invariance]\label{lem:subspace_invariance}
Suppose \(\mathsf R_x\) holds and \(\cX+\ker\bB=\cX\). Apply
Algorithm~\ref{alg:sliding_recursive} to the
decomposition~\eqref{eq:SPP_decomposition}--\eqref{eq:SPP_operator} with
\(\bP=\diag(\delta_x\bI_{d_x},\delta_y\bI_{d_y})\) and input
\(\zin=(\xin,\yin)\). Then every point \(z=(x,y)\) generated by the algorithm
satisfies \(x-\xin\in V_x\). The symmetric statement holds for \(y\) under
\(\mathsf R_y\) and \(\cY+\ker\bB^\top=\cY\).
\end{lemma}

\begin{proof}
Set
\[
  K_x:=\ker\bB,
  \qquad
  \cX_V:=\cX\cap V_x.
\]
The cylinder condition implies
\[
  \cX=\cX_V\oplus K_x.
\]
Indeed, if \(x\in\cX\), then
\(\Pi_{V_x}x=x-\Pi_{K_x}x\in\cX_V\). Conversely,
\(\cX_V+K_x\subseteq\cX\).

For any \(q=q_V+q_K\in V_x\oplus K_x\), projection onto this direct sum
separates across the two subspaces:
\[
  \Pi_{\cX}(q)
  =
  \Pi_{\cX_V}(q_V)+q_K,
  \qquad\text{hence}\qquad
  \Pi_{\cX}(q)-q\in V_x.
\]
Moreover, because \(\cZ=\cX\times\cY\) and
\(\bP=\diag(\delta_x\bI,\delta_y\bI)\), the \(\bP\)-projection is blockwise:
\[
  \Pi_{\cZ}^{\bP}(q_x,q_y)
  =
  \bigl(\Pi_{\cX}(q_x),\Pi_{\cY}(q_y)\bigr).
\]
Thus, if \(q_x\in\xin+V_x\), the projected \(x\)-block remains in
\(\xin+V_x\).

By Appendix~\ref{app:algo_implementability}, every minimization in
Algorithm~\ref{alg:sliding_recursive} has the form
\[
  \Pi^{\bP}_{\cZ}
  \left(
    \frac{\sum_m\theta_ma_m}{\sum_m\theta_m}
    -\frac{1}{\sum_m\theta_m}\bP^{-1}c
  \right),
\]
where \(c\) is a sum of terms \(\tilde\nabla\hat p_i(\cdot)\) and
\(Q_i(\cdot)\). Their \(x\)-blocks satisfy
\[
  \begin{aligned}
    p_1 &: \quad f'(\cdot)\in V_x
      &&\text{by \(\mathsf R_x\)},\\
    p_2 &: \quad 0,\\
    p_3 &: \quad
      \beta_x\bB^\top\bigl(\bB\,\cdot-g'(\yin)\bigr)\in V_x,\\
    Q_3 &: \quad \bB^\top y\in V_x,
      \qquad Q_1=Q_2=0.
  \end{aligned}
\]
Lemma~\ref{lem:Holder_to_inexact} supplies exact subgradients. The affine
rescaling used to construct \(\hat p_i\) changes only their evaluation points
and scalar factors, so the same inclusions hold for
\(\tilde\nabla\hat p_i\). Therefore the \(x\)-block of \(c\) lies in \(V_x\).
Since the \(x\)-block of \(\bP^{-1}\) is
\(\delta_x^{-1}\bI\), the \(x\)-block of \(\bP^{-1}c\) also lies in \(V_x\).

To complete the argument, let
\[
  \cA_x
  :=
  \{(x,y):x\in\xin+V_x\}.
\]
The initial point belongs to \(\cA_x\). This affine set is preserved by
convex combinations, subtraction of
\(\bP^{-1}c/\sum_m\theta_m\), and the final projection. Induction over the
algorithm therefore gives \(x-\xin\in V_x\) for every generated point.

The \(y\)-block statement follows by the same argument with
\((\cY,V_y,\ker\bB^\top)\) in place of \((\cX,V_x,\ker\bB)\).
\end{proof}

\begin{remark}[Fixed representative across restarts]\label{rem:restart_slice}
Choose \(z^*\) as in Section~\ref{sec:bilinear_SPP}, aligned with \(z^0\) in
each block whose range condition is invoked.
Lemma~\ref{lem:subspace_invariance} keeps \(x^s\in x^0+V_x\) when
\(\mathsf R_x\) is invoked and \(y^s\in y^0+V_y\) when \(\mathsf R_y\) is
invoked. Hence every restart remains aligned with the same \(z^*\), and the
contraction inequalities telescope with a fixed \(\Psi\).
\end{remark}

\begin{corollary}\label{cor:mu_xy_valid}
Let \(z^*\in\cS\) be chosen as in Section~\ref{sec:bilinear_SPP}: require
\(\xin-x^*\in V_x\) when \(\mathsf R_x\) is invoked and
\(\yin-y^*\in V_y\) when \(\mathsf R_y\) is invoked. Then the output \(\zout\) of
Algorithm~\ref{alg:sliding_recursive} satisfies
\[
  \|\bB(\xout-x^*)\|^2\geq\mu_{xy}^2\|\xout-x^*\|^2,
  \quad
  \|\bB^\top(\yout-y^*)\|^2\geq\mu_{yx}^2\|\yout-y^*\|^2 .
\]
\end{corollary}

\begin{proof}
If \(\mathsf R_x\) does not hold, Assumption~\ref{ass:properties_of_B} gives
\(\mu_{xy}^2\leq\lmin(\bB^\top\bB)\) and the first inequality holds for every
vector. Under \(\mathsf R_x\),
Lemma~\ref{lem:subspace_invariance} and the choice of \(z^*\) give
\(\xout-x^*\in V_x\), where the Rayleigh bound
\(\lminp(\bB^\top\bB)\geq\mu_{xy}^2\) applies. The second inequality is
symmetric.
\end{proof}

\begin{remark}[Normal conditions at \(z^*\)]\label{rem:exact_optimality}
The conditions \(\zeta=0\) when \(\beta_x>0\) and \(\xi=0\) when
\(\beta_y>0\) yield, respectively,
\[
  g'(y^*)=\bB x^*,
  \qquad
  f'(x^*)=-\bB^\top y^*.
\]
Interior and unconstrained blocks have zero normal vectors. When
\(\beta_x=\beta_y=0\), equivalently \(\nu_x<1\) and \(\nu_y<1\), these
conditions impose no restriction.
\end{remark}

\section{Proof of Theorem~\ref{thm:bilinear_SPP_Holder}}\label{app:proof_bilinear_SPP_Holder}
\begin{proof}

  Let $z = (x,y) \in \mathcal{Z} = \mathcal{X} \times \mathcal{Y}$ and $z^\prime = (x^\prime,y^\prime) \in \mathcal{Z} = \mathcal{X} \times \mathcal{Y}$.
  For brevity, set
  \[
    d_x:=x-x',
    \qquad
    d_y:=y-y',
  \]
  and define
  \[
    \mathcal E
    :=
    p(z_{\mathrm{out}})-p(z^*)
    +
    \langle Q(z^*),z_{\mathrm{out}}-z^*\rangle .
  \]
  At the smooth endpoints, we choose
  \[
    \delta_f=0\quad\text{if }\nu_x=1,
    \qquad
    \delta_g=0\quad\text{if }\nu_y=1.
  \]
  The tolerances for components with exponent below one are balanced later.

  For \(p_1(z)=f(x)\), Lemma~\ref{lem:Holder_to_inexact} gives
\[
  p_1(z)-\tilde p_1(z')
  -\langle \tilde\nabla p_1(z'),z-z'\rangle
  \leq
  \frac{L_x}{2}\|x-x'\|^2+\delta_f.
\]
Since
\[
  \|z-z'\|_{\bP}^2
  =
  \delta_x\|x-x'\|^2+\delta_y\|y-y'\|^2,
\]
we get
  \begin{align*}
  &p_1(z)-\tilde p_1(z')
  -\langle \tilde\nabla p_1(z'),z-z'\rangle\\
  &\quad\leq
  \frac{L_x}{2\delta_x}\|z-z'\|_{\bP}^2+\delta_f\\
  &\quad=
  \frac{\kappa_x}{2}\|z-z'\|_{\bP}^2+\delta_f.
  \end{align*}
Similarly, \(p_2\) admits a \((\delta_g,\kappa_y)\)-oracle with respect to
\(\|\cdot\|_{\bP}\). Hence we may take $L_1=\kappa_x, L_2=\kappa_y$.

For \(p_3\), we have
  \[
    \nabla p_3(z)-\nabla p_3(z')
  =
  \begin{pmatrix}
    \beta_x\bB^\top\bB(x-x')\\
    \beta_y\bB\bB^\top(y-y')
  \end{pmatrix}.
  \]
  Define
  \[
    \gamma_B
    :=
    \max\{\beta_x\delta_y,\beta_y\delta_x\}.
  \]
  Therefore
  \[
  \|\nabla p_3(z)-\nabla p_3(z')\|_{\bP^{-1}}
  \leq
  \kappa_{xy}\gamma_B\|z-z'\|_{\bP}.
  \]

By the assumptions of the theorem,
\[
  \beta_x\delta_y \leq \frac14,
  \qquad
  \beta_y\delta_x \leq \frac14.
\]
Hence
\[
  \|\nabla p_3(z)-\nabla p_3(z')\|_{\bP^{-1}}
  \leq
  \frac{\kappa_{xy}}{4}\|z-z'\|_{\bP}
  \leq
  \kappa_{xy}\|z-z'\|_{\bP}.
\]
Therefore \(p_3\) is \(\kappa_{xy}\)-smooth with respect to
\(\|\cdot\|_{\bP}\), and we take \(L_3=\kappa_{xy}\).

Since
\[
  Q_3(z)-Q_3(z')
  =
  \begin{pmatrix}
    \bB^\top d_y\\
    -\bB d_x
  \end{pmatrix},
\]
we can upper-bound \(\|Q_3(z)-Q_3(z')\|_{\bP^{-1}}\) as follows:
\begin{align*}
  \|Q_3(z)-Q_3(z')\|_{\bP^{-1}}^2
  &=
    \frac{1}{\delta_x}\|\bB^\top d_y\|^2
    +
    \frac{1}{\delta_y}\|\bB d_x\|^2\\
  &\leq
  \kappa_{xy}
  \left(
    \delta_x\|d_x\|^2+\delta_y\|d_y\|^2
  \right)\\
  &=
  \kappa_{xy}\|z-z'\|_{\bP}^2.
\end{align*}

Hence we can choose $M_3 = \sqrt{\kappa_{xy}}$.

We have
\begin{align*}
  \mathcal E
  &=
  f(\xout)-f(x^*)
  +
  \langle \bB^\top y^*,\xout-x^*\rangle\\
  &\quad
  +
  g(\yout)-g(y^*)
  -
  \langle \bB x^*,\yout-y^*\rangle\\
  &\quad+
  \frac{\beta_x}{2}
  \left(
    \|\bB\xout-g'(\yin)\|^2
    -
    \|\bB x^*-g'(\yin)\|^2
  \right)\\
  &\quad+
  \frac{\beta_y}{2}
  \left(
    \|\bB^\top\yout+f'(\xin)\|^2
    -
    \|\bB^\top y^*+f'(\xin)\|^2
  \right).
\end{align*}

By the optimality conditions~\eqref{eq:opt_cond},
\[
  f'(x^*)+\bB^\top y^*=-\xi,
  \qquad
  g'(y^*)-\bB x^*=-\zeta,
\]
with \(\xi\in N_{\cX}(x^*)\) and \(\zeta\in N_{\cY}(y^*)\). Since
\(\xout\in\cX\) and \(\yout\in\cY\), the definition of the normal cone gives
\(\langle\xi,\xout-x^*\rangle\leq0\) and
\(\langle\zeta,\yout-y^*\rangle\leq0\). Hence
\begin{align*}
  f(\xout)-f(x^*)
  +\langle\bB^\top y^*,\xout-x^*\rangle
  &=\D_f(\xout,x^*)\\
  &\quad-\langle\xi,\xout-x^*\rangle\\
  &\geq\D_f(\xout,x^*),\\
  g(\yout)-g(y^*)
  -\langle\bB x^*,\yout-y^*\rangle
  &=\D_g(\yout,y^*)\\
  &\quad-\langle\zeta,\yout-y^*\rangle\\
  &\geq\D_g(\yout,y^*).
\end{align*}
For brevity, write
\[
  \begin{aligned}
  D_{f,o}&:=\D_f(\xout,x^*),
  &D_{g,o}&:=\D_g(\yout,y^*),\\
  D_{f,i}&:=\D_f(\xin,x^*),
  &D_{g,i}&:=\D_g(\yin,y^*).
  \end{aligned}
\]

For the quadratic terms, we use the elementary inequality
\[
  \frac12\|a+b\|^2-\frac12\|b\|^2
  \geq
  \frac14\|a\|^2-\|b\|^2.
\]
Multiplying this inequality by the nonnegative coefficients
\(\beta_x\) and \(\beta_y\), and applying it with \(a=\bB(\xout-x^*)\),
\(b=\bB x^*-g'(\yin)\), and then with
\(a=\bB^\top(\yout-y^*)\),
\(b=\bB^\top y^*+f'(\xin)\), gives
\begin{align*}
  \mathcal E
  &\geq
  D_{f,o}+D_{g,o}\\
  &\quad+
  \frac{\beta_x}{4}\|\bB(\xout-x^*)\|^2
  \\
  &\quad+
  \frac{\beta_y}{4}\|\bB^\top(\yout-y^*)\|^2                                      \\
  &\quad
  -
  \beta_x\|\bB x^*-g'(\yin)\|^2
  -
  \beta_y\|\bB^\top y^*+f'(\xin)\|^2 .
\end{align*}

For active coefficients, the normal conditions in
Section~\ref{sec:bilinear_SPP} and~\eqref{eq:opt_cond} give
\[
  \bB x^*=g'(y^*)
  \ \ \text{in the }\beta_x\text{-term},
  \quad
  \bB^\top y^*=-f'(x^*)
  \ \ \text{in the }\beta_y\text{-term}
\]
(see Remark~\ref{rem:exact_optimality}). We obtain
\begin{align*}
  \mathcal E
  &\geq
  D_{f,o}+D_{g,o}\\
  &\quad+
  \frac{\beta_x}{4}\|\bB(\xout-x^*)\|^2
  \\
  &\quad+
  \frac{\beta_y}{4}\|\bB^\top(\yout-y^*)\|^2                                      \\
  &\quad
  -
  \beta_x\|g'(\yin)-g'(y^*)\|^2
  -
  \beta_y\|f'(\xin)-f'(x^*)\|^2 .
\end{align*}

If \(\beta_x>0\), then
\(\nu_y=1\), for which we chose \(\delta_g=0\), and the ambient self-bounding
inequality~\eqref{eq:ambient_self_bounding} gives
\[
  \|g'(\yin)-g'(y^*)\|^2
  \leq
  2L_y\D_g(\yin,y^*) .
\]
Similarly, if \(\beta_y>0\), then \(\nu_x=1\), for which we chose
\(\delta_f=0\), and
\eqref{eq:ambient_self_bounding} gives
\[
  \|f'(\xin)-f'(x^*)\|^2
  \leq
  2L_x\D_f(\xin,x^*) .
\]

Therefore, using the definition of \(\beta_x\) and \(\beta_y\) in
\eqref{eq:SPP_beta_xy_holder},
\[
  2\beta_x L_y \leq \frac12,
  \qquad
  2\beta_y L_x \leq \frac12.
\]
Hence
\[
  \beta_x\|g'(\yin)-g'(y^*)\|^2
  +\beta_y\|f'(\xin)-f'(x^*)\|^2
  \leq \frac12(D_{g,i}+D_{f,i}).
\]

Substituting this into the previous lower bound gives
\begin{align*}
  \mathcal E
  &\geq
  D_{f,o}+D_{g,o}\\
  &\quad+
  \frac{\beta_x}{4}\|\bB(\xout-x^*)\|^2
  \\
  &\quad+
  \frac{\beta_y}{4}\|\bB^\top(\yout-y^*)\|^2                                      \\
  &\quad
  -
  \frac12(D_{f,i}+D_{g,i}).
\end{align*}

Using Corollary~\ref{cor:mu_xy_valid}, which supplies
\[
  \begin{aligned}
  \|\bB(\xout-x^*)\|^2
  &\geq
  \mu_{xy}^2\|\xout-x^*\|^2,\\
  \|\bB^\top(\yout-y^*)\|^2
  &\geq
  \mu_{yx}^2\|\yout-y^*\|^2,
  \end{aligned}
\]
and strong convexity,
\[
  \begin{aligned}
  \D_f(\xout,x^*)&\geq \frac{\mu_x}{2}\|\xout-x^*\|^2,\\
  \D_g(\yout,y^*)&\geq \frac{\mu_y}{2}\|\yout-y^*\|^2,
  \end{aligned}
\]
we obtain
\[
  \begin{aligned}
  c_x
  &:=
  \frac{\mu_x}{8}
  +
  \frac{\beta_x\mu_{xy}^2}{4},\\
  c_y
  &:=
  \frac{\mu_y}{8}
  +
  \frac{\beta_y\mu_{yx}^2}{4}.
  \end{aligned}
\]
Then
\[
  \mathcal E
  \geq \frac34(D_{f,o}+D_{g,o})-\frac12(D_{f,i}+D_{g,i})
  +c_x\|\xout-x^*\|^2+c_y\|\yout-y^*\|^2 .
\]

By the definitions of the effective strong convexity parameters,
equivalently,
\[
  \begin{aligned}
  \delta_x&=\mu_x+4\beta_x\mu_{xy}^2,\\
  \delta_y&=\mu_y+4\beta_y\mu_{yx}^2,
  \end{aligned}
\]
we have
\[
  c_x\geq\frac{\delta_x}{16},
  \qquad
  c_y\geq\frac{\delta_y}{16}.
\]

Therefore,
\[
  \mathcal E
  \geq \frac34(D_{f,o}+D_{g,o})-\frac12(D_{f,i}+D_{g,i})
  +\frac{1}{16}\|z_{\mathrm{out}}-z^*\|_{\bP}^2 .
\]

Let
\[
  R_{\mathrm{in}}^2
  :=
  \|z_{\mathrm{in}}-z^*\|_{\bP}^2,
  \qquad
  A_i:=\prod_{j=1}^iT_j,
\]
so that \(A_i\) is the number of evaluations of the component placed at
level \(i\), and set
\[
  \bar N_f:=A_{\pi^{-1}(1)},
  \qquad
  \bar N_g:=A_{\pi^{-1}(2)},
  \qquad
  \bar N_B:=A_{\pi^{-1}(3)}.
\]
Thus \(\bar N_f,\bar N_g,\bar N_B\) are the component-oracle call budgets,
irrespective of the levels to which \(\pi\) assigns the components. Applying
Theorem~\ref{thm:VI_sliding_inexact} with comparator \(z=z^*\), we obtain
\[
  \Theta
  :=\frac{\kappa_x}{\bar N_f^2}+\frac{\kappa_y}{\bar N_g^2}
  +\frac{\kappa_{xy}}{\bar N_B^2}
  +\frac{\sqrt{\kappa_{xy}}}{\bar N_B}.
\]
\begin{align}
  \mathcal E
  &\leq
  C_{\mathrm{VI}}
  \Theta R_{\mathrm{in}}^2
  +
  \bar N_f\delta_f
  +
  \bar N_g\delta_g ,
  \label{eq:SPP_gap_upper_compact_constant}
\end{align}
where \(C_{\mathrm{VI}}>0\) absorbs the level-dependent factors
\(2^{2i-1}\) and \(2^{i-1}\) from
Theorem~\ref{thm:VI_sliding_inexact}. These factors are at most \(2^5\) for
\(n=3\), uniformly over \(\pi\), so \(\Theta\) depends only on the component
budgets \(\bar N_f,\bar N_g,\bar N_B\).

Combining the upper bound~\eqref{eq:SPP_gap_upper_compact_constant} with the
lower bound obtained above, the oracle-error terms are
\[
  \bar N_f\delta_f+\bar N_g\delta_g .
\]

We now choose \(\bar N_f,\bar N_g,\bar N_B\) and \(\delta_f,\delta_g\) so that, for a
numerical constant \(\rho>0\),
\begin{equation*}
  C_{\mathrm{VI}}
  \Theta
  \leq
  \rho
\end{equation*}
and
\begin{equation*}
  \bar N_f\delta_f
  \leq
  \rho \Omega_{\mathrm{in}},
  \qquad
  \bar N_g\delta_g
  \leq
  \rho \Omega_{\mathrm{in}}.
\end{equation*}
Since \(R_{\mathrm{in}}^2\leq \Psi(\zin)\leq \Omega_{\mathrm{in}}\), we get
\begin{align*}
  \mathcal E
  &\leq
  \rho R_{\mathrm{in}}^2
  +
  2\rho\Omega_{\mathrm{in}} .
\end{align*}

Hence, combining this upper bound with the lower bound obtained above,
\begin{align*}
  &\frac{1}{16}\|\zout-z^*\|_{\bP}^2
  +\frac34\bigl(\D_f(\xout,x^*)+\D_g(\yout,y^*)\bigr)\\
  &\leq \rho R_{\mathrm{in}}^2+2\rho\Omega_{\mathrm{in}}
  +\frac12\bigl(\D_f(\xin,x^*)+\D_g(\yin,y^*)\bigr).
\end{align*}

Multiplying both sides by \(16\) gives
\begin{align*}
  &\|\zout-z^*\|_{\bP}^2
  +12\bigl(\D_f(\xout,x^*)+\D_g(\yout,y^*)\bigr)\\
  &\leq 16\rho R_{\mathrm{in}}^2+32\rho\Omega_{\mathrm{in}}
  +8\bigl(\D_f(\xin,x^*)+\D_g(\yin,y^*)\bigr).
\end{align*}

Since
\[
  \Psi(\zin)
  =
  R_{\mathrm{in}}^2
  +
  12\D_f(\xin,x^*)
  +
  12\D_g(\yin,y^*)
  \leq
  \Omega_{\mathrm{in}},
\]
and with \(16\rho\leq 2/3\), we have
\begin{align*}
  16\rho R_{\mathrm{in}}^2
  +
  8(D_{f,i}+D_{g,i})
  &\leq
  \frac23
  \left[
    R_{\mathrm{in}}^2+12(D_{f,i}+D_{g,i})
  \right]\\
  &=
  \frac23\Psi(\zin)
  \leq
  \frac23\Omega_{\mathrm{in}}.
\end{align*}
Therefore,
\[
  \Psi(\zout)
  \leq
  \left(
    \frac23+32\rho
  \right)
  \Omega_{\mathrm{in}}.
\]
Choosing
\[
  \rho=\frac{1}{384},
\]
we obtain
\[
  \frac23+32\rho
  =
  \frac23+\frac1{12}
  =
  \frac34.
\]
Consequently,
\[
  \Psi(\zout)
  \leq
  \frac34\Omega_{\mathrm{in}}.
\]

It remains to verify the claimed evaluation bounds. We use the same balancing
argument as in the proof of Theorem~\ref{thm:VI_sliding_Holder}.

We first consider the \(x\)-component. Suppose \(0\leq\nu_x<1\), and define
\[
  a_x:=\frac{1-\nu_x}{1+\nu_x},
  \qquad
  L_x(\delta_f)=K_x\delta_f^{-a_x},
\]
where
\[
  K_x
  =
  \left[
    \frac{1-\nu_x}{2(1+\nu_x)}
  \right]^{\frac{1-\nu_x}{1+\nu_x}}
  H_x^{\frac{2}{1+\nu_x}}.
\]

The two conditions that need to be satisfied are
\[
  C_{\mathrm{VI}}\frac{\kappa_x}{\bar N_f^2}
  =
  C_{\mathrm{VI}}\frac{L_x(\delta_f)}{\delta_x\bar N_f^2}
  \leq
  \frac{\rho}{4},
  \qquad
  \bar N_f\delta_f
  \leq
  \rho\Omega_{\mathrm{in}}.
\]

Choose
\[
  \delta_f := \frac{\rho\Omega_{\mathrm{in}}}{\bar N_f}.
\]
Then the oracle-error condition
\(\bar N_f\delta_f\leq\rho\Omega_{\mathrm{in}}\)
holds with equality. Moreover,
\[
  L_x(\delta_f)
  =
  K_x
  \left(
    \frac{\rho\Omega_{\mathrm{in}}}{\bar N_f}
  \right)^{-a_x}
  =
  K_x\rho^{-a_x}\Omega_{\mathrm{in}}^{-a_x}\bar N_f^{a_x}.
\]
Therefore, the first condition becomes
\[
  C_{\mathrm{VI}}
  \frac{
    K_x\rho^{-a_x}\Omega_{\mathrm{in}}^{-a_x}\bar N_f^{a_x}
  }{
    \delta_x\bar N_f^2
  }
  \leq
  \frac{\rho}{4}.
\]
Equivalently,
\[
  \bar N_f^{2-a_x}
  \geq
  \frac{4C_{\mathrm{VI}}}{\rho^{1+a_x}}
  \frac{K_x}{\delta_x}
  \Omega_{\mathrm{in}}^{-a_x}.
\]
Thus it is enough to choose
\[
  \bar N_f
  \geq
  \left(
    \frac{4 C_{\mathrm{VI}}}{\rho^{1+a_x}}
    \frac{K_x}{\delta_x}
    \Omega_{\mathrm{in}}^{-a_x}
  \right)^{\frac{1}{2-a_x}}.
\]
Since
\[
  2-a_x
  =
  2-\frac{1-\nu_x}{1+\nu_x}
  =
  \frac{1+3\nu_x}{1+\nu_x},
\]
we obtain
\[
  \bar N_f
  =
  \mathcal O\left(
    K_x^{\frac{1+\nu_x}{1+3\nu_x}}
    \delta_x^{-\frac{1+\nu_x}{1+3\nu_x}}
    \Omega_{\mathrm{in}}^{\frac{\nu_x-1}{1+3\nu_x}}
  \right).
\]

Using the definition of \(K_x\), this becomes
\[
  \bar N_f
  =
  \mathcal O\left(
    H_x^{\frac{2}{1+3\nu_x}}
    \delta_x^{-\frac{1+\nu_x}{1+3\nu_x}}
    \Omega_{\mathrm{in}}^{\frac{\nu_x-1}{1+3\nu_x}}
  \right).
\]
Equivalently, with
\[
  \tilde H_x:=H_x\delta_x^{-(1+\nu_x)/2},
\]
we have
\[
  \bar N_f
  =
  \mathcal O\left(
    \tilde H_x^{\frac{2}{1+3\nu_x}}
    \Omega_{\mathrm{in}}^{\frac{\nu_x-1}{1+3\nu_x}}
  \right).
\]

When \(\nu_x=1\), we take \(\delta_f=0\) and \(L_x=H_x\). Then the only
condition involving \(\bar N_f\) is
\[
  C_{\mathrm{VI}}\frac{\kappa_x}{\bar N_f^2}
  =
  C_{\mathrm{VI}}\frac{L_x}{\delta_x\bar N_f^2}
  \leq
  \frac{\rho}{4},
\]
which is satisfied by
\[
  \bar N_f
  =
  \mathcal O\left(\sqrt{\frac{L_x}{\delta_x}}\right)
  =
  \mathcal O(\sqrt{\kappa_x}).
\]

The same two cases, with
\(
  C_{\mathrm{VI}}\kappa_y/\bar N_g^2\leq\rho/4
\),
give
\[
  \bar N_g
  =
  \mathcal O\left(
    \tilde H_y^{\frac{2}{1+3\nu_y}}
    \Omega_{\mathrm{in}}^{\frac{\nu_y-1}{1+3\nu_y}}
  \right),
  \qquad
  \tilde H_y:=H_y\delta_y^{-(1+\nu_y)/2}.
\]

Finally, the \(B\)-dependent terms require
\[
  C_{\mathrm{VI}}\frac{\kappa_{xy}}{\bar N_B^2}\leq\frac{\rho}{4},
  \qquad
  C_{\mathrm{VI}}\frac{\sqrt{\kappa_{xy}}}{\bar N_B}
  \leq\frac{\rho}{4}.
\]
These are satisfied by
\[
  \bar N_B
  =
  \mathcal O\left(\sqrt{\kappa_{xy}}\right).
\]

The preceding estimates hold once component \(m\) receives \(cR_m\)
evaluations for some constant \(c\geq1\). Let \(\pi\) order the budgets as
\(
  R_{\pi(1)}\leq R_{\pi(2)}\leq R_{\pi(3)}
\),
matching the nondecreasing cumulative counts
\(A_i:=\prod_{j\leq i}T_j\). Set \(A_0:=1\) and
\[
  T_i
  :=
  \max\left\{
    1,
    \left\lceil \frac{c\,R_{\pi(i)}}{A_{i-1}}\right\rceil
  \right\},
  \qquad
  A_i:=A_{i-1}T_i .
\]
The induction in Appendix~\ref{app:proof_VI_Holder_complexity} gives
\[
  c\,R_{\pi(i)}\leq A_i\leq 2c\,R_{\pi(i)},
  \qquad i=1,2,3 ,
\]
because every \(R_m\geq1\). Hence
\[
  \bar N_f=\mathcal O(R_1),
  \qquad
  \bar N_g=\mathcal O(R_2),
  \qquad
  \bar N_B=\mathcal O(R_3).
\]
Accounting for one anchor-gradient call of each type in \(p_3\) and a
constant number of \(\bB,\bB^\top\) products per \(p_3,Q_3\) call, the
evaluation counts satisfy, for an absolute constant \(C_B\),
\[
  N_f\leq\bar N_f+1,
  \qquad
  N_g\leq\bar N_g+1,
  \qquad
  N_B\leq C_B(\bar N_B+1).
\]
Since \(R_1,R_2,R_3\geq1\), the claimed bounds for \(N_f,N_g,N_B\) follow.
  \end{proof}

\section{Proof of Corollary~\ref{cor:bilinear_SPP_restart}}
\label{app:proof_bilinear_SPP_restart}

\begin{proof}
By Remark~\ref{rem:restart_slice}, every restart uses the function \(\Psi\)
from Section~\ref{sec:bilinear_SPP} and the same representative \(z^*\), so
the one-step inequalities share a common reference point.

The case \(S=0\) is immediate. Assume \(S\geq1\).

Let \(z^s\) be the point at the beginning of restart \(s\), and write
\[
  \Psi_s:=\Psi(z^s),
  \qquad
  \Psi_0:=\Psi(z^0),
  \qquad
  \Omega_s:=\left(\frac34\right)^s\Psi_0 .
\]
We prove by induction that
\[
  \Psi_s\leq \Omega_s,
  \qquad s=0,1,\dots,S.
\]
The claim is true for \(s=0\). If \(\Psi_s\leq\Omega_s\), then
Theorem~\ref{thm:bilinear_SPP_Holder}, applied with
\(\Omega_{\mathrm{in}}=\Omega_s\), gives
\[
  \Psi_{s+1}
  \leq
  \frac34\Omega_s
  =
  \Omega_{s+1}.
\]
Therefore,
\[
  \Psi_S
  \leq
  \Omega_S
  =
  \left(\frac34\right)^S\Psi_0
  \leq
  \varepsilon
\]
by the definition of \(S\).

We now bound the total number of evaluations. For fixed problem parameters
and every \(\alpha>0\),
\[
  1+\log(\Psi_0/\varepsilon)
  =
  \mathcal O(\varepsilon^{-\alpha})
  \qquad
  \text{as }\varepsilon\downarrow0.
\]
First consider the \(f\)-component. If \(0\leq\nu_x<1\), define
\[
  \alpha_x:=\frac{1-\nu_x}{1+3\nu_x}>0.
\]
At restart \(s\), Theorem~\ref{thm:bilinear_SPP_Holder} gives
\[
  N_{f,s}
  =
  \mathcal O\left(
    \max\left\{
      \tilde H_x^{\frac{2}{1+3\nu_x}}
      \Omega_s^{-\alpha_x},1
    \right\}
  \right).
\]
Hence
\begin{align*}
  N_f^{\mathrm{tot}}
  &=
  \sum_{s=0}^{S-1}N_{f,s}                                                   \\
  &=
  \mathcal O\left(
    \tilde H_x^{\frac{2}{1+3\nu_x}}
    \sum_{s=0}^{S-1}\Omega_s^{-\alpha_x}
    +S
  \right)                                                                  \\
  &=
  \mathcal O\left(
    \tilde H_x^{\frac{2}{1+3\nu_x}}
    \Psi_0^{-\alpha_x}
    \sum_{s=0}^{S-1}
    \left(\frac34\right)^{-\alpha_x s}
    +S
  \right).
\end{align*}

By the minimality of \(S\),
\[
  \varepsilon<\Omega_{S-1}\leq\frac43\varepsilon.
\]
The geometric sum is therefore bounded by
\[
  \sum_{s=0}^{S-1}\Omega_s^{-\alpha_x}
  \leq
  \Omega_{S-1}^{-\alpha_x}
  \sum_{j=0}^{\infty}\left(\frac34\right)^{\alpha_xj}
  =
  \mathcal O\left(\varepsilon^{-\alpha_x}\right).
\]
Thus
\[
  N_f^{\mathrm{tot}}
  =
  \mathcal O\left(
    \tilde H_x^{\frac{2}{1+3\nu_x}}
    \varepsilon^{-\alpha_x}
    +S
  \right)
  .
\]
Because \(\alpha_x>0\), this becomes
\[
  N_f^{\mathrm{tot}}
  =
  \mathcal O\left(
    \tilde H_x^{\frac{2}{1+3\nu_x}}
    \varepsilon^{\frac{\nu_x-1}{1+3\nu_x}}
  \right).
\]

If \(\nu_x=1\), then each restart requires
\[
  N_{f,s}
  =
  \mathcal O(\sqrt{\kappa_x}),
\]
and hence
\[
  N_f^{\mathrm{tot}}
  =
  \sum_{s=0}^{S-1}N_{f,s}
  =
  \mathcal O\left(
    \sqrt{\kappa_x}
    \left(1+\log\frac{\Psi_0}{\varepsilon}\right)
  \right).
\]

The proof for the \(g\)-component is identical. If
\(0\leq\nu_y<1\), with
\[
  \alpha_y:=\frac{1-\nu_y}{1+3\nu_y},
\]
we obtain
\[
  N_g^{\mathrm{tot}}
  =
  \mathcal O\left(
    \tilde H_y^{\frac{2}{1+3\nu_y}}
    \varepsilon^{-\alpha_y}
    +S
  \right).
\]
Because \(\alpha_y>0\), this becomes
\[
  N_g^{\mathrm{tot}}
  =
  \mathcal O\left(
    \tilde H_y^{\frac{2}{1+3\nu_y}}
    \varepsilon^{\frac{\nu_y-1}{1+3\nu_y}}
  \right).
\]
If \(\nu_y=1\), then
\[
  N_g^{\mathrm{tot}}
  =
  \mathcal O\left(
    \sqrt{\kappa_y}
    \left(1+\log\frac{\Psi_0}{\varepsilon}\right)
  \right).
\]

Finally, at each restart,
\[
  N_{B,s}
  =
  \mathcal O\left(\max\{\sqrt{\kappa_{xy}},1\}\right).
\]
Therefore
\[
  N_B^{\mathrm{tot}}
  =
  \sum_{s=0}^{S-1}N_{B,s}
  =
  \mathcal O\left(
    \max\{\sqrt{\kappa_{xy}},1\}
    \left(1+\log\frac{\Psi_0}{\varepsilon}\right)
  \right).
\]
This proves the corollary.
\end{proof}

\section{The Degenerate and Mixed Cases} \label{app:mixed_cases}

The degenerate regime \(\delta_x=\delta_y=0\) of
Section~\ref{sec:bilinear_SPP} yields the complexities collected in
Table~\ref{tab:bilinear_SPP_degenerate_complexity}.

\begin{table*}[h]
  \centering
  \begin{tabular}{ll}
    \toprule
    Oracle
    & Complexity \\
    \midrule
    Evaluations of \(f'\)
    &
    \(\displaystyle
      \mathcal O\!\left(
        \left(
          \frac{H_x\Omega^{(1+\nu_x)/2}}{\varepsilon}
        \right)^{\frac{2}{1+3\nu_x}}
      \right)
    \)
    \\[2ex]

    Evaluations of \(g'\)
    &
    \(\displaystyle
      \mathcal O\!\left(
        \left(
          \frac{H_y\Omega^{(1+\nu_y)/2}}{\varepsilon}
        \right)^{\frac{2}{1+3\nu_y}}
      \right)
    \)
    \\[2ex]

    Products with \(\bB,\bB^\top\)
    &
    \(\displaystyle
      \mathcal O\!\left(
        \frac{L_{xy}\Omega}{\varepsilon}
      \right)
    \)
    \\
    \bottomrule
  \end{tabular}
  \caption{Evaluation complexity for the bilinear saddle-point problem
  in the degenerate case \(\delta_x=\delta_y=0\). Here
  \(\Omega := \sup_{z\in\cZ}\|z^0-z\|^2\), and the bounds hold for
  \(0\leq\nu_x,\nu_y\leq1\).}
  \label{tab:bilinear_SPP_degenerate_complexity}
\end{table*}

\paragraph{The mixed cases
\(\delta_x>0,\delta_y=0\) and
\(\delta_x=0,\delta_y>0\).}

In the mixed regimes, exactly one effective strong-convexity parameter is
positive. We regularize the zero-curvature block and rebuild the active part
of \(p_3\) with the regularized constant and anchor.

By symmetry, it suffices to consider
\[
  \delta_x>0,
  \qquad
  \delta_y=0.
\]
Because
\(\delta_y=\mu_y+4\beta_y\mu_{yx}^2=0\), nonnegativity implies
\[
  \mu_y=0,
  \qquad
  \beta_y\mu_{yx}^2=0.
\]
Assume that \(\cY\) is bounded, fix \(y^0\in\R^{d_y}\), and define
\[
  \Omega_y
  :=
  \sup_{y\in\cY}\|y-y^0\|^2
  <\infty,
  \qquad
  D_{\cY}
  :=
  \sup_{u,v\in\cY}\|u-v\|
  \leq2\sqrt{\Omega_y}.
\]
If \(\Omega_y=0\), then \(\cY=\{y^0\}\) and the \(y\)-block can be
eliminated. Hence assume \(\Omega_y>0\).
For a target accuracy \(\varepsilon>0\), set
\[
  \lambda_y
  :=
  \frac{\varepsilon}{\Omega_y},
\]
and replace \(g\) by
\[
  g_{\lambda}(y)
  :=
  g(y)
  +
  \frac{\lambda_y}{2}\|y-y^0\|^2.
\]

For \(u,v\in\cY\),
\[
  \lambda_y\|u-v\|
  \leq
  \lambda_yD_{\cY}^{1-\nu_y}\|u-v\|^{\nu_y}.
\]
Consequently, \(g_\lambda\) is \(\lambda_y\)-strongly convex and satisfies
the \((\nu_y,H_{y,\lambda})\)-H\"older condition with
\begin{equation}\label{eq:mixed_H_y_lambda}
  H_{y,\lambda}
  :=
  H_y+\lambda_yD_{\cY}^{1-\nu_y}.
\end{equation}
For \(\nu_y=1\), the function \(g_\lambda\) is smooth on the ambient space
with constant \(L_{y,\lambda}=L_y+\lambda_y\). In particular,
\eqref{eq:ambient_self_bounding} holds for \(g_\lambda\) with this constant.
The coefficient of the \(x\)-part of the coupling regularizer must therefore
be recomputed:
\begin{equation}\label{eq:mixed_beta_lambda}
  \beta_x^\lambda
  :=
  \begin{cases}
    \dfrac{1}{4(L_y+\lambda_y)},&\nu_y=1,\\[1ex]
    0,&0\leq\nu_y<1,
  \end{cases}
  \qquad
  \beta_y^\lambda:=\beta_y.
\end{equation}
The effective strong convexity parameters of the regularized problem are
\begin{equation}\label{eq:mixed_delta_lambda}
  \begin{aligned}
  \delta_x^\lambda
  &:=
  \mu_x+4\beta_x^\lambda\mu_{xy}^2,\\
  \delta_y^\lambda
  &:=
  \mu_y+\lambda_y+4\beta_y^\lambda\mu_{yx}^2
  =
  \lambda_y.
  \end{aligned}
\end{equation}
The last equality uses \(\delta_y=0\). Thus
\(\delta_x^\lambda=\delta_x\) when \(\nu_y<1\), but generally not when
\(\nu_y=1\).

At every restart with input
\(z_{\mathrm{in}}=(x_{\mathrm{in}},y_{\mathrm{in}})\), use
\[
  g_\lambda'(y_{\mathrm{in}})
  =
  g'(y_{\mathrm{in}})
  +
  \lambda_y(y_{\mathrm{in}}-y^0)
\]
and replace~\eqref{eq:SPP_decomposition} by
\begin{equation}\label{eq:mixed_p3_lambda}
  \begin{aligned}
  p_{3,\lambda}(x,y)
  &:=
  \frac{\beta_x^\lambda}{2}
  \|\bB x-g_\lambda'(y_{\mathrm{in}})\|^2\\
  &\quad+
  \frac{\beta_y^\lambda}{2}
  \|\bB^\top y+f'(x_{\mathrm{in}})\|^2.
  \end{aligned}
\end{equation}
The other components are \(p_1=f\), \(p_2=g_\lambda\), and the same
bilinear operator \(Q_3\) as in~\eqref{eq:SPP_operator}. The norm and
condition parameters used by Theorem~\ref{thm:bilinear_SPP_Holder}, which
requires positive effective strong convexity parameters, are
\[
  \bP_\lambda
  =
  \diag(\delta_x^\lambda\bI_{d_x},
        \delta_y^\lambda\bI_{d_y}),
\]
\[
  \kappa_x^\lambda
  =
  \frac{L_x}{\delta_x^\lambda},
  \qquad
  \kappa_{xy}^\lambda
  =
  \frac{L_{xy}^2}{\delta_x^\lambda\lambda_y}.
\]

Assume that the regularized problem has a saddle point
\(z_\lambda^*=(x_\lambda^*,y_\lambda^*)\), and choose subgradients and
normal vectors satisfying
\[
  f'(x_\lambda^*)+\bB^\top y_\lambda^*+\xi_\lambda=0,
  \qquad
  g_\lambda'(y_\lambda^*)-\bB x_\lambda^*+\zeta_\lambda=0.
\]
Applying Theorem~\ref{thm:bilinear_SPP_Holder} requires
\begin{equation}\label{eq:mixed_regularized_compatibility}
  \begin{gathered}
  \zeta_\lambda=0
  \quad\text{if }\beta_x^\lambda>0,
  \qquad
  \xi_\lambda=0
  \quad\text{if }\beta_y^\lambda>0,\\
  \beta_x^\lambda\delta_y^\lambda\leq\frac14,
  \qquad
  \beta_y^\lambda\delta_x^\lambda\leq\frac14.
  \end{gathered}
\end{equation}
The first inequality follows from
\(\delta_y^\lambda=\lambda_y\) and
\(\beta_x^\lambda\lambda_y\leq1/4\). The second is automatic when either
component is nonsmooth. In the fully smooth case it is retained as a
compatibility assumption. We set \(\mu_{yx}=0\) on the regularized
\(y\)-block.

Let \(\Psi_\lambda\) be the Lyapunov function~\eqref{eq:Psi_s} formed with
\(\bP_\lambda\), \(g_\lambda\), and \(z_\lambda^*\).
Apply Corollary~\ref{cor:bilinear_SPP_restart} to
\((f,g_\lambda,p_{3,\lambda},Q_3)\), with terminal target
\(\varepsilon/2\). Its output \(z_\lambda^S\) satisfies
\[
  \Psi_\lambda(z_\lambda^S)\leq\frac{\varepsilon}{2}.
\]
Define
\[
  \tilde H_x^\lambda
  :=
  H_x(\delta_x^\lambda)^{-(1+\nu_x)/2}.
\]
The direct bound for the \(f\)-oracle is
\[
  N_f
  =
  \begin{cases}
  \displaystyle
  \mathcal O\!\left(
    (\tilde H_x^\lambda)^{\frac{2}{1+3\nu_x}}
    \varepsilon^{-\frac{1-\nu_x}{1+3\nu_x}}
  \right),
  & 0\leq\nu_x<1,\\[2ex]
  \displaystyle
  \widetilde{\mathcal O}\!\left(
    \sqrt{\kappa_x^\lambda}
  \right),
  & \nu_x=1,
  \end{cases}
  .
\]
For the regularized \(g\)-oracle and the bilinear operator, the direct bounds
are
\begin{equation}\label{eq:mixed_Ng_corrected}
  N_g
  =
  \begin{cases}
  \displaystyle
  \mathcal O\!\left(
    \left(
      \frac{
        H_{y,\lambda}\Omega_y^{(1+\nu_y)/2}
      }{\varepsilon}
    \right)^{\frac{2}{1+3\nu_y}}
  \right),
  &0\leq\nu_y<1,\\[2.2ex]
  \displaystyle
  \widetilde{\mathcal O}\!\left(
    \sqrt{1+\frac{L_y\Omega_y}{\varepsilon}}
  \right),
  &\nu_y=1,
  \end{cases}
\end{equation}
and
\begin{equation}\label{eq:mixed_NB_corrected}
  N_B
  =
  \widetilde{\mathcal O}\!\left(
    L_{xy}
    \sqrt{
      \frac{\Omega_y}{\delta_x^\lambda\varepsilon}
    }
  \right).
\end{equation}

For any saddle function \(H\) on \(\cX\times\cY\), define
\[
  \operatorname{Gap}_H(x,y)
  :=
  \max_{v\in\cY}H(x,v)-\min_{u\in\cX}H(u,y).
\]
Write \(\operatorname{Gap}:=\operatorname{Gap}_F\) and
\(\operatorname{Gap}_{\lambda_y}:=\operatorname{Gap}_{F_{\lambda_y}}\), where
\[
  F_{\lambda_y}(x,y)
  :=
  F(x,y)-\frac{\lambda_y}{2}\|y-y^0\|^2.
\]
The exact
regularized saddle has
\(\operatorname{Gap}_{\lambda_y}(z_\lambda^*)=0\) and satisfies
\[
  \operatorname{Gap}(z_\lambda^*)
  \leq
  \operatorname{Gap}_{\lambda_y}(z_\lambda^*)
  +
  \frac{\lambda_y}{2}\Omega_y
  =
  \frac{\varepsilon}{2}.
\]
Thus \(z_\lambda^*\) is an \(\varepsilon/2\)-weak-gap solution of the
original problem. It remains to transfer the Lyapunov guarantee for
\(z_\lambda^S\) to a weak-gap bound for the original
problem~\eqref{prob:bilinear_SPP_main}.

Recall
\[
  F_{\lambda_y}(x,y)
  =F(x,y)-\frac{\lambda_y}{2}\|y-y^0\|^2
  =f(x)+\langle y,\bB x\rangle-g_\lambda(y),
\]
and set
\[
  \varphi_\lambda(x):=\max_{v\in\cY}F_{\lambda_y}(x,v),
  \qquad
  \psi_\lambda(y):=\min_{u\in\cX}F_{\lambda_y}(u,y),
\]
so that \(\operatorname{Gap}_{\lambda_y}(z)=\varphi_\lambda(x)-\psi_\lambda(y)\)
and \(\varphi_\lambda(x_\lambda^*)=\psi_\lambda(y_\lambda^*)
=F_{\lambda_y}(z_\lambda^*)\).

Define
\begin{equation}\label{eq:mixed_conjugates}
  h^*(w):=\max_{v\in\cY}\bigl\{\langle w,v\rangle-g_\lambda(v)\bigr\},
  \qquad
  q^*(w'):=\max_{u\in\cX}\bigl\{\langle w',u\rangle-f(u)\bigr\},
\end{equation}
so that \(\varphi_\lambda(x)=f(x)+h^*(\bB x)\) and
\(\psi_\lambda(y)=-q^*(-\bB^\top y)-g_\lambda(y)\). The first maximum is
finite because \(\cY\) is bounded, the second by
Assumption~\ref{ass:mixed_gap_transfer}(ii), which is in force from here on.
We write
\[
  \D_{h^*}(w,w')
  :=h^*(w)-h^*(w')-\langle\nabla h^*(w'),w-w'\rangle,
\]
and, for the fixed choice \(x_\lambda^*\in\partial q^*(-\bB^\top y_\lambda^*)\)
supplied by Lemma~\ref{lem:mixed_conjugate_smoothness},
\[
  \D_{q^*}(w,-\bB^\top y_\lambda^*)
  :=q^*(w)-q^*(-\bB^\top y_\lambda^*)
   -\langle x_\lambda^*,w+\bB^\top y_\lambda^*\rangle .
\]

\begin{assumption}\label{ass:mixed_gap_transfer}
  In addition to~\eqref{eq:mixed_regularized_compatibility}:
  \begin{enumerate}
    \item[(i)] \(\xi_\lambda=0\) and \(\zeta_\lambda=0\).
    \item[(ii)] either \(\mu_x>0\), or \(\cX\) is bounded, in which case we set
      \(D_{\cX}:=\sup_{u,u'\in\cX}\|u-u'\|<\infty\).
  \end{enumerate}
\end{assumption}

\begin{lemma}\label{lem:mixed_conjugate_smoothness}
  Let Assumption~\ref{ass:mixed_gap_transfer}(ii) hold, so that \(q^*\) is
  finite. The function \(h^*\) is differentiable
  on \(\R^{d_y}\), with
  \(\nabla h^*(w)=\argmax_{v\in\cY}\{\langle w,v\rangle-g_\lambda(v)\}\) and
  \(1/\lambda_y\)-Lipschitz gradient. In particular
  \begin{equation}\label{eq:mixed_Dh_bound}
    \D_{h^*}(w,w')\leq\frac{\|w-w'\|^2}{2\lambda_y},
    \qquad
    \nabla h^*(\bB x_\lambda^*)=y_\lambda^* .
  \end{equation}
  Moreover \(x_\lambda^*\in\partial q^*(-\bB^\top y_\lambda^*)\), and for every
  \(y\in\cY\),
  \begin{equation}\label{eq:mixed_Dq_variational}
    \D_{q^*}\bigl(-\bB^\top y,-\bB^\top y_\lambda^*\bigr)
    \leq
    \max_{u\in\cX}
    \Bigl\{
      \bigl\langle-\bB^\top(y-y_\lambda^*),\,u-x_\lambda^*\bigr\rangle
      -\D_f(u,x_\lambda^*)
    \Bigr\}.
  \end{equation}
\end{lemma}

\begin{proof}
  The function equal to \(g_\lambda\) on \(\cY\) and to \(+\infty\) elsewhere
  is proper, closed and \(\lambda_y\)-strongly convex, and \(h^*\) is its
  Fenchel conjugate. Hence \(h^*\) is differentiable with
  \(1/\lambda_y\)-Lipschitz gradient given by the unique maximizer, and the
  first part of~\eqref{eq:mixed_Dh_bound} is the descent lemma for \(h^*\).
  Since \(z_\lambda^*\) is a saddle point of \(F_{\lambda_y}\), the point
  \(y_\lambda^*\) maximizes \(v\mapsto\langle\bB x_\lambda^*,v\rangle
  -g_\lambda(v)\) over \(\cY\), which is the second part. Likewise
  \(x_\lambda^*\) minimizes \(u\mapsto f(u)+\langle\bB^\top y_\lambda^*,u
  \rangle\) over \(\cX\), i.e.\ \(x_\lambda^*\in\partial
  q^*(-\bB^\top y_\lambda^*)\).

  For~\eqref{eq:mixed_Dq_variational}, write \(w:=-\bB^\top y\) and
  \(w':=-\bB^\top y_\lambda^*\). Using
  \(q^*(w')=\langle w',x_\lambda^*\rangle-f(x_\lambda^*)\),
  \[
    \D_{q^*}(w,w')
    =\max_{u\in\cX}
      \bigl\{\langle w,u-x_\lambda^*\rangle
        -\bigl(f(u)-f(x_\lambda^*)\bigr)\bigr\}.
  \]
  By the regularized optimality conditions,
  \(f'(x_\lambda^*)=w'-\xi_\lambda\) with
  \(\xi_\lambda\in N_{\cX}(x_\lambda^*)\), so for every \(u\in\cX\),
  \[
    f(u)-f(x_\lambda^*)
    =\D_f(u,x_\lambda^*)+\langle w',u-x_\lambda^*\rangle
     -\langle\xi_\lambda,u-x_\lambda^*\rangle
    \geq\D_f(u,x_\lambda^*)+\langle w',u-x_\lambda^*\rangle,
  \]
  because \(\langle\xi_\lambda,u-x_\lambda^*\rangle\leq0\). Substituting gives
  the claim.
\end{proof}

\begin{lemma}\label{lem:mixed_gap_identity}
  Let Assumption~\ref{ass:mixed_gap_transfer}(ii) hold. For every
  \(z=(x,y)\in\cZ\),
  \begin{equation}\label{eq:mixed_gap_identity}
    \begin{aligned}
    \operatorname{Gap}_{\lambda_y}(z)
    &=\D_f(x,x_\lambda^*)
      +\D_{h^*}(\bB x,\bB x_\lambda^*)
      -\langle\xi_\lambda,x-x_\lambda^*\rangle\\
    &\quad+\D_{g_\lambda}(y,y_\lambda^*)
      +\D_{q^*}(-\bB^\top y,-\bB^\top y_\lambda^*)
      -\langle\zeta_\lambda,y-y_\lambda^*\rangle .
    \end{aligned}
  \end{equation}
\end{lemma}

\begin{proof}
  Since \(\varphi_\lambda=f+h^*\circ\bB\),
  Lemma~\ref{lem:mixed_conjugate_smoothness} gives
  \begin{align*}
    \varphi_\lambda(x)-\varphi_\lambda(x_\lambda^*)
    &=\D_f(x,x_\lambda^*)+\langle f'(x_\lambda^*),x-x_\lambda^*\rangle\\
    &\quad+\D_{h^*}(\bB x,\bB x_\lambda^*)
      +\langle y_\lambda^*,\bB(x-x_\lambda^*)\rangle,
  \end{align*}
  and the two linear terms combine into
  \(\langle f'(x_\lambda^*)+\bB^\top y_\lambda^*,x-x_\lambda^*\rangle
  =-\langle\xi_\lambda,x-x_\lambda^*\rangle\). Similarly,
  \begin{align*}
    \psi_\lambda(y_\lambda^*)-\psi_\lambda(y)
    &=q^*(-\bB^\top y)-q^*(-\bB^\top y_\lambda^*)
      +g_\lambda(y)-g_\lambda(y_\lambda^*)\\
    &=\D_{q^*}(-\bB^\top y,-\bB^\top y_\lambda^*)
      -\langle\bB x_\lambda^*,y-y_\lambda^*\rangle\\
    &\quad+\D_{g_\lambda}(y,y_\lambda^*)
      +\langle g_\lambda'(y_\lambda^*),y-y_\lambda^*\rangle,
  \end{align*}
  whose linear terms combine into
  \(\langle g_\lambda'(y_\lambda^*)-\bB x_\lambda^*,y-y_\lambda^*\rangle
  =-\langle\zeta_\lambda,y-y_\lambda^*\rangle\). Adding the two displays and
  using \(\varphi_\lambda(x_\lambda^*)=\psi_\lambda(y_\lambda^*)\) yields
  \eqref{eq:mixed_gap_identity}. The Bregman terms are nonnegative by
  convexity, and the two remaining terms are nonnegative because
  \(\xi_\lambda\in N_{\cX}(x_\lambda^*)\) with \(x\in\cX\), and
  \(\zeta_\lambda\in N_{\cY}(y_\lambda^*)\) with \(y\in\cY\).
\end{proof}

\begin{theorem}\label{thm:mixed_gap_transfer}
  Let Assumption~\ref{ass:mixed_gap_transfer} hold and set
  \begin{equation}\label{eq:mixed_kappa_hat}
    \hat\kappa:=\frac{L_{xy}^2}{\delta_x^\lambda\lambda_y}
      =\frac{L_{xy}^2\Omega_y}{\delta_x^\lambda\varepsilon},
    \qquad
    \hat\kappa_\mu:=\frac{L_{xy}^2}{\mu_x\lambda_y}
      =\frac{L_{xy}^2\Omega_y}{\mu_x\varepsilon}
    \ \ (\mu_x>0).
  \end{equation}
  If \(z\in\cZ\) satisfies \(\Psi_\lambda(z)\leq\varepsilon'\), then
  \begin{equation}\label{eq:mixed_gap_transfer}
    \operatorname{Gap}(z)
    \leq
    \frac{\varepsilon}{2}
    +\frac{\varepsilon'}{6}
    +\frac{\hat\kappa}{2}\varepsilon'
    +
    \begin{cases}
      \dfrac{\hat\kappa_\mu}{2}\,\varepsilon', & \mu_x>0,\\[1.8ex]
      L_{xy}D_{\cX}\sqrt{\dfrac{\varepsilon'}{\lambda_y}}, & \mu_x=0 .
    \end{cases}
  \end{equation}
  In particular \(\operatorname{Gap}(z)\leq\varepsilon\) whenever
  \begin{equation}\label{eq:mixed_target}
    \varepsilon'\leq\varepsilon'_\star:=
    \begin{cases}
      \dfrac{\varepsilon}
        {2\left(\frac16+\frac12\hat\kappa+\frac12\hat\kappa_\mu\right)},
        &\mu_x>0,\\[2.6ex]
      \min\left\{
        \dfrac{\varepsilon}{4\left(\frac16+\frac12\hat\kappa\right)},\
        \dfrac{\lambda_y\varepsilon^2}{16L_{xy}^2D_{\cX}^2}
      \right\},&\mu_x=0 ,
    \end{cases}
  \end{equation}
  where the second entry of the minimum is read as \(+\infty\) when
  \(D_{\cX}=0\).
\end{theorem}

\begin{proof}
  For \(v\in\cY\) we have \(\|v-y^0\|^2\leq\Omega_y\), hence
  \(F(x,v)\leq F_{\lambda_y}(x,v)+\tfrac{\lambda_y}{2}\Omega_y\), while
  \(F(u,y)\geq F_{\lambda_y}(u,y)\). Taking the maximum over \(v\in\cY\) and
  the minimum over \(u\in\cX\),
  \[
    \operatorname{Gap}(z)
    \leq\operatorname{Gap}_{\lambda_y}(z)+\frac{\lambda_y}{2}\Omega_y
    =\operatorname{Gap}_{\lambda_y}(z)+\frac{\varepsilon}{2}.
  \]
  We bound \(\operatorname{Gap}_{\lambda_y}(z)\)
  by~\eqref{eq:mixed_gap_identity}, in which the two normal-cone terms vanish
  by Assumption~\ref{ass:mixed_gap_transfer}(i). Since
  \(\Psi_\lambda(z)\leq\varepsilon'\) and every summand of \(\Psi_\lambda\) is
  nonnegative,
  \begin{equation}\label{eq:mixed_from_Psi}
    \D_f(x,x_\lambda^*)\leq\frac{\varepsilon'}{12},
    \quad
    \D_{g_\lambda}(y,y_\lambda^*)\leq\frac{\varepsilon'}{12},
    \quad
    \|x-x_\lambda^*\|^2\leq\frac{\varepsilon'}{\delta_x^\lambda},
    \quad
    \|y-y_\lambda^*\|^2\leq\frac{\varepsilon'}{\lambda_y}.
  \end{equation}
  By~\eqref{eq:mixed_Dh_bound} and
  \(\|\bB(x-x_\lambda^*)\|\leq L_{xy}\|x-x_\lambda^*\|\),
  \[
    \D_{h^*}(\bB x,\bB x_\lambda^*)
    \leq\frac{L_{xy}^2\|x-x_\lambda^*\|^2}{2\lambda_y}
    \leq\frac{\hat\kappa}{2}\varepsilon'.
  \]
  For the dual Bregman term we use~\eqref{eq:mixed_Dq_variational}. If
  \(\mu_x>0\), then \(\D_f(u,x_\lambda^*)\geq
  \tfrac{\mu_x}{2}\|u-x_\lambda^*\|^2\), and maximizing the resulting concave
  quadratic in \(u-x_\lambda^*\) over \(\R^{d_x}\supseteq\cX-x_\lambda^*\)
  gives
  \[
    \D_{q^*}(-\bB^\top y,-\bB^\top y_\lambda^*)
    \leq\frac{\|\bB^\top(y-y_\lambda^*)\|^2}{2\mu_x}
    \leq\frac{\hat\kappa_\mu}{2}\varepsilon'.
  \]
  If \(\mu_x=0\), we discard \(\D_f\geq0\) and use \(u,x_\lambda^*\in\cX\):
  \[
    \D_{q^*}(-\bB^\top y,-\bB^\top y_\lambda^*)
    \leq L_{xy}\|y-y_\lambda^*\|D_{\cX}
    \leq L_{xy}D_{\cX}\sqrt{\varepsilon'/\lambda_y}.
  \]
  Adding the four contributions gives~\eqref{eq:mixed_gap_transfer}, and
  \eqref{eq:mixed_target} makes everything after \(\varepsilon/2\) at most
  \(\varepsilon/2\).
\end{proof}

\begin{corollary}
\label{cor:mixed_gap_complexity}
  Let Assumption~\ref{ass:mixed_gap_transfer} hold and set
  \[
    \Psi_0:=\Psi_\lambda(z^0),
    \qquad
    \Gamma:=\frac{\varepsilon}{\varepsilon'_\star},
    \qquad
    \alpha_x:=\frac{1-\nu_x}{1+3\nu_x},
    \qquad
    \alpha_y:=\frac{1-\nu_y}{1+3\nu_y}.
  \]
  Let \(z_\lambda^S:=z^0\) if \(\Psi_0\leq\varepsilon'_\star\), and otherwise
  let \(z_\lambda^S\) be the output of
  Corollary~\ref{cor:bilinear_SPP_restart} applied to
  \((f,g_\lambda,p_{3,\lambda},Q_3)\) with terminal target
  \(\varepsilon'_\star\) of~\eqref{eq:mixed_target}. Then
  \[
    \operatorname{Gap}(z_\lambda^S)\leq\varepsilon
    \quad\text{for~\eqref{prob:bilinear_SPP_main}},
  \]
  and, with \(N_f\), \(N_g\), \(N_B\) as
  in~\eqref{eq:mixed_Ng_corrected}--\eqref{eq:mixed_NB_corrected} and
  \(S=\mathcal O\bigl(1+\log_+(\Psi_0\Gamma/\varepsilon)\bigr)\),
  \begin{equation}\label{eq:mixed_gap_counts}
    N_f^{\mathrm{tot}}=\mathcal O(\Gamma^{\alpha_x}N_f),
    \qquad
    N_g^{\mathrm{tot}}=\mathcal O(\Gamma^{\alpha_y}N_g),
    \qquad
    N_B^{\mathrm{tot}}=\mathcal O(N_B).
  \end{equation}
  Moreover
  \[
    \nu_x=\nu_y=1
    \ \Longrightarrow\
    \alpha_x=\alpha_y=0,
    \qquad
    \nu_y<1
    \ \Longrightarrow\
    \hat\kappa=\hat\kappa_\mu=\kappa_{xy}^\lambda,
    \ \ \Gamma=\Theta(1+\kappa_{xy}^\lambda).
  \]
\end{corollary}

\begin{proof}
  If \(\Psi_0\leq\varepsilon'_\star\), then
  \(\Psi_\lambda(z_\lambda^S)=\Psi_0\leq\varepsilon'_\star\) at zero oracle
  cost. Otherwise \(\varepsilon'_\star\in(0,\Psi_0]\) and
  Corollary~\ref{cor:bilinear_SPP_restart} returns \(z_\lambda^S\) with
  \(\Psi_\lambda(z_\lambda^S)\leq\varepsilon'_\star\). In both cases
  \eqref{eq:mixed_target} gives
  \(\operatorname{Gap}(z_\lambda^S)\leq\varepsilon\).

  By Appendix~\ref{app:proof_bilinear_SPP_restart}, the terminal target
  \(\varepsilon'\) enters only through
  \[
    N_f^{\mathrm{tot}}
    =\mathcal O\bigl((\varepsilon')^{-\alpha_x}\bigr)\ (\nu_x<1),
    \qquad
    N_g^{\mathrm{tot}}
    =\mathcal O\bigl((\varepsilon')^{-\alpha_y}\bigr)\ (\nu_y<1),
  \]
  and, in the remaining cases, through
  \(S=\mathcal O(1+\log_+(\Psi_0/\varepsilon'))\), with
  \(N_B^{\mathrm{tot}}=\mathcal O(\max\{\sqrt{\kappa_{xy}^\lambda},1\}\,S)\).
  Setting \(\varepsilon'=\varepsilon'_\star=\varepsilon/\Gamma\),
  \[
    (\varepsilon'_\star)^{-\alpha}=\Gamma^{\alpha}\varepsilon^{-\alpha},
    \qquad
    S=\mathcal O\bigl(1+\log_+(\Psi_0\Gamma/\varepsilon)\bigr),
  \]
  while \(\lambda_y=\varepsilon/\Omega_y\), \(\kappa_x^\lambda\) and
  \(\kappa_{xy}^\lambda\) are unchanged, which
  gives~\eqref{eq:mixed_gap_counts}.

  Finally, \(\nu_y<1\) forces \(\beta_x^\lambda=0\)
  in~\eqref{eq:mixed_beta_lambda}, hence
  \(\delta_x^\lambda=\mu_x=\delta_x>0\) and
  \[
    \hat\kappa=\hat\kappa_\mu
    =\frac{L_{xy}^2\Omega_y}{\delta_x\varepsilon}
    =\kappa_{xy}^\lambda,
    \qquad
    \Gamma=2\Bigl(\tfrac16+\hat\kappa\Bigr)
    =\Theta(1+\kappa_{xy}^\lambda)
  \]
  by~\eqref{eq:mixed_kappa_hat} and~\eqref{eq:mixed_target}.
\end{proof}

\section{Stochastic Settings} \label{app:stoch_cases}

We allow function components with arbitrary H\"older exponents in
\([0,1]\), accessed through conditionally unbiased stochastic
(sub)gradient oracles; the monotone operators and bilinear products are
evaluated exactly. The analysis retains both the H\"older modeling error
and the martingale contribution. On bounded domains it gives a uniform
expected gap. When \(f,g\) have positive intrinsic strong convexity on
their feasible sets, a restarted decomposition with \(p_3=0\) gives an
expected Lyapunov guarantee without anchor samples. The two guarantees
and their domain assumptions are stated separately below.

\subsection{Stochastic H\"older Oracle Extension}\label{sec:VI_stochastic}

The function components \(p_i\) are accessed through conditionally unbiased
stochastic first-order oracles, while the operators \(Q_i\) are evaluated
exactly. We allow each true function to have its own H\"older exponent
\(\nu_i\in[0,1]\). The H\"older modeling error and the stochastic error
enter the recursion additively and can be optimized independently. The
smooth stochastic bounds follow by setting every \(\nu_i=1\).

Let \(\|\cdot\|_{\bP,*}\) denote the dual norm of \(\|\cdot\|_{\bP}\),
\[
  \|s\|_{\bP,*}:=\sqrt{\langle s,\bP^{-1}s\rangle}.
\]

\paragraph{The stochastic algorithm and its filtration.}
Replace each component-gradient call in
Algorithm~\ref{alg:sliding_recursive} by a stochastic (sub)gradient using
a fresh sample, leaving the other steps unchanged. Only the sampled
vectors are needed: the function expressions in the recursion specify
affine changes of variables, and no function-value oracle is required.
The oracle-call nodes are
\[
  \mathcal T
  :=\bigl\{(k,\bt):1\leq k\leq n,\;
    \bt=(t_1,\ldots,t_k),\;0\leq t_\ell<T_\ell\bigr\}.
\]
Enumerate them as \(\rho_1,\ldots,\rho_N\) in the algorithm's depth-first
execution order, where \(N=\sum_{k=1}^n\prod_{j=1}^kT_j\). If
\(\xi_{\rho_m}\) is the sample drawn at \(\rho_m\), set
\[
  \cF_0:=\sigma(z_{\mathrm{in}}),\qquad
  \cF_m:=\cF_0\vee\sigma(\xi_{\rho_1},\ldots,\xi_{\rho_m}),\qquad
  \cF_{\rho_m}:=\cF_{m-1}.
\]
Thus \(\cF_\rho\) is the history before the call at \(\rho\). For
\(\rho=(k,\bt)\), the points \(z_{\bt}^k,\bar z_{\bt}^k\) and the
function \(\hat p_k^{k,\bt}\) are \(\cF_\rho\)-measurable, whereas
\(\tilde z_{\bt}^k\) generally depends on the current sample and later
samples in the same subtree. All iteration budgets and step-size
parameters below are deterministic. For a restarted invocation,
\(\cF_0\) may instead contain the entire history preceding that invocation.

\begin{assumption}[Stochastic H\"older first-order oracle]
\label{ass:stochastic_oracle_VI}
For every \(1\leq i\leq n\), the convex function \(p_i\) satisfies
Assumption~\ref{ass:Holder_continuous_VI}, with
\(0\leq\nu_i\leq1\) and \(H_i\geq0\). Let
\(s_i(v)\in\partial p_i(v)\) be a measurable selection; when
\(\nu_i>0\), it is the true gradient \(\nabla p_i(v)\).
There is a measurable stochastic oracle \(G_i\) such that, writing
\[
  \varepsilon_i(v,\xi):=G_i(v,\xi)-s_i(v),
\]
for every call to component \(i\), at node \(\rho\), and every
\(\cF_\rho\)-measurable query \(v\in\mathcal Z\),
\begin{equation}\label{eq:stoch_conditional_oracle}
  \mathbb E[\varepsilon_i(v,\xi_\rho)\mid\cF_\rho]=0,
  \qquad
  \mathbb E[\|\varepsilon_i(v,\xi_\rho)\|_{\bP,*}^2\mid\cF_\rho]
  \leq\sigma_i^2
  \quad\text{a.s.}
\end{equation}
The input \(z_{\mathrm{in}}\) has finite second moment.
\end{assumption}

The assumptions concern the true functions and the conditional moments of
the oracle; individual samples need not be H\"older continuous. H\"older
subgradients and Lipschitz operators have at most linear growth. Since the
algorithm makes finitely many averaging and weighted projection steps with
positive deterministic quadratic coefficients, its iterates have finite
second moments. Thus the martingale products used below are integrable.

For \(0\leq\nu_i<1\) and \(H_i>0\), choose a deterministic modeling
tolerance \(d_i>0\), and use \(L_i(d_i)\) from
\eqref{eq:Holder_L_of_delta}. For \(\nu_i=1\), set
\(d_i=0\) and \(L_i(d_i)=H_i\). If \(H_i=0\), the function is affine;
set \(d_i=L_i(d_i)=0\). Lemma~\ref{lem:Holder_to_inexact} gives, for
every \(u,v\in\mathcal Z\),
\begin{equation}\label{eq:stoch_exact_two_sided}
  0\leq p_i(u)-p_i(v)-\langle s_i(v),u-v\rangle
  \leq\frac{L_i(d_i)}{2}\|u-v\|_{\bP}^2+d_i.
\end{equation}
Here \(d_i\) is a modeling tolerance, not an oracle bias. The sampled
vector need not satisfy the lower bound in
\eqref{eq:stoch_exact_two_sided} pathwise. We therefore use the following
three-point inequality, which also permits the prox output to depend on
the sample used to compute it.

\begin{lemma}[Stochastic three-point inequality]
\label{lem:stoch_three_point}
Let \(h\colon\mathcal Z\to\R\) be convex, let
\(s(v)\in\partial h(v)\), and suppose
\[
  0\leq h(w)-h(v)-\langle s(v),w-v\rangle
  \leq\frac L2\|w-v\|_{\bP}^2+d,
  \qquad v,w\in\mathcal Z.
\]
At a \(\cF_\rho\)-measurable query \(v\), write
\(G=s(v)+\varepsilon\). For every \(\tau>0\), the following inequality
holds pathwise for all \(u,y\in\mathcal Z\), including random choices:
\begin{equation}\label{eq:stoch_three_point}
  h(u)-h(y)
  \geq\langle G,u-y\rangle
  -\frac{L+\tau}{2}\|y-v\|_{\bP}^2-d
  -\frac{\|\varepsilon\|_{\bP,*}^2}{2\tau}
  -\langle\varepsilon,u-v\rangle.
\end{equation}
If \(\varepsilon=0\) almost surely, the inequality holds with
\(\tau=0\) and the last two terms omitted. If the conditional moments
\eqref{eq:stoch_conditional_oracle} hold with variance \(\sigma^2\), and
\(u\) is \(\cF_\rho\)-measurable with \(u,v\) of finite second moment, then
\begin{equation}\label{eq:stoch_three_point_moments}
  \mathbb E[\langle\varepsilon,u-v\rangle\mid\cF_\rho]=0,
  \qquad
  \mathbb E[\|\varepsilon\|_{\bP,*}^2\mid\cF_\rho]\leq\sigma^2.
\end{equation}
\end{lemma}

\begin{proof}
Subtract the quadratic upper model at \(y\) from the convex lower bound
at \(u\), both based at \(v\), to obtain
\[
  h(u)-h(y)
  \geq\langle s(v),u-y\rangle-\frac L2\|y-v\|_{\bP}^2-d.
\]
Substitute \(s(v)=G-\varepsilon\), and use
\[
  -\langle\varepsilon,u-y\rangle
  =-\langle\varepsilon,u-v\rangle+\langle\varepsilon,y-v\rangle,
  \qquad
  \langle\varepsilon,y-v\rangle
  \geq-\frac\tau2\|y-v\|_{\bP}^2
        -\frac{\|\varepsilon\|_{\bP,*}^2}{2\tau}.
\]
These are pathwise inequalities, so no measurability restriction on
\(u,y\) is needed for \eqref{eq:stoch_three_point}. Conditional centering
in \eqref{eq:stoch_three_point_moments} additionally uses the
\(\cF_\rho\)-measurability of \(u,v\). The zero-noise case follows
directly from the first display.
\end{proof}

Setting \(\varepsilon=0\) and \(\tau=0\) recovers
\eqref{eq:oracle_lower_three_point} with deterministic tolerance \(d\).
For positive noise, the quadratic prox term absorbs the dependence on
\(y\); only the comparator term must be centered conditionally.

\begin{lemma}[Rescaled stochastic three-point inequality]
\label{lem:stoch_affine_rescaled}
Let Assumption~\ref{ass:stochastic_oracle_VI} hold, and suppose
\(0<\alpha_t\leq1\). Fix a node \(\rho=(k,\bt)\in\mathcal T\), and set
\(\Gamma_{\bt}:=\prod_{\ell=1}^k\alpha_{t_\ell}\).
The induced oracle has a selected mean subgradient
\(\hat s_k^{k,\bt}\). Write its error as
\[
  \varepsilon^{k,\bt}
  :=G_k^{k,\bt}(z_{\bt}^k,\xi_\rho)
    -\hat s_k^{k,\bt}(z_{\bt}^k).
\]
For every \(\tau_k>0\), almost surely and simultaneously for all
\(u,y\in\mathcal Z\),
\begin{equation}\label{eq:stoch_three_point_rescaled}
\begin{aligned}
  \hat p_k^{k,\bt}(u)-\hat p_k^{k,\bt}(y)
  \geq{}&\bigl\langle G_k^{k,\bt}(z_{\bt}^k,\xi_\rho),u-y\bigr\rangle
  -\frac{L_k^{k,\bt}+\tau_k^{k,\bt}}2\|y-z_{\bt}^k\|_{\bP}^2\\
  &-r_k^{k,\bt}
  -\bigl\langle\varepsilon^{k,\bt},u-z_{\bt}^k\bigr\rangle,
\end{aligned}
\end{equation}
where
\[
  L_k^{k,\bt}=L_k(d_k)\Gamma_{\bt},\qquad
  \tau_k^{k,\bt}=\tau_k\Gamma_{\bt},\qquad
  r_k^{k,\bt}
  =\frac{d_k+\|\varepsilon^{k,\bt}\|_{\bP,*}^2/(2\tau_k)}{\Gamma_{\bt}}.
\]
For a \(\cF_\rho\)-measurable comparator \(u\) with finite second moment,
\begin{equation}\label{eq:stoch_rescaled_moments}
  \mathbb E[\langle\varepsilon^{k,\bt},u-z_{\bt}^k\rangle\mid\cF_\rho]=0,
  \qquad
  \mathbb E[r_k^{k,\bt}\mid\cF_\rho]
  \leq\bar r_k^{k,\bt}
  :=\frac{d_k+\sigma_k^2/(2\tau_k)}{\Gamma_{\bt}}.
\end{equation}
If \(\sigma_k=0\), then \(\varepsilon^{k,\bt}=0\) almost surely, and
the conclusions hold with \(\tau_k=0\) and
\(r_k^{k,\bt}=\bar r_k^{k,\bt}=d_k/\Gamma_{\bt}\).
\end{lemma}

\begin{proof}
The affine transformation in line~\ref{line:define_hat_p} does not depend
on the type of oracle. As in Lemma~\ref{lem:affine_rescaled_oracle},
\[
  \hat p_k^{k,\bt}(z)=\Gamma_{\bt}^{-1}p_k(A_{\bt}^k(z)),
\]
where \(A_{\bt}^k\) is a \(\cF_\rho\)-measurable affine map with linear
part \(\Gamma_{\bt}\,\mathrm{Id}\). Its image lies in \(\mathcal Z\)
because the averaging points lie in the convex set \(\mathcal Z\).
The subgradient inequality gives the measurable selection and oracle
\[
  \hat s_k^{k,\bt}(z)=s_k(A_{\bt}^k(z)),\qquad
  G_k^{k,\bt}(z,\xi)=G_k(A_{\bt}^k(z),\xi).
\]
Apply Lemma~\ref{lem:stoch_three_point} to \(p_k\) at
\(u'=A_{\bt}^k(u)\), \(y'=A_{\bt}^k(y)\), and
\(v'=A_{\bt}^k(z_{\bt}^k)\), using its model with parameters
\((d_k,L_k(d_k))\). Divide by \(\Gamma_{\bt}\), and use
\[
  u'-y'=\Gamma_{\bt}(u-y),\qquad
  u'-v'=\Gamma_{\bt}(u-z_{\bt}^k),\qquad
  \|y'-v'\|_{\bP}^2=\Gamma_{\bt}^2\|y-z_{\bt}^k\|_{\bP}^2.
\]
This proves \eqref{eq:stoch_three_point_rescaled}. The query \(v'\) is
measurable before the call, so the original conditional oracle moments
give \eqref{eq:stoch_rescaled_moments}. The zero-noise conclusion follows
from the last clause of Lemma~\ref{lem:stoch_three_point}.
\end{proof}

The modeling and squared-noise errors both scale as
\(\Gamma_{\bt}^{-1}\), while the martingale term carries no
\(\Gamma_{\bt}\) factor in the rescaled coordinates. In particular,
zero noise removes only the stochastic part of \(r_k^{k,\bt}\).

\begin{lemma}[Stochastic recursive gap bound]
\label{lem:stoch_recursive_gap_bound}
Under Assumptions~\ref{ass:Monotone_and_Lipschitz_Operators_VI}
and~\ref{ass:stochastic_oracle_VI}, let \(T_i\in\mathbb N\), and let
\(\{\alpha_t\}\) satisfy \eqref{eq:alpha_sequence}. Choose the modeling
tolerances above and \(\tau_i\geq0\), with \(\tau_i>0\) whenever
\(\sigma_i>0\), such that \(L_i(d_i)+\tau_i+M_i>0\). Run stochastic
Algorithm~\ref{alg:sliding_recursive} with
\begin{equation}\label{eq:eta_stochastic}
  \eta_{\bt}^{k}
  =\bigl(L_k(d_k)+\tau_k\bigr)\prod_{\ell=1}^k\alpha_{t_\ell}
  +M_k\prod_{\ell=1}^k
        \frac{\alpha_{t_\ell}}{\alpha_{T_\ell-1}}.
\end{equation}
For every deterministic \(\hat z\in\mathcal Z\) and every admissible
\((k,\bt)\), including \((0,\varnothing)\),
\[
  \mathbb E[\Phi_{\bt,\hat z}^k(\tilde z_{\bt}^k,\hat z)]
  \leq\mathbb E[S_{\bt}^k]+R_{\bt}^k.
\]
Here \(\Phi_{\bt,\hat z}^k,S_{\bt}^k\) are the stochastic analogues of
the quantities in Lemma~\ref{lem:recursive_gap_bound}, with sampled
linear terms and with \(L_i^{k,\bt}\) replaced by
\((L_i(d_i)+\tau_i)\Gamma_{\bt}\), and
\[
  R_{\bt}^k
  :=\sum_{i=k+1}^n
     \frac{\bar r_i^{k,\bt}}{\prod_{j=k+1}^i\alpha_{T_j-1}},\qquad
  \bar r_i^{k,\bt}
  :=\frac{d_i+\sigma_i^2/(2\tau_i)}{\Gamma_{\bt}}.
\]
The quantity \(R_{\bt}^k\) is deterministic. We set
\(\sigma_i^2/(2\tau_i)=0\) when \(\sigma_i=\tau_i=0\), and use
\(\Gamma_\varnothing=1\).
\end{lemma}

\begin{proof}
Repeat the downward induction of Lemma~\ref{lem:recursive_gap_bound}.
At \(k=n\), the minimizing property gives
\(\Phi_{\bt,\hat z}^n(\tilde z_{\bt}^n,\hat z)\leq0\) pathwise, while
\(S_{\bt}^n=R_{\bt}^n=0\).

At a node \((k,\bt)\), replace the deterministic three-point inequality
in \eqref{eq:delta_h_lower_true} by
\eqref{eq:stoch_three_point_rescaled}, with \(u=\hat z\) and
\(y=\tilde z_{\bt}^k\). The prox output may depend on the sample, as
permitted by that lemma. The deterministic derivation then holds pathwise
with curvature \(L_k^{k,\bt}+\tau_k^{k,\bt}\) and residual
\(-r_k^{k,\bt}-\mathcal E_{\bt}^k\), where
\[
  \mathcal E_{\bt}^k
  :=\langle\varepsilon^{k,\bt},\hat z-z_{\bt}^k\rangle.
\]
Indeed, with \(\eta_{\bt}^k=L_k^{k,\bt}+\tau_k^{k,\bt}+M_k^{k,\bt}\),
the coefficient discarded in \eqref{eq:phi_lower_level_k} remains
nonnegative:
\[
  \eta_{\bt}^k-L_k^{k,\bt}-\tau_k^{k,\bt}
        -\frac{(M_k^{k,\bt})^2}{\eta_{\bt}^k}
  =\frac{M_k^{k,\bt}(L_k^{k,\bt}+\tau_k^{k,\bt})}{\eta_{\bt}^k}
  \geq0.
\]
Conditional centering and the remainder bound give
\[
  \mathbb E[\mathcal E_{\bt}^k\mid\cF_{(k,\bt)}]=0,
  \qquad
  \mathbb E[r_k^{k,\bt}\mid\cF_{(k,\bt)}]\leq\bar r_k^{k,\bt}.
\]
The remaining inequalities and telescoping identities in the deterministic
proof use deterministic nonnegative weights, so they commute with
expectation. Finally,
\(\bar r_i^{k,\bt}=\bar r_i^{k-1,\bt'}/\alpha_{t_k}\), exactly as for
the deterministic oracle tolerances; \eqref{eq:R_bound_level_k} therefore
gives \(R_{\bt'}^{k-1}\). This closes the induction.

The same induction may instead retain every error term pathwise. Each
transition from level \(k\) to \(k-1\) multiplies that error by
\(\alpha_{T_k-1}^2/\alpha_{t_k}\), as seen in
\eqref{eq:rearranged_induction_step}. This observation will be used for
the uniform gap below.
\end{proof}

\begin{theorem}[Stochastic H\"older sliding]\label{thm:VI_sliding_stochastic}
Let the assumptions and parameter choices of
Lemma~\ref{lem:stoch_recursive_gap_bound} hold. Define
\(A_i:=\prod_{j=1}^iT_j\). For every deterministic comparator
\(z\in\mathcal Z\), the output satisfies
\begin{equation}\label{eq:stoch_holder_master}
\begin{aligned}
  \mathcal G_{\mathrm{st}}(z)
  &:={\mathbb E}\bigl[p(z_{\mathrm{out}})-p(z)
                    +\langle Q(z),z_{\mathrm{out}}-z\rangle\bigr]\\
  &\leq\sum_{i=1}^n\left[
    \left(\frac{2^{2i-1}(L_i(d_i)+\tau_i)}{A_i^2}
          +\frac{2^{i-1}M_i}{A_i}\right)
       \mathbb E\|z_{\mathrm{in}}-z\|_{\bP}^2
    +d_iA_i+\frac{\sigma_i^2A_i}{2\tau_i}
  \right].
\end{aligned}
\end{equation}
The zero-noise convention of
Lemma~\ref{lem:stoch_recursive_gap_bound} applies.
\end{theorem}

\begin{proof}
At the root, the definitions in the proof of
Theorem~\ref{thm:VI_sliding_inexact} give \(H_\Sigma^0=0\),
\(z_{\mathrm{out}}=\tilde z_\varnothing^0\), and
\[
  \Phi_{\varnothing,z}^0(u,z)
  =p(u)-p(z)+\langle Q(z),u-z\rangle.
\]
Thus Lemma~\ref{lem:stoch_recursive_gap_bound} yields
\(\mathcal G_{\mathrm{st}}(z)\leq\mathbb E[S_\varnothing^0]+R_\varnothing^0\).
By Lemma~\ref{lem:alpha_properties},
\(1/T_j\leq\alpha_{T_j-1}\leq2/T_j\). Consequently,
\[
  R_\varnothing^0
  =\sum_{i=1}^n
     \frac{d_i+\sigma_i^2/(2\tau_i)}{\prod_{j=1}^i\alpha_{T_j-1}}
  \leq\sum_{i=1}^n
     \left(d_iA_i+\frac{\sigma_i^2A_i}{2\tau_i}\right).
\]
Use the initialization \(z^i_{\mathbf 0_i}=z_{\mathrm{in}}\), drop the
nonpositive terminal squared distances in \(S_\varnothing^0\), and then
take expectations. This gives
\[
  \mathbb E[S_\varnothing^0]
  \leq\sum_{i=1}^n
   \left(\frac{2^{2i-1}(L_i(d_i)+\tau_i)}{A_i^2}
         +\frac{2^{i-1}M_i}{A_i}\right)
    \mathbb E\|z_{\mathrm{in}}-z\|_{\bP}^2.
\]
Combining the two bounds proves \eqref{eq:stoch_holder_master}.
\end{proof}

\begin{corollary}[Optimized stochastic H\"older rate]
\label{cor:VI_sliding_stochastic_optimized}
Let Assumptions~\ref{ass:Monotone_and_Lipschitz_Operators_VI}
and~\ref{ass:stochastic_oracle_VI} hold, with
\(H_i+M_i+\sigma_i>0\) for each retained component pair. Fix positive
integers \(T_1,\ldots,T_n\), define \(A_i:=\prod_{j=1}^iT_j\), and use
the sequence \(\{\alpha_t\}\) in \eqref{eq:alpha_sequence}. Let
\(z\in\mathcal Z\) be deterministic and suppose that a deterministic
\(R>0\) satisfies
\begin{equation}\label{eq:stoch_radius_in_expectation}
  \mathbb E\|z_{\mathrm{in}}-z\|_{\bP}^2\leq R^2.
\end{equation}
For \(0\leq\nu_i<1\) and \(H_i>0\), set
\begin{equation}\label{eq:stoch_holder_optimal_parameters}
\begin{aligned}
  a_i&:=\frac{1-\nu_i}{1+\nu_i},\qquad
  K_i:=\left[\frac{1-\nu_i}{2(1+\nu_i)}\right]^{a_i}
            H_i^{2/(1+\nu_i)},\\
  d_i^\star&:=\left[
        \frac{a_i2^{2i-1}K_iR^2}{A_i^3}
       \right]^{(1+\nu_i)/2},\qquad
  \tau_i^\star:=\frac{\sigma_iA_i^{3/2}}{2^iR}.
\end{aligned}
\end{equation}
At \(\nu_i=1\), use \(d_i^\star=0\) and \(L_i(d_i^\star)=H_i\);
at \(H_i=0\), use \(d_i^\star=L_i(d_i^\star)=0\). The formula for
\(\tau_i^\star\) applies in every case, including \(\sigma_i=0\).
With these choices and the step sizes \eqref{eq:eta_stochastic}, the
stochastic version of Algorithm~\ref{alg:sliding_recursive} satisfies
\begin{equation}\label{eq:stoch_holder_optimized_rate}
  \mathcal G_{\mathrm{st}}(z)
  \leq\sum_{i=1}^n\left[
    \frac{2^{i(1+\nu_i)}H_iR^{1+\nu_i}}
         {(1+\nu_i)A_i^{(1+3\nu_i)/2}}
    +\frac{2^{i-1}M_iR^2}{A_i}
    +\frac{2^i\sigma_iR}{\sqrt{A_i}}
  \right].
\end{equation}
\end{corollary}

\begin{proof}
For each component, the parameter-dependent terms in
\eqref{eq:stoch_holder_master}, after using
\eqref{eq:stoch_radius_in_expectation}, separate as
\[
  \left[\frac{2^{2i-1}L_i(d_i)R^2}{A_i^2}+d_iA_i\right]
  +\left[\frac{2^{2i-1}\tau_iR^2}{A_i^2}
          +\frac{\sigma_i^2A_i}{2\tau_i}\right].
\]
The first bracket is the deterministic H\"older optimization
\eqref{eq:Holder_gi} with \(D=R\) and \(N_i=A_i\). Since
\(L_i(d)=K_id^{-a_i}\), its first-order condition yields
\(d_i^\star\) in \eqref{eq:stoch_holder_optimal_parameters}, and its
minimum, by \eqref{eq:Holder_per_component}, is
\[
  \frac{2^{i(1+\nu_i)}H_iR^{1+\nu_i}}
       {(1+\nu_i)A_i^{(1+3\nu_i)/2}}.
\]
For \(\nu_i=1\), the same expression is
\(2^{2i-1}H_iR^2/A_i^2\); for \(H_i=0\), it is zero.
For \(\sigma_i>0\), \(\tau_i^\star\) balances the two terms in the
second bracket, each of which equals
\(2^{i-1}\sigma_iR/\sqrt{A_i}\). If \(\sigma_i=0\), both vanish by
convention. This proves \eqref{eq:stoch_holder_optimized_rate}.
\end{proof}

The expected radius bound permits random restart inputs while the
parameters remain deterministic. A component with
\(H_i=M_i=\sigma_i=0\) contributes only a deterministic affine function
and a constant operator; it can be evaluated once and absorbed into
another component's fixed linear term without changing that component's
H\"older, Lipschitz, or variance constants. The optimized statements are
written for the remaining nonzero pairs.

\begin{corollary}[Stochastic H\"older oracle complexity]
\label{cor:VI_sliding_stochastic_complexity}
Let the assumptions of Corollary~\ref{cor:VI_sliding_stochastic_optimized}
hold with comparator \(z^*\), a solution of \eqref{prob:general_VI}.
For \(\varepsilon>0\), define the complete target budgets
\begin{equation}\label{eq:stoch_holder_target_budgets}
  B_i:=1+
    \left(\frac{H_iR^{1+\nu_i}}{\varepsilon}\right)^{2/(1+3\nu_i)}
    +\frac{M_iR^2}{\varepsilon}
    +\frac{\sigma_i^2R^2}{\varepsilon^2}.
\end{equation}
Relabel the component pairs so that \(B_1\leq\cdots\leq B_n\).
There are positive integers \(T_i\) with
\(A_i=\prod_{j=1}^iT_j=\mathcal O_n(B_i)\) such that, with the optimized
parameters in \eqref{eq:stoch_holder_optimal_parameters},
\[
  \mathbb E\bigl[p(z_{\mathrm{out}})-p(z^*)
       +\langle Q(z^*),z_{\mathrm{out}}-z^*\rangle\bigr]
  \leq\varepsilon.
\]
The method uses \(A_i\) stochastic function-oracle calls and at most
\(2A_i\) evaluations of \(Q_i\) for the \(i\)-th component pair.
\end{corollary}

\begin{proof}
In \eqref{eq:stoch_holder_optimized_rate}, the three terms for component
\(i\) are bounded by \(\varepsilon/(3n)\) whenever
\(A_i\geq c_nB_i\), for a sufficiently large constant \(c_n\) depending
only on \(n\). The exponent \((1+3\nu_i)/2\) lies in \([1/2,2]\), so
the constant can be chosen uniformly for \(\nu_i\in[0,1]\).
The integer construction of
Appendix~\ref{app:proof_VI_Holder_complexity} is explicit: with
\(A_0=1\), choose
\[
  T_i:=\max\{1,\lceil c_nB_i/A_{i-1}\rceil\},\qquad
  A_i:=A_{i-1}T_i.
\]
The nondecreasing order of the \(B_i\) gives
\(c_nB_i\leq A_i\leq2c_nB_i\) by induction. Summing the error bounds
proves the claim. Algorithm~\ref{alg:sliding_recursive} makes one sampled
function-oracle call and at most two operator evaluations per node at
that level.
\end{proof}

The ordering in \eqref{eq:stoch_holder_target_budgets} includes the
variance contribution. Sorting only the deterministic costs can force an
exact operator into a deeper level and make its evaluation count inherit
an unnecessarily large stochastic function budget.

\begin{remark}[Uniform and fixed-comparator residuals]
\label{rem:stoch_fixed_comparator}
For \(p\equiv0\), \(Q(z)=\bS z\) with \(\bS^\top=-\bS\), and solution
\(z^*=0\), the fixed-comparator residual is zero for every output.
The uniform gap
\[
  \sup_{z\in\mathcal Z}\langle Q(z),z_{\mathrm{out}}-z\rangle
  =\sup_{z\in\mathcal Z}\langle z,-\bS z_{\mathrm{out}}\rangle
\]
can detect the remaining bilinear residual. Therefore the bounded
convex--concave analysis uses the uniform expected gap below, while the
strongly convex analysis uses a fixed comparator to control the
Lyapunov function.
\end{remark}

\begin{theorem}[Uniform stochastic H\"older gap]
\label{thm:VI_sliding_stochastic_uniform}
Let the assumptions of Theorem~\ref{thm:VI_sliding_stochastic} hold.
Assume that a deterministic constant \(0<\Omega<\infty\) satisfies
\[
  \sup_{z\in\mathcal Z}\|z_{\mathrm{in}}-z\|_{\bP}^2
  \leq\Omega\qquad\text{almost surely}.
\]
Then
\begin{equation}\label{eq:stoch_holder_uniform_master}
\begin{aligned}
  &\mathbb E\Bigl[\sup_{z\in\mathcal Z}
     \{p(z_{\mathrm{out}})-p(z)
       +\langle Q(z),z_{\mathrm{out}}-z\rangle\}\Bigr]\\
  &\quad\leq\sum_{i=1}^n\left[
     \frac{2^{2i-1}(L_i(d_i)+\tau_i)\Omega}{A_i^2}
     +\frac{2^{i-1}M_i\Omega}{A_i}
     +d_iA_i+\frac{\sigma_i^2A_i}{2\tau_i}
   \right]
   +\sqrt{\Omega}\left(\sum_{i=1}^n\frac{2^i\sigma_i^2}{A_i}\right)^{1/2}.
\end{aligned}
\end{equation}
The same zero-noise convention applies.
\end{theorem}

\begin{proof}
The pathwise induction in Lemma~\ref{lem:stoch_recursive_gap_bound} holds
simultaneously for all \(\hat z\in\mathcal Z\): the three-point
inequality is pathwise for every comparator, and the iterates do not
depend on the comparator. A residual at node \((k,\bt)\) reaches the root
with deterministic weight
\[
  w_{\bt}^k:=\prod_{\ell=1}^k
                     \frac{\alpha_{T_\ell-1}^2}{\alpha_{t_\ell}}.
\]
Therefore, almost surely and for all \(\hat z\in\mathcal Z\),
\begin{equation}\label{eq:pathwise_uniform}
\begin{aligned}
  &p(z_{\mathrm{out}})-p(\hat z)
       +\langle Q(\hat z),z_{\mathrm{out}}-\hat z\rangle\\
  &\quad\leq S_\varnothing^0(\hat z)
      +\sum_{(k,\bt)\in\mathcal T}w_{\bt}^k r_k^{k,\bt}
      +\sum_{(k,\bt)\in\mathcal T}w_{\bt}^k
         \langle\varepsilon^{k,\bt},\hat z-z_{\bt}^k\rangle.
\end{aligned}
\end{equation}
By the conditional bound on \(r_k^{k,\bt}\) and
Lemma~\ref{lem:alpha_properties},
\[
\begin{aligned}
  \mathbb E\sum_{\bt}w_{\bt}^k r_k^{k,\bt}
  &\leq\left(d_k+\frac{\sigma_k^2}{2\tau_k}\right)
    \prod_{\ell=1}^k\left(
        \alpha_{T_\ell-1}^2\sum_{t=0}^{T_\ell-1}\alpha_t^{-2}
    \right)\\
  &\leq\frac{d_k+\sigma_k^2/(2\tau_k)}
               {\prod_{\ell=1}^k\alpha_{T_\ell-1}}
   \leq d_kA_k+\frac{\sigma_k^2A_k}{2\tau_k}.
\end{aligned}
\]
This includes \(\sigma_k=\tau_k=0\), in which case the modeling term
\(d_kA_k\) is retained.

Let \(V:=\sum_{(k,\bt)\in\mathcal T}w_{\bt}^k\varepsilon^{k,\bt}\).
The last sum in \eqref{eq:pathwise_uniform} equals
\[
  \langle V,\hat z-z_{\mathrm{in}}\rangle
  -\sum_{(k,\bt)\in\mathcal T}w_{\bt}^k
       \langle\varepsilon^{k,\bt},z_{\bt}^k-z_{\mathrm{in}}\rangle.
\]
The second term has zero expectation because its query point and input
are measurable before the corresponding call. The radius assumption
gives the pathwise bound
\[
  \sup_{\hat z\in\mathcal Z}
       \langle V,\hat z-z_{\mathrm{in}}\rangle
  \leq\sqrt{\Omega}\|V\|_{\bP,*}.
\]
Conditional centering, in the depth-first sampling order, eliminates all
cross terms in \(\mathbb E\|V\|_{\bP,*}^2\). Jensen's inequality yields
\[
  \mathbb E\|V\|_{\bP,*}
  \leq\left(\sum_{k=1}^n\sigma_k^2
                           \sum_{\bt}(w_{\bt}^k)^2\right)^{1/2}.
\]
Again using Lemma~\ref{lem:alpha_properties},
\[
  \sum_{\bt}(w_{\bt}^k)^2
  =\prod_{\ell=1}^k\left(
      \alpha_{T_\ell-1}^4\sum_{t=0}^{T_\ell-1}\alpha_t^{-2}
    \right)
  \leq\prod_{\ell=1}^k\alpha_{T_\ell-1}
  \leq\frac{2^k}{A_k}.
\]
Take the supremum in \eqref{eq:pathwise_uniform}, followed by expectation.
Dropping the terminal distances in \(S_\varnothing^0\) and using the
radius bound gives the curvature and operator terms. The preceding
bounds give the modeling, variance, and martingale terms in
\eqref{eq:stoch_holder_uniform_master}.
\end{proof}

\begin{corollary}[Uniform stochastic H\"older complexity]
\label{cor:VI_sliding_stochastic_uniform_complexity}
Under the assumptions of Theorem~\ref{thm:VI_sliding_stochastic_uniform},
suppose \(H_i+M_i+\sigma_i>0\) for every retained pair and choose
\(d_i,\tau_i\) as in \eqref{eq:stoch_holder_optimal_parameters} with
\(R=\sqrt{\Omega}\), including the stated endpoint conventions. Then
\begin{equation}\label{eq:stoch_holder_uniform_optimized}
\begin{aligned}
  &\mathbb E\Bigl[\sup_{z\in\mathcal Z}
     \{p(z_{\mathrm{out}})-p(z)
       +\langle Q(z),z_{\mathrm{out}}-z\rangle\}\Bigr]\\
  &\quad\leq\sum_{i=1}^n\left[
    \frac{2^{i(1+\nu_i)}H_i\Omega^{(1+\nu_i)/2}}
         {(1+\nu_i)A_i^{(1+3\nu_i)/2}}
    +\frac{2^{i-1}M_i\Omega}{A_i}
    +\frac{(2^i+2^{i/2})\sigma_i\sqrt{\Omega}}{\sqrt{A_i}}
  \right].
\end{aligned}
\end{equation}
Consequently, the complete budgets
\eqref{eq:stoch_holder_target_budgets}, with \(R^2=\Omega\), sorted and
rounded as in Corollary~\ref{cor:VI_sliding_stochastic_complexity}, also
suffice to make the uniform expected gap at most \(\varepsilon\).
\end{corollary}

\begin{proof}
Optimize the first four terms of
\eqref{eq:stoch_holder_uniform_master} exactly as in
Corollary~\ref{cor:VI_sliding_stochastic_optimized}. Bound the final square
root by \(\sum_i2^{i/2}\sigma_i/\sqrt{A_i}\), and apply the same budget
construction with a larger constant depending only on \(n\).
\end{proof}

At \(\nu_i=1\), the deterministic function contribution in
\eqref{eq:stoch_holder_uniform_optimized} is
\(2^{2i-1}H_i\Omega/A_i^2\), recovering the smooth stochastic bound.
At \(\sigma_i=0\) for every component, the bound becomes exactly the
deterministic H\"older rate in Theorem~\ref{thm:VI_sliding_Holder}.

\subsection{Stochastic Bilinear Saddle-Point Problem}
\label{sec:bilinear_SPP_stochastic_degenerate}

Consider Problem~\eqref{prob:bilinear_SPP_main} on nonempty closed convex
sets \(\cX,\cY\), with bounded product set \(\cZ=\cX\times\cY\). Assume that
\(f,g\) are convex and satisfy the \((\nu_x,H_x)\)- and
\((\nu_y,H_y)\)-H\"older conditions on their respective feasible sets, where
\(0\leq\nu_x,\nu_y\leq1\). Their stochastic oracles satisfy
Assumption~\ref{ass:stochastic_oracle_VI} with Euclidean conditional
variances \(\sigma_x^2,\sigma_y^2\). In particular, at exponent zero the
conditional oracle mean is a selected ambient subgradient. Assume
\(\|\bB\|_{\mathrm{op}}\leq L_{xy}\), and evaluate products with
\(\bB,\bB^\top\) exactly. No positive curvature is needed in this
subsection, so the result includes the regime \(\delta_x=\delta_y=0\).

For deterministic \(z^0\in\cZ\), define
\[
  \Omega:=\sup_{z\in\cZ}\|z^0-z\|^2<\infty.
\]
The case \(\Omega=0\) is immediate; hence assume \(\Omega>0\). Use
Algorithm~\ref{alg:sliding_recursive} in the Euclidean norm with
\begin{equation}\label{eq:stoch_bilinear_unregularized_decomposition}
  \begin{aligned}
  &p_1(x,y)=f(x),\qquad p_2(x,y)=g(y),\qquad p_3\equiv0,\\
  &Q_1=Q_2=0,\qquad
  Q_3(x,y)=
  \begin{pmatrix}
    \bB^\top y\\-\bB x
  \end{pmatrix}.
  \end{aligned}
\end{equation}
The operator \(Q_3\) is monotone and \(L_{xy}\)-Lipschitz. Let
\(N_f,N_g,N_B\) count, respectively, stochastic evaluations of the two
function oracles and products with \(\bB,\bB^\top\). A constant factor of
two between a call to \(Q_3\) and its two matrix-vector products is absorbed
in the bounds below.

\begin{theorem}[Bounded-domain stochastic H\"older complexity]
\label{thm:bilinear_SPP_stochastic_Holder_bounded}
Under the assumptions above, define
\[
  \operatorname{Gap}(x,y)
  :=
  \sup_{v\in\cY}F(x,v)-\inf_{u\in\cX}F(u,y).
\]
With the modeling tolerances and noise parameters of
Corollary~\ref{cor:VI_sliding_stochastic_uniform_complexity}, the output
\(z_{\mathrm{out}}=(x_{\mathrm{out}},y_{\mathrm{out}})\) satisfies
\begin{equation}\label{eq:stoch_bilinear_Holder_gap}
  \begin{aligned}
  \mathbb E\operatorname{Gap}(z_{\mathrm{out}})
  \leq C\Bigg[
    &\frac{H_x\Omega^{(1+\nu_x)/2}}
           {N_f^{(1+3\nu_x)/2}}
    +\frac{H_y\Omega^{(1+\nu_y)/2}}
           {N_g^{(1+3\nu_y)/2}}\\
    &+\frac{L_{xy}\Omega}{N_B}
    +\sigma_x\sqrt{\frac{\Omega}{N_f}}
    +\sigma_y\sqrt{\frac{\Omega}{N_g}}
  \Bigg],
  \end{aligned}
\end{equation}
where \(C>0\) is an absolute constant. In particular, for any
\(\varepsilon>0\), there is a component ordering and a choice of integer
iteration counts giving \(\mathbb E\operatorname{Gap}(z_{\mathrm{out}})
\leq\varepsilon\) with
\begin{equation}\label{eq:stoch_bilinear_Holder_complexity}
  \begin{aligned}
  N_f
  &=\mathcal O\left(
    1+
    \left(\frac{H_x\Omega^{(1+\nu_x)/2}}{\varepsilon}
    \right)^{\frac{2}{1+3\nu_x}}
    +\frac{\sigma_x^2\Omega}{\varepsilon^2}
  \right),\\
  N_g
  &=\mathcal O\left(
    1+
    \left(\frac{H_y\Omega^{(1+\nu_y)/2}}{\varepsilon}
    \right)^{\frac{2}{1+3\nu_y}}
    +\frac{\sigma_y^2\Omega}{\varepsilon^2}
  \right),\\
  N_B
  &=\mathcal O\left(
    1+\frac{L_{xy}\Omega}{\varepsilon}
  \right).
  \end{aligned}
\end{equation}
\end{theorem}

\begin{proof}
For the decomposition~\eqref{eq:stoch_bilinear_unregularized_decomposition},
skew-symmetry gives the pointwise identity
\[
  p(z_{\mathrm{out}})-p(z)
  +\langle Q(z),z_{\mathrm{out}}-z\rangle
  =
  F(x_{\mathrm{out}},y)-F(x,y_{\mathrm{out}}).
\]
Taking the supremum over \(z=(x,y)\in\cZ\) therefore gives precisely
\(\operatorname{Gap}(z_{\mathrm{out}})\). Apply
Corollary~\ref{cor:VI_sliding_stochastic_uniform_complexity} with the
component data
\[
  (H,\nu,M,\sigma)
  =
  (H_x,\nu_x,0,\sigma_x),\quad
  (H_y,\nu_y,0,\sigma_y),\quad
  (0,1,L_{xy},0).
\]
The constants for at most three recursive levels are absorbed into \(C\),
which proves~\eqref{eq:stoch_bilinear_Holder_gap}.

To obtain~\eqref{eq:stoch_bilinear_Holder_complexity}, form the three target
budgets on its right-hand side, multiply them by a sufficiently large
absolute constant, and sort the complete budgets, including their
variance terms, in nondecreasing order. The integer construction of
Appendix~\ref{app:proof_VI_Holder_complexity} realizes them within absolute
factors. Consequently, each summand in
\eqref{eq:stoch_bilinear_Holder_gap} is a sufficiently small multiple of
\(\varepsilon\). The ordering is part of the construction: a coupling
component permanently assigned to the innermost level could otherwise
inherit a larger stochastic function budget.
\end{proof}

An exact function component with \(H_q=\sigma_q=0\) is affine on its
feasible set. Its constant subgradient can be evaluated once and included
as a fixed linear term in another retained component. An identically zero
component pair can be omitted. If all components are exact and affine
and \(\bB=0\), one projection minimizing their linear sum plus
\(\varepsilon\|z-z^0\|^2/(2\Omega)\) gives gap at most \(\varepsilon/2\);
thus the additive constants in
\eqref{eq:stoch_bilinear_Holder_complexity} also cover this case.

At \(\nu_x=\nu_y=1\), the bounds recover the smooth stochastic rates.
When \(\sigma_x=\sigma_y=0\), they recover the deterministic H\"older
bounds. In both cases, the guarantee is the uniform expected primal--dual
gap, with supremization performed before expectation.

\subsection{Stochastic strongly convex--strongly concave regime}
\label{sec:bilinear_SPP_stochastic_strongly_convex}

Assume that \(\cX,\cY\) are nonempty closed convex sets and that \(f,g\) are
respectively \(\mu_x\)- and \(\mu_y\)-strongly convex on these sets, with
\(\mu_x,\mu_y>0\). Suppose they satisfy the
\((\nu_x,H_x)\)- and \((\nu_y,H_y)\)-H\"older conditions on these sets, for
\(0\leq\nu_x,\nu_y\leq1\). Assume a saddle point \(z^*=(x^*,y^*)\) exists,
and fix ambient subgradients and normal vectors satisfying
\eqref{eq:opt_cond}. Define \(\D_f,\D_g\) using this fixed choice as in
\eqref{eq:def_D_f_g}. The positive intrinsic curvatures make the saddle
point unique.

The stochastic oracles have conditionally unbiased errors and Euclidean
conditional second moments bounded by \(\sigma_x^2,\sigma_y^2\), as in
Assumption~\ref{ass:stochastic_oracle_VI}. These conditions hold relative to
the history before every query, including all earlier restarts. Assume
\(\|\bB\|_{\mathrm{op}}\leq L_{xy}\), and evaluate products with
\(\bB,\bB^\top\) exactly.

Use the decomposition
\eqref{eq:stoch_bilinear_unregularized_decomposition}, including
\(p_3\equiv0\) when either function is smooth, with norm matrix
\[
  \bP=\diag(\mu_x\bI_{d_x},\mu_y\bI_{d_y}).
\]
Thus the curvature used here is supplied entirely by \(f,g\); in the
effective-curvature notation we take \(\mu_{xy}=\mu_{yx}=0\) and
\(\delta_x=\mu_x,\delta_y=\mu_y\). The component constants in this norm are
\begin{equation}\label{eq:stoch_strong_Holder_normalized}
  \begin{aligned}
  \widetilde H_x&=H_x\mu_x^{-(1+\nu_x)/2},
  &\widetilde H_y&=H_y\mu_y^{-(1+\nu_y)/2},\\
  \bar\sigma_x^2&=\frac{\sigma_x^2}{\mu_x},
  &\bar\sigma_y^2&=\frac{\sigma_y^2}{\mu_y}.
  \end{aligned}
\end{equation}
The operator constant is
\[
  M_B=\frac{L_{xy}}{\sqrt{\mu_x\mu_y}}.
\]
Thus \(H_3=0\), \(M_1=M_2=0\), \(M_3=M_B\), and \(\sigma_3=0\).

\paragraph{A fixed residual controls the Lyapunov function.}
For \(z=(x,y)\in\cZ\), set
\[
  \mathcal R_*(z)
  :=
  F(x,y^*)-F(x^*,y)
  =
  p(z)-p(z^*)+\langle Q(z^*),z-z^*\rangle.
\]
By \eqref{eq:opt_cond} and the definition of the normal cones,
\begin{equation}\label{eq:stoch_strong_residual_lower_bound}
  \begin{aligned}
  \mathcal R_*(z)
  &=
  \D_f(x,x^*)+\D_g(y,y^*)
  -\langle\xi,x-x^*\rangle-\langle\zeta,y-y^*\rangle\\
  &\geq \D_f(x,x^*)+\D_g(y,y^*)
  \geq \tfrac12\|z-z^*\|_{\bP}^2.
  \end{aligned}
\end{equation}
Consequently, for the Lyapunov function~\eqref{eq:Psi_s} with this \(\bP\),
\begin{equation}\label{eq:stoch_strong_Psi_residual}
  \begin{aligned}
  \Psi(z)
  &=\|z-z^*\|_{\bP}^2
    +12\D_f(x,x^*)+12\D_g(y,y^*)\\
  &\leq14\bigl(\D_f(x,x^*)+\D_g(y,y^*)\bigr)
  \leq14\mathcal R_*(z).
  \end{aligned}
\end{equation}
In particular, the normal vectors need not vanish. This estimate uses
strong convexity on the feasible sets and does not require the ambient
self-bounding inequality~\eqref{eq:ambient_self_bounding}.

\paragraph{One restart.}
Let \(\Gamma_s>0\) be deterministic with
\(\mathbb E\Psi(z^s)\leq\Gamma_s\). For \(q\in\{x,y\}\), define
\begin{equation}\label{eq:stoch_strong_Holder_exponents}
  \theta_q:=\frac{1-\nu_q}{1+3\nu_q},
  \qquad
  \mathcal K_q:=
  \widetilde H_q^{\frac{2}{1+3\nu_q}}
  =
  H_q^{\frac{2}{1+3\nu_q}}
  \mu_q^{-\frac{1+\nu_q}{1+3\nu_q}}.
\end{equation}
Choose the stage oracle budgets, up to sufficiently large absolute
factors, as
\begin{equation}\label{eq:stoch_strong_Holder_stage_budgets}
  \begin{aligned}
  N_{f,s}&\asymp
    1+\mathcal K_x\Gamma_s^{-\theta_x}
    +\frac{\sigma_x^2}{\mu_x\Gamma_s},\\
  N_{g,s}&\asymp
    1+\mathcal K_y\Gamma_s^{-\theta_y}
    +\frac{\sigma_y^2}{\mu_y\Gamma_s},\\
  N_{B,s}&\asymp1+M_B.
  \end{aligned}
\end{equation}
Sort these complete budgets at each restart and use the integer
construction of Appendix~\ref{app:proof_VI_Holder_complexity}. Choose the
modeling tolerances and noise parameters from
Corollary~\ref{cor:VI_sliding_stochastic_optimized} with
\(R^2=\Gamma_s\) and the actual cumulative iteration counts after sorting.
They are deterministic even though \(z^s\) is random.

Since \(\mathbb E\|z^s-z^*\|_{\bP}^2\leq\Gamma_s\), the fixed-comparator
bound of Corollary~\ref{cor:VI_sliding_stochastic_optimized} gives
\[
  \begin{aligned}
  \mathbb E\mathcal R_*(z^{s+1})
  \leq C\Bigg[
    &\frac{\widetilde H_x\Gamma_s^{(1+\nu_x)/2}}
           {N_{f,s}^{(1+3\nu_x)/2}}
    +\frac{\widetilde H_y\Gamma_s^{(1+\nu_y)/2}}
           {N_{g,s}^{(1+3\nu_y)/2}}\\
    &+\frac{M_B\Gamma_s}{N_{B,s}}
    +\bar\sigma_x\sqrt{\frac{\Gamma_s}{N_{f,s}}}
    +\bar\sigma_y\sqrt{\frac{\Gamma_s}{N_{g,s}}}
  \Bigg]
  \leq\frac{\Gamma_s}{28}.
  \end{aligned}
\]
The final inequality follows by taking the constants in
\eqref{eq:stoch_strong_Holder_stage_budgets} sufficiently large.
Combining this with \eqref{eq:stoch_strong_Psi_residual} yields
\begin{equation}\label{eq:stoch_strong_Holder_contraction}
  \mathbb E\Psi(z^{s+1})\leq\tfrac12\Gamma_s.
\end{equation}

\begin{theorem}[Restarted stochastic H\"older complexity]
\label{thm:bilinear_SPP_stochastic_Holder_strongly_convex}
Under the assumptions of this subsection, let \(z^0\in\cZ\) be
deterministic and let \(\Gamma_0>0\) be a known upper bound on
\(\Psi(z^0)\). For \(0<\varepsilon<\Gamma_0\), set
\[
  \Gamma_s:=2^{-s}\Gamma_0,\qquad
  S:=\left\lceil\log_2\frac{\Gamma_0}{\varepsilon}\right\rceil,
  \qquad
  \ell_\varepsilon:=1+\log\frac{\Gamma_0}{\varepsilon}.
\]
Run \(S\) restarts with
\eqref{eq:stoch_strong_Holder_stage_budgets} and the parameters specified
above. Then
\[
  \mathbb E\Psi(z^S)\leq\varepsilon.
\]
For \(q\in\{x,y\}\), write \(N_x^{\mathrm{tot}}:=N_f^{\mathrm{tot}}\)
and \(N_y^{\mathrm{tot}}:=N_g^{\mathrm{tot}}\). Define
\[
  G_\theta(S):=
  \begin{cases}
    \dfrac{2^{\theta S}-1}{2^\theta-1},&\theta>0,\\[1.2ex]
    S,&\theta=0.
  \end{cases}
\]
The total oracle counts satisfy
\begin{equation}\label{eq:stoch_strong_Holder_exact_sum}
  N_q^{\mathrm{tot}}
  =
  \mathcal O\left(
    S+\mathcal K_q\Gamma_0^{-\theta_q}G_{\theta_q}(S)
    +\frac{\sigma_q^2}{\mu_q\Gamma_0}(2^S-1)
  \right),
  \qquad q\in\{x,y\},
\end{equation}
and
\begin{equation}\label{eq:stoch_strong_Holder_coupling}
  N_B^{\mathrm{tot}}
  =
  \mathcal O\left(
    \left(1+\frac{L_{xy}}{\sqrt{\mu_x\mu_y}}\right)
    \ell_\varepsilon
  \right).
\end{equation}
In particular, for each function block,
\begin{equation}\label{eq:stoch_strong_Holder_total}
  N_q^{\mathrm{tot}}
  =
  \begin{cases}
  \displaystyle
  \mathcal O_{\nu_q}\left(
    \ell_\varepsilon
    +\left(
      \frac{H_q^2}
           {\mu_q^{1+\nu_q}\varepsilon^{1-\nu_q}}
     \right)^{\frac{1}{1+3\nu_q}}
    +\frac{\sigma_q^2}{\mu_q\varepsilon}
  \right),
  &0\leq\nu_q<1,\\[2.4ex]
  \displaystyle
  \mathcal O\left(
    \left(1+\sqrt{\frac{H_q}{\mu_q}}\right)\ell_\varepsilon
    +\frac{\sigma_q^2}{\mu_q\varepsilon}
  \right),
  &\nu_q=1.
  \end{cases}
\end{equation}
The constants in \eqref{eq:stoch_strong_Holder_exact_sum} are absolute.
The notation \(\mathcal O_{\nu_q}\) in
\eqref{eq:stoch_strong_Holder_total} permits dependence on the fixed
H\"older exponent.
\end{theorem}

\begin{proof}
Induction using \eqref{eq:stoch_strong_Holder_contraction} gives
\(\mathbb E\Psi(z^s)\leq\Gamma_s\), and
\(\Gamma_S\leq\varepsilon\). Summing
\eqref{eq:stoch_strong_Holder_stage_budgets} yields
\eqref{eq:stoch_strong_Holder_exact_sum}, since
\[
  \sum_{s=0}^{S-1}\Gamma_s^{-\theta}
  =\Gamma_0^{-\theta}G_\theta(S),
  \qquad
  \sum_{s=0}^{S-1}\Gamma_s^{-1}
  =\frac{2^S-1}{\Gamma_0}<\frac{2}{\varepsilon}.
\]
For \(\theta_q>0\), the first sum is at most
\[
  \frac{2^{\theta_q}}{2^{\theta_q}-1}
  \varepsilon^{-\theta_q}.
\]
This gives the first case of
\eqref{eq:stoch_strong_Holder_total}. For \(\nu_q=1\),
\(\theta_q=0\) and
\(\mathcal K_q=\sqrt{H_q/\mu_q}\), giving its second case. The coupling
budget is constant across stages, so its sum is
\eqref{eq:stoch_strong_Holder_coupling}.
\end{proof}

If a known upper bound already satisfies
\(\Gamma_0\leq\varepsilon\), including the case \(\Gamma_0=0\), return
\(z^0\). Formula
\eqref{eq:stoch_strong_Holder_exact_sum} displays the smooth limit:
\(G_\theta(S)\to S\) as \(\theta\downarrow0\). In contrast, the constant in
the simplified polynomial bound is not uniform as \(\nu_q\uparrow1\).
The variance contribution has no restart logarithm, because its stage
budgets form a geometric sum. The construction uses no anchor samples
and gives \(\varepsilon\)-accuracy without an anchor error term.

The guarantee is in the Lyapunov metric \(\Psi\), and implies
\[
  \mathbb E\|x^S-x^*\|^2\leq\frac{\varepsilon}{\mu_x},
  \qquad
  \mathbb E\|y^S-y^*\|^2\leq\frac{\varepsilon}{\mu_y}.
\]
A uniform primal--dual gap guarantee requires a separate conversion.

\begin{remark}[Scope of the strongly convex extension]
\label{rem:stoch_strong_Holder_scope}
The H\"older and strong-convexity assumptions above are imposed on the
feasible sets. Indeed, if a function is \(\mu\)-strongly convex and
satisfies a \((\nu,H)\)-H\"older condition with \(\mu>0\) and \(\nu<1\),
then for any two distinct feasible points \(u,v\),
\[
  \mu\|u-v\|^2
  \leq\langle h'(u)-h'(v),u-v\rangle
  \leq H\|u-v\|^{1+\nu}.
\]
Hence their distance is at most
\((H/\mu)^{1/(1-\nu)}\). In particular, these assumptions cannot hold
globally on a positive-dimensional full space.

Theorem~\ref{thm:bilinear_SPP_stochastic_Holder_strongly_convex} uses
\(\mu_x,\mu_y>0\) and \(p_3=0\). It therefore needs neither active-anchor
ambient smoothness nor samplewise range invariance. Exploiting
coupling-induced curvature, or treating a block with zero intrinsic
curvature, requires additional analysis: active anchor terms must retain
the appropriate genuine smoothness assumptions, and any subspace
invariance argument must also control the stochastic oracle samples.
The inexact-model constants \(L_i(d_i)\) do not by themselves imply
\eqref{eq:ambient_self_bounding}. The uniform gap conversions for
regularized partially strongly convex problems in
Appendix~\ref{app:mixed_cases} also require their own
stochastic accuracy accounting.
\end{remark}


\section{Experimental Details}\label{app:experiments}

All experiments are CPU-only and single-threaded, run on an Intel Core
i5-8250U under x86-64 Linux 6.8 with Python 3.12, NumPy 2.5, SciPy 1.18, and
Matplotlib 3.11.

\subsection{Synthetic experiments}\label{app:exp_synthetic}

\paragraph{Instances.}\label{app:exp_instances}

We use separable primal and dual functions with a linear tilt,
\[
  \begin{aligned}
    f(x)&=\tfrac12\textstyle\sum_i a_ix_i^2
      +\tfrac{\alpha_x}{1+\nu_x}\sum_i|x_i|^{1+\nu_x}-\langle c_x,x\rangle,\\
    g(y)&=\tfrac12\textstyle\sum_j b_jy_j^2
      +\tfrac{\alpha_y}{1+\nu_y}\sum_j|y_j|^{1+\nu_y}-\langle c_y,y\rangle,
  \end{aligned}
\]
on boxes \(\cX=[-R_x,R_x]^{d_x}\), \(\cY=[-R_y,R_y]^{d_y}\), with
\(d_x=d_y=60\) and \(R_x=R_y=1\), so \(\Omega=R_x^2d_x+R_y^2d_y=120\).
Separability is what makes the exact primal--dual gap computable, and the tilts
\(c_x,c_y\sim\cN(0,I)\) move the saddle point away from the origin --- without
them \(z^0=0\) is already optimal and every method looks identical. Tilts change
no constant, since \(H\), \(\mu\) and \(L\) depend only on the even part.

The coupling is \(\bB=\bU\mathbf{\Sigma}\bV^\top\) with \(\bU,\bV\) Haar-distributed
(from the \(QR\) factorization of Gaussians) and singular values log-spaced in
\([\sigma_{\min},\sigma_{\max}]\). Hence
\(\lambda_{\max}(\bB^\top\bB)=\sigma_{\max}^2\),
\(\lminp(\bB^\top\bB)=\sigma_{\min}^2\), and \(\bB\) has full column rank, so
the compatibility conditions in Assumption~\ref{ass:properties_of_B} hold
automatically.

\paragraph{Constants used in the runs.}
Because Assumptions~\ref{ass:properties_of_f}--\ref{ass:properties_of_B}
allow any valid strong convexity lower bounds, we set
\[
  \mu_x=\mu_y=0,
  \qquad
  \mu_{xy}=\mu_{yx}=0,
\]
so \(\delta_x=\delta_y=0\) and
Table~\ref{tab:bilinear_SPP_degenerate_complexity} applies. We use
\(p_3\equiv0\), \(\bP=\bI\), and no restarts, and set the available curvature
lower bounds to zero to isolate the H\"older exponents.

The H\"older constant of \(\nabla f\) on \(\cX\) is
\[
  H_x
  =
  2^{1-\nu_x}\alpha_x d_x^{(1-\nu_x)/2}
  +
  \big(\max_i a_i\big)\,\mathrm{diam}(\cX)^{1-\nu_x} .
\]
The first term follows from the scalar constant \(2^{1-\nu_x}\) and
coordinate aggregation, with equality for equally spread increments. The
second converts the quadratic part's Lipschitz constant into a H\"older bound
on the bounded set.

\paragraph{Metric and oracle accounting.}

We report
\[
  \operatorname{Gap}(x,y)
  :=\max_{y'\in\cY}F(x,y')-\min_{x'\in\cX}F(x',y).
\]
Separability and box constraints reduce both extrema to scalar problems, solved
by vectorized bisection to machine precision. For \(p=f+g\) and the skew
coupling \(Q\), this is precisely the residual in
Theorem~\ref{thm:VI_sliding_Holder}:
\[
  \operatorname{Gap}(\xout,\yout)
  =\sup_z\{p(\zout)-p(z)+\langle Q(z),\zout-z\rangle\}.
\]

Instrumented wrappers count each \(f'\), \(g'\), \(\bB v\), and
\(\bB^\top u\) evaluation. Gap computations are excluded. With the coupling
innermost, the three-level scheme uses
\(N_f=T_1\), \(N_g=T_1T_2\), and \(N_B=4T_1T_2T_3\).

\paragraph{Baselines and tuning.}\label{app:exp_baselines}

Figure~\ref{fig:e2_separation} compares Generalized Universal
Mirror-Prox~\citep{stonyakin2022generalized} with lockstep and two-level
ablations. Generalized Universal Mirror-Prox covers H\"older-continuous
operators. Lockstep assigns every component the largest required budget, while
two-level groups \(g'\) with the coupling. Both schedules follow
Corollary~\ref{cor:VI_sliding_Holder_complexity}. The universal and lockstep
methods query \(f'\) and \(g'\) equally often, so \(N_g/N_f=1\).

The smooth cost study uses Mirror-Prox~\citep{nemirovski2004prox} under the
same oracle budget. The application also includes Accelerated
Mirror-Prox~\citep{chen2017accelerated} and two-level Mirror-Prox
Sliding~\citep{lan2021mirrorprox}, with \(f\) outside and \(g+\bD\) inside.
Each of the four methods selects among twelve configurations on one validation
instance per regime, then freezes them for five test instances. The sliding
method uses the constants supplied by the synthetic generator, whereas the
universal baseline estimates its parameters adaptively.

\begin{figure*}[!tp]
  \centering
  \includegraphics[width=\textwidth]{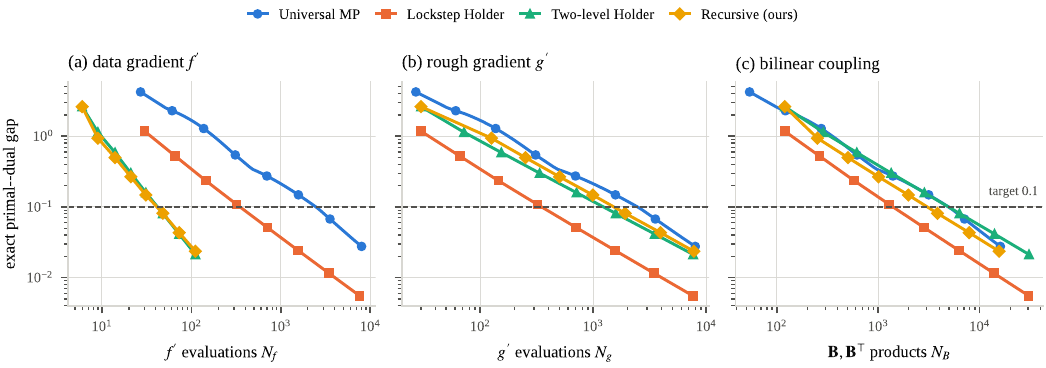}
  \caption{Oracle-wise convergence for
  \((\nu_x,\nu_y)=(0.75,0.25)\): exact gap versus \(f'\), \(g'\), and
  \(\bB,\bB^\top\) evaluations (median, seeds \(0,1,2\)). Universal
  Mirror-Prox is the external baseline. Lockstep and two-level H\"older are
  no- and partial-separation ablations. Settings are frozen after tuning on
  seed \(1729\).}
  \label{fig:e2_separation}
\end{figure*}

\begin{table}[!htb]
  \centering
  \footnotesize
  \setlength{\tabcolsep}{3pt}
  \begin{tabular}{cccccc}
  \toprule
  $\nu_x$ & predicted & measured & per-seed range & decades & gap/bound \\
  \midrule
    $0.00$ & $2.000$ & $1.922$ & $[1.897, 1.931]$ & $1.8$ & $[0.172,0.270]$ \\
    $0.25$ & $1.143$ & $1.103$ & $[1.099, 1.111]$ & $2.7$ & $[0.157,0.196]$ \\
    $0.50$ & $0.800$ & $0.819$ & $[0.817, 0.819]$ & $3.6$ & $[0.135,0.210]$ \\
    $0.75$ & $0.615$ & $0.642$ & $[0.642, 0.645]$ & $3.7$ & $[0.119,0.210]$ \\
    $1.00$ & $0.500$ & $0.509$ & $[0.506, 0.510]$ & $2.9$ & $[0.080,0.112]$ \\
  \bottomrule
\end{tabular}

  \caption{Measured and predicted exponent of \(N_f\) in \(1/\varepsilon\)
  (median and per-seed slope range over three seeds) and the measured
  gap-to-bound range over all budgets. ``Decades'' is the fitted gap range.
  Using the unrounded fits, the largest relative deviation from the predicted
  exponent is \(4.28\%\), attained at \(\nu_x=0.75\).}
  \label{tab:e1_slopes}
\end{table}

\begin{figure*}[!tp]
  \centering
  \includegraphics[width=\textwidth]{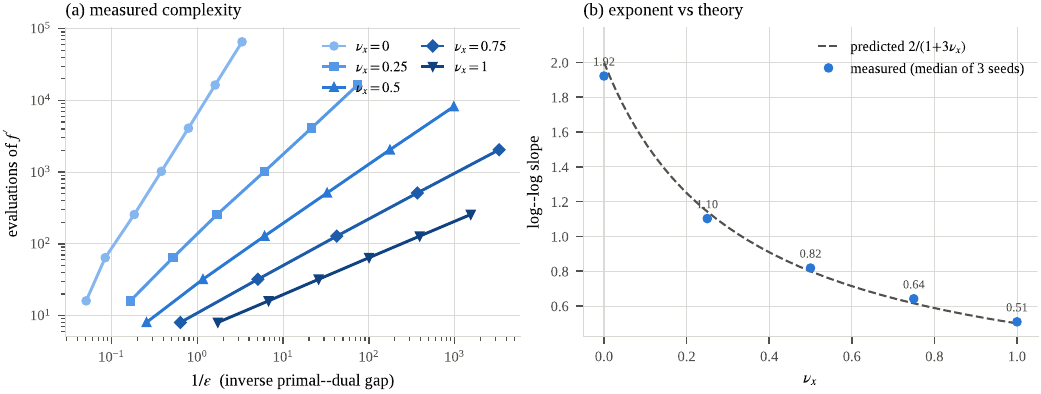}
  \caption{\(f'\)-complexity across the H\"older range in the degenerate
  regime: \(N_f\) versus inverse exact gap (left) and fitted versus predicted
  slope \(2/(1+3\nu_x)\) (right), medians over three seeds.}
  \label{fig:e1}
\end{figure*}

Sorting the recursive levels is essential when \(R_c\ll R_g<R_f\). As
\(\varepsilon\) decreases from \(8\) to \(1\), the sorted order uses
\(4\)--\(40\) coupling products, compared with
\(28{,}804\)--\(1{,}843{,}204\) for the fixed order, while both attain the
requested targets.

\begin{table}[!htb]
  \centering
  \scriptsize
  \setlength{\tabcolsep}{2pt}
  \begin{tabular}{cccrrrcc}
  \toprule
  $(\nu_x,\nu_y)$ & levels & $N_f$ & $N_g$ & $N_g/N_f$ & $N_B$ & gap & seed range \\
  \midrule
    $(1,0.5)$ & \texttt{fgB} & $16$ & $96$ & $6$ & $1920$ & $0.152$ & $[0.108,0.196]$ \\
    $(1,0)$ & \texttt{fBg} & $16$ & $115680$ & $7230$ & $1920$ & $0.417$ & $[0.332,0.419]$ \\
    $(0.75,0.25)$ & \texttt{fBg} & $31$ & $992$ & $32$ & $1984$ & $0.147$ & $[0.094,0.185]$ \\
    $(0.5,0.5)$ & \texttt{fgB} & $93$ & $93$ & $1$ & $2232$ & $0.127$ & $[0.073,0.159]$ \\
    $(0,0)$ & \texttt{Bfg} & $115680$ & $115680$ & $1$ & $1920$ & $0.532$ & $[0.456,0.572]$ \\
  \bottomrule
\end{tabular}

  \caption{Component budgets at \(\varepsilon=0.5\) (median, three seeds).
  Levels are sorted outermost first. The last columns give the attained gap and
  range. Reversing unequal exponent pairs exchanges the roles of \(f\) and
  \(g\), so mirrored rows are omitted. All displayed rows meet the target
  except the fully nonsmooth case, which is \(6\%\) above it with
  \(\kappa_n=1\).}
  \label{tab:e2_mixed}
\end{table}

\paragraph{The cost trade-off.}\label{app:exp_cost}

Table~\ref{tab:e4_cost} reports, at matched accuracy on a smooth instance, the
gradient and matrix-vector counts of the sliding method and of Mirror-Prox, the
resulting ratios, and the crossover
\(\rho^\ast=\mathrm{cost}(f')/\mathrm{cost}(\bB v)\) above which sliding attains
the lower total cost \(\rho(N_f+N_g)+N_B\). The saving in gradient calls grows
as \(\varepsilon\) decreases (from \(4.7\times\) to \(8.5\times\)) and so does
the extra spend on products (from \(18\times\) to \(40\times\)), leaving
\(\rho^\ast\) in the range \(22\)--\(44\). Sliding is therefore the cheaper
method exactly when \(\rho>\rho^\ast\), that is, when one
\((\text{sub})\)gradient evaluation costs more than a few dozen matrix-vector
products. Below that threshold Mirror-Prox is cheaper.

\begin{table}[!htb]
  \centering
  \small
  \begin{tabular}{crrcrrcccc}
  \toprule
  & \multicolumn{3}{c}{sliding} & \multicolumn{3}{c}{Mirror-Prox}
  & \multicolumn{2}{c}{ratio} & \\
  \cmidrule(lr){2-4}\cmidrule(lr){5-7}\cmidrule(lr){8-9}
  $\varepsilon$ & grad & $\bB$ & gap & grad & $\bB$ & gap & grad saved & $\bB$ extra & $\rho^\ast$ \\
  \midrule
    $1e-01$ & $48$ & $4096$ & $5.3e-02$ & $224$ & $224$ & $9.5e-02$ & $4.7\times$ & $18.3\times$ & $22$ \\
    $3e-02$ & $96$ & $16384$ & $1.4e-02$ & $592$ & $592$ & $2.9e-02$ & $6.2\times$ & $27.7\times$ & $32$ \\
    $1e-02$ & $192$ & $65536$ & $3.7e-03$ & $1632$ & $1632$ & $1.0e-02$ & $8.5\times$ & $40.2\times$ & $44$ \\
  \bottomrule
\end{tabular}

  \caption{Oracle counts at matched accuracy and the cost crossover
  \(\rho^\ast\). The entries are medians over three seeds. ``grad'' is
  \(N_f+N_g\). For each
  method the reported counts are those at which its gap first crosses
  \(\varepsilon\), so the accuracy actually attained is at most \(\varepsilon\)
  by construction.}
  \label{tab:e4_cost}
\end{table}

\paragraph{Parameters and seeds.}\label{app:exp_params}

The swept values and selection criteria are:

\begin{itemize}\itemsep2pt
  \item \textbf{Loop counts.} \(T_1\) is swept geometrically, with \(T_2=2\) and
    \(T_3=\lceil 4T_1^{(3\nu_x-1)/2}\rceil\), the rule that keeps the \(g\)- and
    coupling terms of~\eqref{eq:Holder_gap_bound} below the \(f\)-term so that
    the measured exponent is the one belonging to \(f\). In the ordering
    experiment the \(T_i\) are instead those the theory prescribes, from the
    budgets \(R_i\) of Corollary~\ref{cor:VI_sliding_Holder_complexity}.
  \item \textbf{Oracle tolerances.} \(\delta_i\) and the induced \(L_i\) are
    computed from~\eqref{eq:Holder_delta_star}. No parameter is tuned.
  \item \textbf{Baseline parameters.} For the smooth cost experiment, the
    Mirror-Prox step-size grid is \(\{2^k\}_{k=-6}^{6}\) times
    \(1/(L_{xy}+H_x+H_y)\), and the run reaching the smallest gap within the
    shared oracle budget is reported. For Figure~\ref{fig:e2_separation},
    Universal Mirror-Prox uses a target of \(0.0625\) in its backtracking test and
    selects \(L_0\in\{0.1,1,10\}\) on validation seed \(1729\). The selected
    value \(L_0=1\) is then frozen.
  \item \textbf{Instance constants.} \(a_i\equiv0\) and \(b_j\equiv0\) (so the
    quadratic part is absent and the power term alone shapes \(f,g\)),
    \(\alpha_x=\alpha_y=1\), \(\sigma_{\min}=0.5\), and
    \(\sigma_{\max}=2\). Two experiments deviate: the exponent sweep of
    Table~\ref{tab:e1_slopes} fixes \(\nu_y=1\) and takes
    \(\alpha_y=0.05\), \(b_j\in[0,0.05]\) so that the \(g\)-term never binds
    while \(\nu_x\) is swept over \(\{0,\tfrac14,\tfrac12,\tfrac34,1\}\). The
    convergence comparison of Figure~\ref{fig:e2_separation} uses the genuinely
    H\"older pair \((\nu_x,\nu_y)=(\tfrac34,\tfrac14)\). Reversed exponent
    pairs were also run and exchange the \(f\)- and \(g\)-budgets as expected.
    The ordering experiment uses \(\nu_x=0\),
    \(\nu_y=1\), \(\alpha_x=4\), \(\alpha_y=0.1\),
    \(\sigma\in[0.02,0.05]\) to place \(R_c\ll R_g<R_f\).
  \item \textbf{Loop-count constant.} The construction of
    Corollary~\ref{cor:VI_sliding_Holder_complexity} is applied with
    \(\kappa_n=1\), using the raw budgets \(R_i\). The theorem's sufficient
    value is \(2^{2n+2}=256\). Four of five displayed rows in
    Table~\ref{tab:e2_mixed} attain the target. The fully nonsmooth row is
    \(6\%\) above it.
  \item \textbf{Synthetic seeds.} Results for seeds \(0,1,2\) are reported as
    medians, with the per-seed range added in the exponent table. Runs are
    deterministic given the seed.
\end{itemize}

\paragraph{Validity of the bound.}\label{app:exp_bound}

The last column of Table~\ref{tab:e1_slopes} reports the measured gap as a
fraction of the right-hand side of~\eqref{eq:Holder_gap_bound} over the whole
grid of Section~\ref{sec:experiments}. There are no violations in \(90\) runs.
The ratio drifts upward slowly with the budget, which is why a log--log slope
fitted over too narrow a range of accuracies overestimates the exponent. The
fits therefore use the largest four budgets, over which the gap spans
\(1.8\)--\(3.7\) decades.


\subsection{Tomographic reconstruction}\label{app:exp_ct}

\paragraph{Objective, domains, and scaling.}
For \(n=64\), let \(d=n^2\), \(m=180\cdot91=16{,}380\), and consider
\begin{equation}\label{eq:ct_saddle}
 \min_{x\in\cX}\max_{y\in\cY}F_{\tau,c}(x,y)
 =
 \tfrac12\|\bA x-b\|^2+\tfrac{\mu}{2}\|x\|^2
 +\langle y,\bD x\rangle
 -\tfrac{\beta}{2}\|y\|^2-\tfrac{\tau}{2}\langle y,\bL_{\rm nl}y\rangle
 -\tfrac{c}{1+\nu_y}\sum_{i=1}^{2d}|y_i|^{1+\nu_y},
\end{equation}
that is, in the notation of~\eqref{prob:bilinear_SPP_main},
\[
 f(x)=\tfrac12\|\bA x-b\|^2+\tfrac\mu2\|x\|^2,
 \qquad
 g_{\tau,c}(y)=\tfrac\beta2\|y\|^2+
 \tfrac\tau2\langle y,\bL_{\rm nl}y\rangle+
 \tfrac{c}{1+\nu_y}\sum_{i=1}^{2d}|y_i|^{1+\nu_y},
 \qquad \bB=\bD,
\]
\[
 \cX=\{x:\|x\|\leq R\},
 \quad
 \cY=[-\lambda,\lambda]^{2d},
 \quad
 R=1.05\bigl(\|\hat x(0)\|+\|\bD\|\lambda\sqrt{2d}/\mu\bigr),
 \quad
 \hat x(0)=(\bA^\top\bA+\mu\bI)^{-1}\bA^\top b.
\]
The definition of \(R\) bounds
\(\hat x(y)=\hat x(0)-(\bA^\top\bA+\mu\bI)^{-1}\bD^\top y\) uniformly over
\(\cY\). Thus the ball leaves the primal minimizers unchanged and gives the
finite domain constant \(\Omega=R^2+2d\lambda^2\). In the reported runs,
\(R\approx29.0\), \(\|x^\ast\|\approx14.9\), and the projection is inactive.

The projector \(\bA\in\R^{m\times d}\) is sparse pixel-driven parallel-beam,
splatting each pixel, of width \(2/(n-1)\), onto its two neighboring detector
bins. Its stored transpose is the exact adjoint, and \(\bA\) is divided by its
spectral norm. The target intensity lies in \([0,1]\), and
\[
 b=\bA x^\dagger+\xi,\qquad
 \xi_j\stackrel{\rm iid}{\sim}\cN(0,\sigma^2),\qquad
 \sigma=0.01\|\bA x^\dagger\|/\sqrt m.
\]
The map \(\bD\) concatenates the unscaled forward differences
\((D_hx)_{i,j}=x_{i,j+1}-x_{i,j}\) and
\((D_vx)_{i,j}=x_{i+1,j}-x_{i,j}\), with zero in the last column or row and
the corresponding exact transpose implemented.

\paragraph{The three regimes.}
All share
\[
 \mu=10^{-3},\qquad \beta=0.05,\qquad \lambda=5\cdot10^{-5},
\]
and differ only in the dual parameters \((\tau,c,\nu_y)\).

\emph{Standard} (\(\tau=c=0\)). Eliminating \(y\) in~\eqref{eq:ct_saddle} gives
the primal problem
\begin{equation}\label{eq:ct_huber_primal}
 \min_{x\in\R^d}
 \tfrac12\|\bA x-b\|^2+\tfrac\mu2\|x\|^2
 +\lambda\sum_{i=1}^{2d}h_{\beta\lambda}((\bD x)_i),
 \qquad
 h_\delta(t)=
 \begin{cases}
 t^2/(2\delta),&|t|\leq\delta,\\
 |t|-\delta/2,&|t|>\delta,
 \end{cases}
\end{equation}
so the Huber threshold is \(\beta\lambda\), not \(\beta\).

\emph{Nonlocal} (\(\tau=0.05\), \(c=0\)), which tests the third level. Here
\(\bL_{\rm nl}=\operatorname{diag}(L,L)\), where
\(L\) is the graph Laplacian connecting all pixels at offsets
\(|\Delta i|,|\Delta j|\leq5\) with weights
\(\exp(-(\Delta i^2+\Delta j^2)/(2\cdot2.5^2))\), divided by the maximum
weighted degree. This normalization gives \(\|\bL_{\rm nl}\|_2\leq2\), with
\(908{,}552\) stored entries at \(64^2\). Here \(g'\) applies the sparse graph
while \(\bD\) remains a local stencil, and eliminating \(y\) gives the
spatially correlated penalty \(g_{\tau,0}^*(\bD x)\).

\emph{H\"older} (\(\tau=0.05\), \(\nu_y=\tfrac12\),
\(c=7.955\cdot10^{-3}\)) adds a power term whose gradient
\(c\operatorname{sign}(y)|y|^{\nu_y}\) is \((\nu_y,H_y)\)-H\"older but not
Lipschitz. We set the coefficient using the dimensionless boundary-gradient
ratio \(r_c=22.5\):
\[
 c=\operatorname{round}_{10^{-6}}
 \bigl(r_c\beta\lambda^{1-\nu_y}\bigr)=7.955\cdot10^{-3},
\]
which makes the power-gradient magnitude \(22.5\) times the diagonal
quadratic gradient at \(|y_i|=\lambda\). The H\"older constant is the
two-part bound of
Appendix~\ref{app:exp_instances},
\(
  H_y=c\,2^{1-\nu_y}(2d)^{(1-\nu_y)/2}
      +(\beta+2\tau)\operatorname{diam}(\cY)^{1-\nu_y},
\)
and the sliding curvature is \(L_2(\delta_2^\circ)\)
of~\eqref{eq:Holder_delta_star} rather than a Lipschitz constant.

\paragraph{H\"older diagnostics.}
At a reference saddle point, \(3.10\%\) of coordinates satisfy
\(|y_i|\leq0.01\lambda\), and the \(99\)th percentile of the coordinatewise
power-gradient Lipschitz quotient is \(85.8\), compared with the quadratic
constant \(0.15\). The observed global H\"older quotient is
\(0.0616<H_y=0.1213\). Varying the coefficient over \(c/2,c,2c\) leaves the
weighted-cost ordering unchanged and all median gaps below
\(3.1\cdot10^{-4}\).

The values \((\mu,\beta,\lambda)\) were selected before method tuning using a
validation-only pilot over eight candidate pairs. Each received the same
250-step L-BFGS screen, and the selected pair maximized validation PSNR subject
to \(\mu\geq10^{-3}\).

\paragraph{Metric and timing.}
Let \(\operatorname{Gap}_{\tau,c}\) be the gap of
Appendix~\ref{app:exp_synthetic} formed with \(F_{\tau,c}\) on
\(\cX\times\cY\). We report
\[
 \operatorname{RelGap}_{\tau,c}(x,y)
 =\frac{\operatorname{Gap}_{\tau,c}(x,y)}
 {\operatorname{Gap}_{\tau,c}(0,0)},
\]
with the denominator recomputed on every instance. The primal infimum is
computed by diagonally preconditioned CG to relative tolerance
\(2\cdot10^{-11}\), and the dual maximum by proximal gradient with
coordinatewise bisection in the proximal map. Strong-convexity residual bounds
certify both solves. On the recorded H\"older run, the resulting gap-value
error bound is \(9.73\cdot10^{-18}\).

The weights \(c_f,c_g,c_B\) are median isolated latencies of the actual
\(f'\), \(g'\), and one \(\bD\) or \(\bD^\top\) kernel, and
\(C_{\rm weighted}=c_fN_f+c_gN_g+c_BN_B\). Median seedwise ratios are
\((c_f/c_B,c_g/c_B)=(157,0.118)\) in the diagonal regime and
\((160,49.4)\) in the nonlocal regime. End-to-end wall time includes all
algorithm and projection overhead but excludes construction, validation, cost
calibration, and gap evaluation.

\paragraph{Tuning protocol.}
The configurations of Appendix~\ref{app:exp_baselines}, run on validation
phantom/noise pair \((0,0)\), are as follows. Mirror-Prox uses
\[
 \eta/\eta_0\in\{0.75,1,1.25\},\quad
 N\in\{1536,1792,2048,2304\},\quad
 \eta_0=(H_x+H_y+L_{xy})^{-1},
\]
Accelerated MP uses the same iteration set with curvature scales
\(\{0.75,1,1.25\}\). Each sliding grid contains twelve configurations.
The MPS grid consists of \((T_1,T_2)=(96,32),(112,32)\) and every pair in
\(\{128,144,160,176,192\}\times\{24,32\}\). The recursive grid consists of
\((T_1,T_2,T_3)=(96,4,8),(112,4,8)\) and, for each
\(T_1\in\{128,144,160,176,192\}\), both
\((T_2,T_3)=(3,10)\) and \((4,8)\).
Selection minimizes \(C_{\rm weighted}\) subject to validation
\(\operatorname{RelGap}\leq3.75\cdot10^{-4}\), \(75\%\) of the test target,
with no restarts and no test-time stopping. For \(\tau=0\), the selected
fixed configurations are Mirror-Prox
\((\eta,N)=(0.321384,1792)\), Accelerated MP
\((s,N)=(0.75,1536)\), MPS \(T=(128,32)\), and recursive
\(T=(160,3,10)\). For \(\tau=0.05\), they are respectively
\((0.313328,1792)\), \((0.75,1536)\), \((160,24)\), and \((160,3,10)\).
In the H\"older regime, Generalized Universal
Mirror-Prox~\citep{stonyakin2022generalized} is the external baseline. The
fixed-step methods are reported as ablations. Universal Mirror-Prox sweeps the
controlled inexactness
\(\varepsilon\in\{2\cdot10^{-3},5\cdot10^{-4},10^{-4}\}\) of its backtracking
test against \(N\in\{384,448,512,576\}\) at \(L_0=1\). Validation selects
\((\varepsilon,N)=(2\cdot10^{-3},576)\), whose test relative gaps,
\(2.32\cdot10^{-4}\) to \(3.21\cdot10^{-4}\), are comparable to those of the
other methods.

\paragraph{Test evaluation.}
The frozen counts are evaluated on five unseen pairs
\((\text{seed},\text{phantom})=(1,0),(2,1),(3,2),(4,0),(5,1)\), so variability
includes three object geometries as well as independent measurement noise.
Each method is timed three times per pair. The reported wall time is the median
of the five per-pair medians, with their full range. All sixty frozen
method/regime runs attain the common
\(\operatorname{RelGap}\leq5\cdot10^{-4}\) target.

\paragraph{Scope.}
At the reported accuracies, the component budgets are
\(R_f\approx1.7\cdot10^{3}\), \(R_g\approx6.9\cdot10^{3}\), and
\(R_c\approx7.9\cdot10^{6}\), so coupling calls determine the rate. Tomography
therefore demonstrates oracle separation with an expensive, non-Lipschitz
\(g'\), but not the \(\nu_y\)-dependent exponent. The synthetic experiments
test that exponent.

\clearpage
\begin{table*}[t]
  \centering
  \scriptsize
  \setlength{\tabcolsep}{3pt}
  \resizebox{\textwidth}{!}{\begin{tabular}{lrrrrrrr}
  \toprule
  method & $N_f$ & $N_g$ & $N_B$ & weighted ms & wall ms [range] & rel. gap & PSNR [range] \\
  \midrule
    \multicolumn{8}{l}{\emph{Huber--TV}} \\
    Mirror-Prox & $3584$ & $3584$ & $7168$ & $16236$ & $17540$ $[17484,17660]$ & $2.55\times10^{-4}$ & $25.67$ $[23.84,25.82]$ \\
    Accelerated MP & $1536$ & $1536$ & $6144$ & $7052$ & $7899$ $[7882,7937]$ & $3.14\times10^{-4}$ & $25.75$ $[23.99,26.05]$ \\
    MP Sliding & $128$ & $4096$ & $16384$ & $1086$ & $2433$ $[2420,2461]$ & $2.80\times10^{-4}$ & $25.77$ $[24.00,26.08]$ \\
    Recursive (ours) & $160$ & $480$ & $19200$ & $1302$ & $3177$ $[3120,3194]$ & $2.98\times10^{-4}$ & $25.68$ $[23.91,25.98]$ \\
    \multicolumn{8}{l}{\emph{Nonlocal Huber--TV}} \\
    Mirror-Prox & $3584$ & $3584$ & $7168$ & $21145$ & $23070$ $[22996,23365]$ & $2.66\times10^{-4}$ & $25.59$ $[23.77,25.74]$ \\
    Accelerated MP & $1536$ & $1536$ & $6144$ & $9157$ & $10348$ $[10195,10506]$ & $3.14\times10^{-4}$ & $25.75$ $[23.99,26.05]$ \\
    MP Sliding & $160$ & $3840$ & $15360$ & $6498$ & $8387$ $[8227,8445]$ & $2.94\times10^{-4}$ & $25.71$ $[23.94,26.01]$ \\
    Recursive (ours) & $160$ & $480$ & $19200$ & $1968$ & $3957$ $[3935,4000]$ & $2.98\times10^{-4}$ & $25.67$ $[23.91,25.98]$ \\
    \multicolumn{8}{l}{\emph{Nonlocal H\"older--TV}} \\
    Universal MP & $1728$ & $1728$ & $3456$ & $10508$ & $10780$ $[10698,11331]$ & $2.59\times10^{-4}$ & $25.64$ $[23.82,25.80]$ \\
    Accelerated MP & $1536$ & $1536$ & $6144$ & $9426$ & $10429$ $[9824,10458]$ & $3.15\times10^{-4}$ & $25.74$ $[23.98,26.05]$ \\
    MP Sliding & $160$ & $3840$ & $15360$ & $6889$ & $8562$ $[8322,8618]$ & $2.94\times10^{-4}$ & $25.70$ $[23.94,26.01]$ \\
    Recursive (ours) & $160$ & $480$ & $19200$ & $2029$ & $3981$ $[3922,4022]$ & $3.03\times10^{-4}$ & $25.63$ $[23.87,25.94]$ \\
  \bottomrule
\end{tabular}
}
  \caption{Application comparison at the common relative primal--dual-gap
  target \(5\cdot10^{-4}\). Weighted cost uses measured per-component latency.
  Wall time is end-to-end optimization time. Entries are medians over five
  unseen phantom/noise instances. Brackets give the observed instance range
  for wall time and PSNR.}
  \label{tab:e5_ct}
\end{table*}

\begin{figure*}[t]
  \centering
  \includegraphics[width=\textwidth]{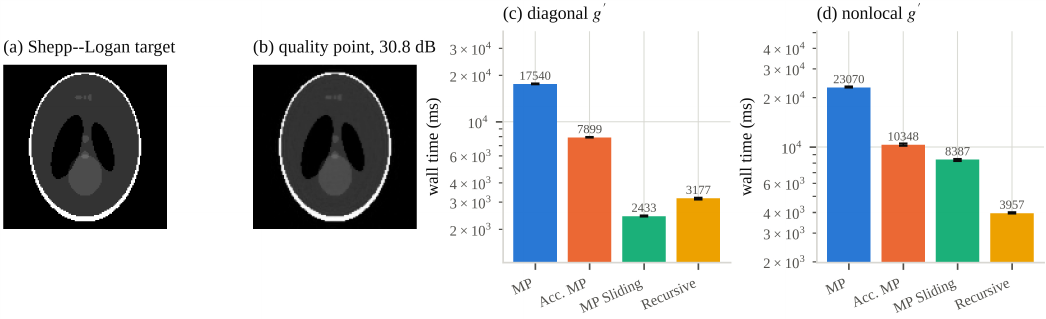}
  \caption{Tomographic reconstruction and timing. (a) \(128^2\) modified
  Shepp--Logan target. (b) Recursive result at \(T=(768,4,8)\), reaching
  \(30.81\) dB at relative gap
  \(3.26\cdot10^{-5}\). (c--d) Median end-to-end time at the common
  \(5\cdot10^{-4}\) target, with bars spanning the five test-instance medians.
  The standard diagonal and nonlocal regimes use the counts and times in
  Table~\ref{tab:e5_ct}.}
  \label{fig:e5_ct}
\end{figure*}

\FloatBarrier

\begin{figure}[H]
  \centering
  \includegraphics[width=\textwidth]{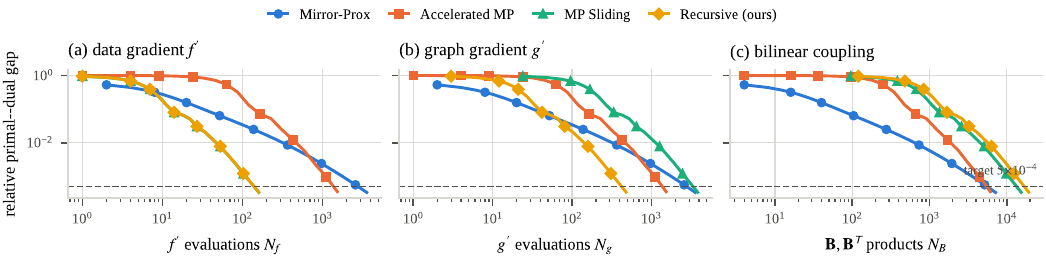}
  \caption{Oracle-wise convergence on the first unseen nonlocal tomography
  instance. The three panels plot the relative primal--dual gap against,
  respectively, evaluations of the data gradient \(f'\), evaluations of the
  graph gradient \(g'\), and products with \(\bB,\bB^\top\). The dashed line
  is the common target \(5\cdot10^{-4}\).}
  \label{fig:e5_oracle_convergence}
\end{figure}

\section{Broader Complexity-Separation Results}\label{app:separation}

This appendix expands the positioning summarized in
Section~\ref{sec:related_work} by comparing the oracle models and assumptions
of the closest complexity-separation results.

\paragraph{Joint versus separated oracle counts.}
\citet{kovalev2022accelerated} solve the smooth bilinear
problem~\eqref{prob:bilinear_SPP_main} with the Accelerated Primal-Dual
Gradient method, which is linearly convergent and optimal in the
strongly-convex--strongly-concave regime. Its oracle model allows
\(\mathcal O(1)\) evaluations of \(\nabla f\), of \(\nabla g\), and of products
with \(\bB,\bB^\top\) at every iteration, and its complexity is a single count
\(\mathcal O(\min\{T_a,T_b,T_c,T_d\}\log(1/\varepsilon))\) whose leading term is
\[
  T_a
  =
  \max\Bigl\{
    \sqrt{\kappa_x},\;
    \sqrt{\kappa_y},\;
    \sqrt{\kappa_{xy}}
  \Bigr\},
  \qquad
  \sqrt{\kappa_{xy}}=\frac{L_{xy}}{\sqrt{\delta_x\delta_y}},
\]
charged simultaneously to all three oracles. Separation replaces this maximum by
the three terms individually, one per oracle.
\citet{borodich2025linear} obtain such separate counts with optimal dependence
on \(\sqrt{\kappa_x}\), \(\sqrt{\kappa_y}\), and
\(\sqrt{\kappa_{xy}}\) for the smooth bilinear problem.

\paragraph{Separation for general coupling.}
\citet{tominin2021accelerated} consider
\(\min_x\max_y\{f(x)+G(x,y)-h(y)\}\) with \(h\) of finite-sum form,
strongly convex in \(x\) and strongly concave in \(y\), and all components
smooth. They estimate separately the numbers of evaluations of \(\nabla f\),
\(\nabla_xG\), \(\nabla_yG\) and \(\nabla h_i\), obtaining budgets such as
\(\widetilde{\mathcal O}(\sqrt{\kappa_x^{(f)}\kappa_y^{(G)}})\) for \(\nabla f\)
and \(\widetilde{\mathcal O}(\sqrt{\kappa_x^{(G)}\kappa_y^{(G)}})\) for the
coupling. Their coupling budgets depend on products of primal and dual
condition numbers of \(G\), whereas the bilinear specialization considered
here uses the spectral quantity \(\sqrt{\kappa_{xy}}\). Their separation is
obtained through a two-level accelerated construction for \(\min_xF(x)\), with
\(F(x)=f(x)+\max_y\{G(x,y)-h(y)\}\), driven by inexact proximal steps and
probabilistic inexact oracles. The resulting bounds carry polylogarithmic
factors in \(\varepsilon^{-1}\) and the confidence level, and assume the
strongly-convex--strongly-concave regime.

\paragraph{Universal H\"older primal--dual methods.}
Affine-constrained models have bilinear Lagrangians.
\citet{yurtsever2015universal} adapt to unknown H\"older regularity in the dual
formulation, while \citet{luo2024universal} handle Lipschitz and H\"older
gradients universally and prove an optimal mixed-type rate. These methods use
joint iteration or line-search counts, while the present analysis tracks
function-gradient and \(A,A^\top\) evaluations separately.

\paragraph{Parameter-free composite separation.}
For composite minimization, the universal, parameter-free PFUGS method
of~\citet{wu2026parameterfree} separately counts a H\"older \(f'\)-oracle and a
smooth \(g'\)-oracle. Its model has two function-oracle families and no
coupling oracle. The model studied here tracks \(f'\), \(g'\), and
\(\bB,\bB^\top\), with separate H\"older exponents for the two function
components.

\end{document}